\documentclass{article}

\usepackage[preprint]{neurips_2026}

\usepackage[utf8]{inputenc} 
\usepackage[T1]{fontenc}    
\usepackage{hyperref}       
\usepackage{url}            
\usepackage{booktabs}       
\usepackage{amsfonts}       
\usepackage{nicefrac}       
\usepackage{microtype}      
\usepackage{xcolor}         
\usepackage{amsmath,amssymb,amsthm,mathtools}
\usepackage{thm-restate}
\usepackage{booktabs,tabularx,array}
\usepackage{enumitem}

  \providecommand{\R}{\mathbb{R}} %

  \DeclareMathOperator{\E}{{\mathbb E}}

  \providecommand{\cB}{\mathcal{B}}

  \providecommand{\cF}{\mathcal{F}}
  \providecommand{\cG}{\mathcal{G}}
  \providecommand{\cH}{\mathcal{H}}
  \providecommand{\cI}{\mathcal{I}}

  \providecommand{\cN}{\mathcal{N}}
  
  \providecommand{\cP}{\mathcal{P}}
  
  \providecommand{\cR}{\mathcal{R}}
  
  \providecommand{\cT}{\mathcal{T}}

  \providecommand{\cX}{\mathcal{X}}
  \providecommand{\cY}{\mathcal{Y}}
  \providecommand{\cW}{\mathcal{W}}

  \usepackage{bm}

\providecommand{\mycomment}[3]{\todo[caption={},size=footnotesize,color=#1!20]{\textbf{#2: }#3}}%
\providecommand{\inlinecomment}[3]{%
  {\color{#1}#2: #3}}%
\newcommand\commenter[2]%
{%
  \expandafter\newcommand\csname i#1\endcsname[1]{\inlinecomment{#2}{#1}{##1}}
  \expandafter\newcommand\csname #1\endcsname[1]{\mycomment{#2}{#1}{##1}}
}

\newtheorem{proposition}{Proposition}
\newtheorem{lemma}{Lemma}

\newtheorem{definition}{Definition}
\newtheorem{remark}[lemma]{Remark}

\newtheorem{theorem}[lemma]{Theorem}

\newcommand{\e}{\mathbb{E}}

\newcommand{\vp}{\varphi}
\newcommand{\ve}{\varepsilon}
\newcommand{\II}{\mathbb{I}}

\title{Statistical Unlearning of Distributions: 
\\
A Hypothesis Testing Approach}

\author{%
  Aaradhya Pandey \\
  Princeton University\\
\texttt{aaradhyapandey@princeton.edu} 
  \And
  Sanjeev Kulkarni\\
  Princeton University \\
\texttt{kulkarni@princeton.edu} \\
}

\begin{document}

\vspace{-4 mm}
\maketitle
\vspace{-4 mm}

\begin{abstract} 
\vspace{-2 mm}
Machine learning systems increasingly face requirements to forget not only individual data points, but entire domains of information, such as toxic language, copyrighted corpora, or demographic biases. This raises a fundamental dilemma of statistical-computational tradeoffs: removing all samples from an unwanted domain may be computationally prohibitive, while randomly removing a subset may not provide distribution-level statistical guarantees.  We propose a statistical framework for distributional unlearning, in which domains are modeled as probability distributions, and the goal is to remove a carefully chosen subset of samples that reduces the effect of an unwanted distribution while preserving performance on a desired one.   We formalize this  using a hypothesis test of the edited data with the desired and unwanted domains, leading to an interpretable and robust criterion for selecting samples to remove. Within this statistical framework, we characterize the fundamental region of the allowable edited data distributions and the removal-preservation Pareto frontier for a broad class of distribution families. This includes parametric families such as shifted Gaussians of arbitrary dimension, a one-dimensional location family with log-concave noise, and the one-dimensional Poisson family. It also includes nonparametric families such as the Gaussian white noise model, a canonical model for nonparametric regression. We prove composition rules that describe how distributional unlearning behaves across multimodal unwanted domains, and introduce a central-limit behavior for the removal-preservation baselines when composing a large number of such families. Finally, we provide finite sample guarantees by providing Pareto frontiers for some selection algorithms, and observe an information-computation gap.
\vspace{-2 mm}
\end{abstract}
\vspace{-3 mm}
\section{Introduction and motivation} \label{sec:Intro}
\vspace{-2mm}

As machine learning models remain deployed over time, portions of their training data may later become legally or ethically unacceptable, requiring deletion. Increasingly, such deletion requests concern not just individual records but whole subpopulations or domains, corresponding to unwanted concepts, data-defined groups, harmful biases, or toxic language \citet{kurmanji2024towards}. This shift creates a basic tension: removing every example from the target domain can be computationally expensive, since the cost of many efficient unlearning methods grows with the size of the forget set \citet{guo2020certified}. A concrete example is the attempt to unlearn creative works such as the Harry Potter series \citet{eldan2023s}, where the volume of identified text and the scale of the trained model make the unlearning step costly. At the same time, discarding a small random subset is often statistically ineffective, because the domain may continue to influence the model through the remaining data. For example, large language models can still reproduce memorized training sequences even after the specific source document is removed ~\citep{liu2025language}, owing to overlapping contexts elsewhere in the corpus. This exposes a methodological gap between costly full deletion and unreliable naive subsampling, and points to the need for principled domain unlearning.

\vspace{-1 mm}
To address this gap, we begin from the observation that a domain’s statistical influence is often carried by a small, high-impact subset of its samples. This suggests an intermediate strategy: remove only the influential subset, achieving both computational efficiency and statistical effect. We formalize this idea by modeling a domain or subpopulation as a probability distribution, a standard abstraction in statistical learning~\citep{shalev2014understanding} and natural language processing~\citep{blei2003latent,srivastava2017autoencoding}. This leads to the central question: \emph{what is the smallest set of data points whose removal makes the edited data far from an unwanted domain while keeping it close to a retained one?} This is known as \emph{distributional unlearning} ~\citep{allouah2025distributionalmachineunlearningselective}. This perspective complements what is called \emph{sample-level} unlearning ~\citep{pandey2026gaussian}: methods for approximating retraining after a specified set of records is removed. Prior approaches based on influence functions~\citep{guo2020certified} or data sharding~\hbox{\citep{bourtoule2021machine}} address the \emph{computational} side of unlearning, but not the \emph{statistical} problem of choosing the most impactful samples to forget. Likewise, class-level unlearning~\citep{tarun2023fast,kodge2024deep} and concept erasure~\citep{ravfogel2020null,belrose2023leace} operate on domain-like objects, but either modify internal representations, making erasure potentially reversible, or do not provide a principled rule for selecting a small removal set. Thus, the basic question remains: which samples to remove?

\vspace{-4mm}
\subsection{Summary of our contributions} \label{ssec:contribution}
\vspace{-2mm}

\textbf{Hypothesis testing framework:}
We formalize distributional unlearning through a hypothesis-testing framework
(Definition~\ref{def:dist-unlearning}) to balance removal and preservation guarantees. This framework generalizes the KL-divergence-based
formulation of~\citet{allouah2025distributionalmachineunlearningselective} (Proposition~\ref{thm:CoD}) and admits an equivalent characterization through a
data-processing principle (Theorem~\ref{thm:Blackwell}). We also prove composition rules for distributional unlearning over product and multimodal domains (Lemma~\ref{lem:CDU}).

\textbf{Robust downstream guarantees:} We prove that our unlearning framework  \eqref{eq:dist-unlearning} provides distributional guarantees for downstream models (Lemma \ref{lem:DG}),  robust guarantees under the total variation loss  (Proposition \ref{prop:predictive}), and more generally for any divergence-based loss satisfying data-processing.

\textbf{Feasible region and Pareto frontier:}  We provide an interpretable and  clear geometric description (Figure \ref{fig:alpha_e_plot_1}) of the fundamental region of allowable edited data distributions and characterize the closed-form Pareto frontier (Figure \ref{fig:alpha-eps-feasibleM}) of achievable removal-preservation levels for several families of distributions, including parametric families such as shifted Gaussians (Proposition \ref{thm:FFR}, Lemma \ref{lem:PFGF}), a one-dimensional location family with symmetric log-concave noise (Proposition \ref{prop:FFRLC}), and the one-dimensional Poisson family (Proposition \ref{thm:FRP}). It also includes nonparametric families with infinite-dimensional parameter spaces such as the Gaussian white noise model (Proposition \ref{thm:GWN}).

\textbf{Algorithms \& information-computation gaps:} We analyze the random removal algorithm of ~\citet{allouah2025distributionalmachineunlearningselective} in Propositions~\ref{th:random} and~\ref{th:randomLC}, together with Remark~\ref{rem:RRLaplace}. We also analyze the selective removal mechanism of ~\citet{allouah2025distributionalmachineunlearningselective} in Propositions~\ref{th:selective} and~\ref{th:randomLCSR}. We show that both the algorithmic procedures attain smaller feasible regions than the information-theoretic feasible region, thereby giving rise to information-computation gaps (Remarks \ref{rem:ICGRRM}, \ref{rem:ICGSR}). Moreover, in the one-dimensional Gaussian case, we show that under large separation (Proposition \ref{th:CSRRODG}, Remark \ref{rem:CSROGH}), selective removal yields a more favorable information-computation gap than random removal.

\vspace{-4mm}
\subsection{Comparison with previous work}
\label{ssec:Comparison}
\vspace{-2mm}

\textbf{Sample-level unlearning.}
Sample-level unlearning has advanced rapidly, providing efficient model updates and formal deletion guarantees~\citep{neel2021descent,zhang2024towards,chien2024langevin,allouah2024utility,koloskova2025certified,pandey2026gaussian}. 
Influence-function methods approximate the effect of deleting individual points without retraining~\citep{izzo2021approximate}; certified unlearning bounds the distance to a scratch-retrained model~\citep{guo2020certified}; and data sharding enables efficient removal of small batches~\citep{bourtoule2021machine}. 
These methods, however, typically assume the forget set is already specified. Our work is complementary: we ask which samples should be flagged and removed, and in what quantity, to erase a domain’s statistical footprint.

\textbf{Comparison with ~\citet{allouah2025distributionalmachineunlearningselective}.}  
We propose an interpretable and robust definition of distributional unlearning based on hypothesis testing  (Definition \ref{def:dist-unlearning}) that extends and encompasses the definition of \citep{allouah2025distributionalmachineunlearningselective} (Proposition \ref{thm:CoD}), as well as the accompanying downstream guarantees (Lemma \ref{lem:DG}). We also prove a robust version of the downstream distributional guarantee (Proposition \ref{prop:predictive}). For shifted Gaussian families, we find the information-theoretic feasible regions (Proposition \ref{thm:FFR}) and the corresponding statistical Pareto frontier (Lemma \ref{lem:PFGF}) as well as the algorithmic Pareto frontier for both the random removal (Proposition \ref{th:random}) and selective removal (Proposition \ref{th:selective}) algorithms. On the statistical front, we extend the one-dimensional Gaussian results of \citep{allouah2025distributionalmachineunlearningselective} to higher dimensions (Proposition \ref{thm:FFR}, Lemma \ref{lem:PFGF}), including the non-parametric Gaussian white noise model (Proposition \ref{thm:GWN}), as well as for non-Gaussian log-concave location families (Proposition \ref{prop:FFRLC}) such as Laplace. We also provide a characterization of the feasible region for the Poisson family (Proposition \ref{thm:FRP}), which is a non-location family. On the algorithmic front, we provide theorems on the behavior of both the removal algorithms on all the parametric families mentioned above (Propositions \ref{th:random}, \ref{th:randomLC}, \ref{th:selective}, \ref{th:randomLCSR}). We refer to \citep{allouah2025distributionalmachineunlearningselective} for more details on the connections of unlearning to  concept \textit{ erasure} methods  
~\citep{ravfogel2020null, kaushik2020learning, elazar2018adversarial}, and \textit{Domain adaptation} ~\citep{ben2010theory}.

\vspace{-4mm}
\section{Unlearning a distribution: the hypothesis testing framework}
\label{sec:problem}
\vspace{-2mm}
The primary focus in machine unlearning has been at the level of samples,  by efficiently removing the influence of selected data points. On the contrary, \cite{allouah2025distributionalmachineunlearningselective} introduced distributional unlearning under a divergence-based constraint, where the goal is to erase the statistical influence of an entire subpopulation. In this section, we formalize this using an interpretable hypothesis testing approach, introduce its feasible regions, and discuss the consequences for downstream tasks.

\vspace{-1 mm}
\textbf{Problem statement.}
We first model the unwanted population as a probability distribution $p_1$.  Now, the objective is to construct a new data distribution $p$ that is statistically distinguishable (far) from $p_1$ yet remains indistinguishable (close) to a retained distribution $q_1$. In practice, these true distributions are unknown, and we have finite samples $S_1 = \{x_i^{(1)}\}_{i=1}^{n_1}$ drawn IID from the unwanted distribution $p_1$, and $S_2 = \{x_j^{(2)}\}_{j=1}^{n_2}$ from the retained distribution $q_1$.
We further assume that these sets are obtained via some upstream process, and throughout we assume that the random vector $S_1$ is generated independently of $S_2$. Now, we would like to solve the following problem from a statistical or information theoretic perspective: given these identified samples, which subset of the entire data should be removed to most efficiently achieve our objective defined at the distribution level?
\vspace{-4mm}
\subsection{\texorpdfstring{$(f_d,f_c)$}{(f_d,f_c)} distributional unlearning}
\vspace{-2mm}
We formalize this using a hypothesis testing approach instead of the Kullback-Leibler (KL) divergence  \cite{allouah2025distributionalmachineunlearningselective} because of its interpretability and tightness under post-processing.

\begin{definition}[Trade-off function \cite{dong2022gaussian}]    \label{def:TOF} 
Given two probability distributions $P, Q$ on a measurable space $(\cW, \mathcal{F}_{\cW})$, we define the  \emph{trade-off function} as  the map $= T(P,Q):[0,1]\to[0,1]$  
\vspace{-1mm}
\begin{equation}\label{eq:TOF}
 T(P,Q)(\alpha):=
  \inf_{\phi}\Bigl\{
        \beta_{\vp} := \e_{Q}[1-\vp] 
        \,\Bigm\vert\,
        \alpha_{\vp}:= \e_{P}[\vp]\le \alpha,\;
        \vp:\cW\!\to[0,1]\text{ measurable}
      \Bigr\}.
\end{equation}
\textbf{Intuition:} For type I error $\alpha$, the trade-off function (TOF) returns the smallest value of type II error $\beta_{\vp}$ over all possible test functions $\vp\colon \cW\to[0,1]$. So, $T(P,Q)(\alpha) \geq T(P',Q')(\alpha)$  $\forall \alpha \in [0,1]$ captures the intuition that the pair $(P,Q)$ is more difficult to distinguish than  $(P',Q')$ across every scale $\alpha$. \cite{dong2022gaussian} proved that  trade off functions $\cT:=\{T(P,Q): (P,Q)\}= \{f: [0,1] \to [0,1] \text{ convex, continuous, decreasing and } f(x)\leq 1-x \forall x\in [0,1]\}$ satisfy Blackwell ordering.
\end{definition}
\begin{definition}[$(f_d,f_c)$ distributional unlearning]
\label{def:dist-unlearning}
For TOFs $f_d,f_c \in \cT$, a distribution $p \in \mathcal{P}$ (a class of distributions on $(\cW, \cF_{\cW})$) satisfy $(f_d,f_c)$ unlearning with respect to $p_1, q_1 \in \cP$ if:
\begin{align} \label{eq:dist-unlearning}
T(p, p_1) \leq f_d =T(P_d, Q_d) \quad\text{(removal)},
\quad 
T(p, q_1) \geq f_c = T(P_c, Q_c) \quad\text{(preservation)}.
\end{align}
\end{definition}
\vspace{-1mm}
\textbf{Generalization to multiple unwanted and retained populations:}
In the appendix (Definition \ref{def:dist-unlearning_multi}) we generalize the above definition  to $k$ many unwanted populations $\{p_1, \cdots, p_{k}\}$ with baseline TOFs $\{f_{d1}, \cdots, f_{dk}\}$ and $l$ many retained ones $\{q_1, \cdots, q_{l}\}$ with baseline TOFs $\{f_{c1}, \cdots, f_{cl}\}$. 

\textbf{Intuition and flexibility in the choice of the baseline pairs $(f_d,f_c)$ depending on $\cP$:}  Let $f_d= T(N(0,1), N(\alpha, 1))$ for $\alpha >0$. Then from the definition of $T(P,Q)$, the inequality  $T(p, p_1) \leq f_d$ has the interpretation that the pair (edited data, unwanted data) $=$ $(p, p_1)$ is statistically easier to differentiate by an adversary than the pair of normally distributed samples $(N(0,1), N(\alpha, 1))$. Similarly, for $f_c= T(N(0,1), N(\ve, 1))$ for $\ve>0$, the inequality  $T(p, q_1)\geq f_c$ captures the intuition that the pair (edited data, retained data) $= (p,q_1)$ is statistically more difficult to separate by an adversary than the pair $(N(0,1), N(\ve, 1))$. In comparison to $(\alpha, \ve)$ unlearning of \cite{allouah2025distributionalmachineunlearningselective} our definition gives us the flexibility to adaptively choose the baseline TOFs $(f_d,f_c)$ depending upon our prior knowledge about the type of populations $\cP \ni  p_1, q_1 $.

\vspace{-1mm}

\textbf{Consistency of the framework: }  If $p$ satisfies $(f_d,f_c)$ unlearning with respect to $p_1,q_1\in \cP$, then $p$ also satisfies $(f_{d'},f_{c'})$ unlearning for TOFs $f_{d'} \geq f_d$ and $f_c \geq f_{c'}$. For $f_d=T(N(0,1), N(\alpha, 1))$ and  $f_c=T(N(0,1), N(\ve, 1))$, if $p$ satisfies $(f_d \leftrightarrow \alpha ,f_c \leftrightarrow \ve)$ unlearning, then as a consequence of Lemma \ref{lem: GTOF} it also satisfies  $(f_{d'} \leftrightarrow \alpha' ,f_{c'} \leftrightarrow \ve')$ unlearning for any $0\leq \alpha' \leq \alpha$ and $\ve' \geq \ve$.

In the remainder of this section, we analyze the properties of Definition \ref{def:dist-unlearning}  at the population level, capturing the fundamental feasible region and Pareto Frontier of extremal $(\alpha, \ve)$ pairs of the problem. 

\begin{restatable}[Comparison with $(\alpha, \varepsilon)$ unlearning ~\citet{allouah2025distributionalmachineunlearningselective}]{proposition}{CoD}
\label{thm:CoD}
For TOFs $f_d =T(P_d, Q_d),f_c= T(P_c, Q_c) $, if  a distribution $p \in \mathcal{P}$ (a class of distributions on $(\cW, \cF_{\cW})$) satisfy $(f_d,f_c)$ unlearning with respect to $p_1, q_1 \in \cP$ in the sense of \eqref{eq:dist-unlearning}, then it satisfies $(\alpha, \varepsilon)$ distributional unlearning of ~\citet{allouah2025distributionalmachineunlearningselective} with $\alpha = KL(Q_d, P_d)$ and $\varepsilon= KL(Q_c,P_c)$. Equivalently,
\vspace{-1 mm}
\begin{equation} \label{eq:aeu}
KL(p_1, p) \geq \alpha =KL(Q_d,P_d) \text{ (removal) },
\quad 
KL(q_1, p) \leq \varepsilon= KL(Q_c, P_c) \text{ (preservation)}.
\end{equation}
\end{restatable}
\vspace{-4 mm}
\subsection{Removal-preservation trade-off and the feasible region}
\vspace{-2 mm}
The baseline trade-off functions $(f_d,f_c)$ captures how different the edited data $p$ can be from the unwanted distribution $p_{1}$ while remaining indistinguishable to the desired one $q_1$. 
For a given population $\cP$, we would have to choose an appropriate pair $(f_d,f_c)$ to understand whether the output $p$ achieves  \eqref{eq:dist-unlearning} jointly for $p_1, q_1 \in \cP$. Now, we formalize the feasible region.
\vspace{-1 mm}
\begin{definition}[Feasible region] \label{def:FR}
    Given probability distributions $p_1,q_1 \in \cP$ and trade-off functions $f_d,f_c \in \cT$ we define the feasible region  $\cR=\{p \in \cP: T(p, p_1) \leq f_d, \text{ and } T(p, q_1) \geq f_c\}$
\end{definition}
\begin{remark}
This fundamental region of feasibility and its boundary \textit{generalizes} the \emph{Pareto frontier} concept of \cite{allouah2025distributionalmachineunlearningselective} and provides tighter boundary regions in our examples of interest. 
\end{remark}
\vspace{-4 mm}
\subsection{Compositional rules of unlearning under multimodal inputs}
\vspace{-2 mm}
The hypothesis testing approach to unlearning \eqref{eq:dist-unlearning} extends naturally to multimodal settings. It allows us to produce a product distribution $p\otimes p'$, when there are multiple modes $p\in \cP, p'\in \cP'$ that need to be combined. For instance, the unwanted distribution $p_1 \in \cP$ may correspond to toxic language in text, while $p_1' \in \cP'$  captures inappropriate content in images, both of which are to be jointly unlearned under a framework. We formalize this as a compositional rule in the following Lemma.
\vspace{-1 mm}
\begin{restatable}[Composition laws of distributional unlearning]{lemma}{CDU}
\label{lem:CDU}
If a distribution $p \in \cP$ satisfy $(f_d=T(P_d,Q_d),f_c=T(P_c,Q_c))$ unlearning with respect to  probability distributions $p_1,q_1\in \cP$ on some $(\cW, \cF_{\cW})$ \eqref{eq:dist-unlearning}, and a distribution $p' \in \cP'$ satisfy $(f_d'=T(P_d',Q_d'),f_c'=T(P_c',Q_c'))$ unlearning with respect to  probability distributions $p_1',q_1'\in \cP'$ on some $(\cW', \cF_{\cW}')$, then the product distribution $p\otimes p'$ satisfy $(f_d\otimes f_d', f_c\otimes f_c')$  unlearning with respect to $p_1\otimes p_1', q_1\otimes q_1' \in \cP \otimes \cP'$, where  for trade-off functions $f=T(P,Q)$ and $f'=T(P',Q')$ we define $f\otimes f':= T(P\otimes P', Q\otimes Q')$, and for families of distributions $\cP$ and $\cP'$ we define $\cP\otimes \cP':=\{p\otimes p': p\in \cP, p'\in \cP'\}$.
\end{restatable}
\vspace{-2 mm}
 \begin{remark}(Compositing large number of multimodal inputs and natural baselines)
      For desired and unwanted pairs of domains $q_n, p_n$ respectively within a mode of input $\cP_n$ such as text, image, video, audio, if $p^{(n)}\in \cP_n$ satisfy $(f_{dn},f_{cn})$ unlearning \eqref{eq:dist-unlearning} with respect to $p_n,q_n \in \cP_n$, then $p^{[n]}= p^{(1)} \otimes  \cdots \otimes p^{(n)}$ satisfy $(f_{dn}^{\otimes n},f_{cn}^{\otimes n})$ unlearning \eqref{eq:dist-unlearning} with respect to $\otimes_{i=1}^n p_i, \otimes_{i=1}^n q_i \in \otimes_{i=1}^n \cP_i$. Now, as $n\uparrow \infty$ or equivalently, number of forms of inputs becomes large, as a consequence of ~\citep[Theorem 6]{pandey2025infinitely}, we have for a sequence of trade-off functions $\{f_n\}_n$ with $f_n(0)=1$, if $f_n^{\otimes n} \to f_{\infty} \in \cT$ and the convergence is pointwise on $[0,1]$, then $f_{\infty} \in \cI$, where \vspace{-1 mm} 
\begin{equation}
    \cI := \{T(P,Q): P \text{ infinitely divisible  on } \R \text{ with } P(e^x) =1 \text{ and } dQ(x)= e^x dP(x)\}.\footnote{We refer to ~\citep{kallenberg2021foundations} for definitions of infinitely divisible distributions on the real line $(\R, \cB_{\R})$.}
\end{equation}
Therefore, we obtain a natural class of removal-preservation baselines $(f_d,f_c)$ whenever we have multimodal inputs with a large number of modes of input. A consequence of a central limit type theorem dictates the natural choices of baselines $f_d,f_c \in \cI$ ~\citet{pandey2025infinitely}. Mathematically, if one considers $\cP_n= \{\mu+Z:\mu\in \R\}$ with $Z\sim N(0,1)$, then $\otimes_{i=1}^n\cP_i=\{\mu +Z :\mu\in \R^n\}$ with $Z\sim N(0,\II_n)$. So, when $n$ is large, considering baselines  $(f_d,f_c)$ with $f_d,f_c\in \cI$ becomes natural.  Moreover, the Gaussian baseline with $f=T(N(0,1), N(\alpha,1))$ for any $\alpha >0$ and the Poisson baseline $g=T(P(\lambda_1), P(\lambda_2))$ with any $\lambda_1, \lambda_2 >0$ belong to $\cI$ ~\citep[Lemma 4]{pandey2025infinitely}.
\end{remark}
\vspace{-1 mm}
\begin{remark}(Behavior under additional source of information)
    The hypothesis-testing approach to unlearning \eqref{eq:dist-unlearning} behaves naturally when we add more \textit{independent columns} to our data. More precisely, if $\cP'=\{p'\}$ a single ton, and $p\in \cP$ satisfy $(f_d,f_c)$ unlearning with respect to $p_1,q_1\in \cP$, then $p\otimes p'$ satisfy $(f_d,f_c)$ unlearning with respect to $p_1\otimes p', q_1\otimes p' \in \cP\otimes \{p'\}$.  This follows from Proposition \ref{prop:tensor-basic-properties} as $T(p\otimes p', p_1\otimes p')= T(p,p_1) \leq f_d$ and $T(p\otimes p', q_1\otimes p')= T(p,q_1) \geq f_c$. 
\end{remark}
\vspace{-3mm}
\subsection{Downstream guarantees}
\vspace{-2 mm}
 Blackwell ordering (Theorem \ref{thm:Blackwell}) implies our unlearning framework (Definition \ref{def:dist-unlearning}) encompasses the KL divergence-based unlearning of \cite{allouah2025distributionalmachineunlearningselective} (Proposition \ref{thm:CoD}). More generally, the framework \ref{def:dist-unlearning} encompasses every divergence-based unlearning $D(\cdot, \cdot)$ satisfying data processing inequality (Proposition \ref{prop:data-processing-functional}). This means the bivariate function $D(P, Q)$ of pairs of probability measures that satisfies $D(R\circ P, R\circ Q) \leq D(P,Q)$ for any Markov kernel $R$.

\textbf{Downstream performance.}
We now connect our unlearning objective of outputting a distribution $p$ on $(\cW, \cF_{\cW})$ that is close to $q_1$ and far from $p_1$ to a robust notion of predictive performance by measuring the generalization error difference on $\cF=\{h:(\cW, \cF_{\cW}) \to ([-1,1], \cB_{[-1,1]})\}$.
\vspace{-1 mm}
\begin{equation}
2TV(p, p_1):=\sup_{h \in \cF}|\E_p[h]- \E_{p_1}[h]| \text{ and }
2TV(p, q_1):=\sup_{h \in \cF}|\E_p[h]- \E_{q_1}[h]|.
\vspace{-1 mm}
\end{equation}
\vspace{-2mm}
 \begin{remark}[Intuition]
     We measure predictive performance of $h$ at the distributional level by looking at averaged quantities $\E_p[h], \E_{p_1}[h],\E_{q_1}[h]$ for a given loss $h$. Moreover, in our definition, we simultaneously allow different choices of $h$.  More precisely, we  robustly measure the differences $|\E_p[h]- \E_{p_1}[h]|$ and  $|\E_p[h]- \E_{q_1}[h]|$ not just for a fixed $h$, but for a collection. To be concrete, we consider the largest such class by allowing all measurable functions $h$ with $\|h\|_{\infty}\leq 1$.
 \end{remark}
\vspace{-2 mm}
\begin{restatable}{proposition}{predictive}
\label{prop:predictive}
Consider $\mathcal{P}$, a class of distributions on $(\cW, \cF_{\cW})$. For TOFs $f_d= T(P_d, Q_d), f_c=T(P_c, Q_c)$, if a distribution $p\in \cP$ satisfy $(f_d,f_c)$ unlearning with respect to $p_1, q_1 \in \cP$ \ref{def:dist-unlearning}, then,
\begin{align}
TV(p,p_1) \geq TV(P_d, Q_d) \text{ and } TV(p,q_1) \leq TV(P_c,Q_c).
\end{align}
\end{restatable}
\vspace{-2 mm}
\begin{remark}[Intuition]
    These bounds show that unlearning guarantees increased loss under the unwanted domain $p_1$ and decreased loss under the preserved one $q_1$. In this sense, our framework provides meaningful control over downstream predictive behavior under this robust notion. Moreover, with $f_d=T(N(0,1), N(\alpha, 1))$ and $f_c= T(N(0,1), N(\ve, 1))$, we have $TV(P_d, Q_d) = 2 \Phi(\frac{\alpha}{2}) -1$ and $TV(P_c, Q_c) = 2 \Phi(\frac{\ve}{2}) -1$ \citep[Equation 7.40]{polyanskiy_wu_2025}.
\end{remark}
\vspace{-2 mm}
\begin{remark}[Downstream guarantees]
    From Proposition \ref{thm:CoD}, and as shown in the appendix (Lemma \ref{lem:DG}), our $(f_d,f_c)$ unlearning framework (Definition \ref{def:dist-unlearning}) also gives analogous downstream guarantees in the sense of ~\citep{allouah2025distributionalmachineunlearningselective} with $\alpha= KL(Q_d,P_d)$ and $\varepsilon=KL(Q_c,P_c)$. 
\end{remark}
\vspace{-2 mm}

\vspace{-4mm}
\section{Results: feasible regions} \label{sec:statistics}
\vspace{-4mm}
In this section, we determine feasible regions for several parametric and non-parametric models.
\vspace{-4mm}
\subsection{Feasible region of shifted Gaussians}
\vspace{-2mm}
\begin{restatable}[Feasible Region of Shifted Gaussians]{proposition}{FFR}
\label{thm:FFR}
Let $p_1 = N(\mu_1, \Sigma), q_1= N(\nu_1, \Sigma)\in \mathcal{P}:= \{N(\mu, \Sigma): \mu \in \mathbb{R}^d\}$ be class of shifted Gaussian distributions with common covariance $\Sigma \succ 0$. Take $f_d= T(N(0,1), N(\alpha, 1))$ and $f_c= T(N(0,1), N(\ve,1))$ for $\alpha, \ve \geq 0$. Then  
\vspace{-2mm}
\begin{equation} \label{eq:GFR}
\cR(\cP, (p_1,q_1), (f_d,f_c)) =
\cR= \{p= N(\mu, \Sigma): \mu \in \mathbb{R}^d: \alpha'(\mu) \geq \alpha, \ve'(\mu) \leq \ve\} 
\end{equation}
\begin{equation} \label{eq:GFRparameters0}
\text{ where } \alpha'^2= \langle \mu_1 -\mu, \Sigma^{-1}(\mu_1-\mu) \rangle \text{ and }\ve'^2= \langle \nu_1 -\mu, \Sigma^{-1}(\nu_1-\mu) \rangle. 
\end{equation}
Let $\Delta= \|\Sigma^{-1/2}(\mu_1- \nu_1)\|_2$,  then $\cR$ is empty $\leftrightarrow$ $ \Delta + \epsilon <\alpha $. Moreover, $\cR= \overline{B}(\nu_1, \ve) \leftrightarrow \Delta - \ve \geq \alpha $, and $\cR$ is annular cap otherwise.  See the proof  \ref{proof:FFRG} for plots of the fundamental region (Figure \ref{fig:alpha_e_plot_1}). 
\end{restatable}
\vspace{-3mm}
\begin{figure}[t]\label{alpha_e_plot_2}
  \centering
  \includegraphics[width=0.9\linewidth]{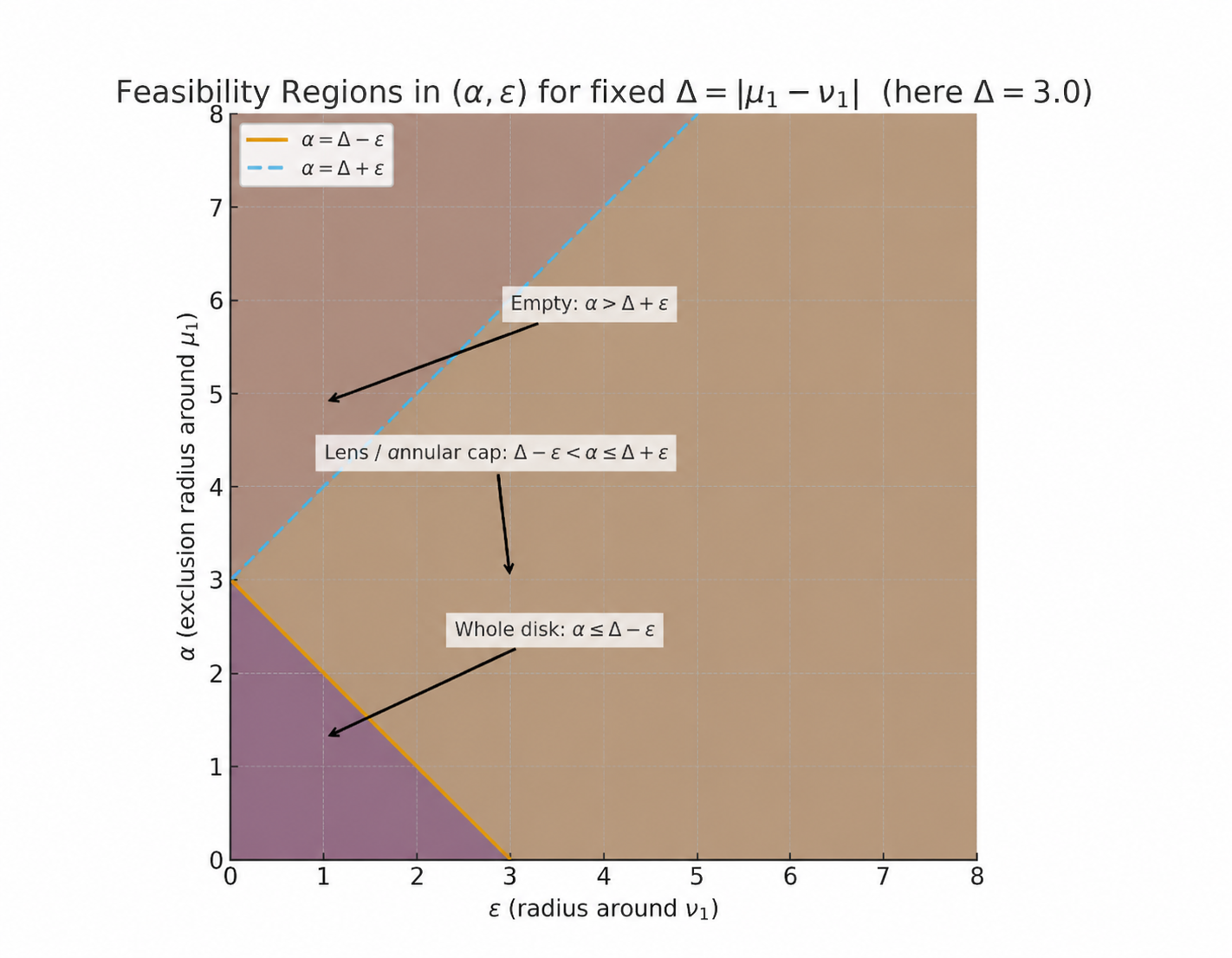}
  \caption{Feasibility regions in $(\alpha,\varepsilon)$ for fixed $\Delta=\|\mu_1-\nu_1\|_2$. 
  The lines $\alpha=\Delta-\varepsilon$ (solid) and $\alpha=\Delta+\varepsilon$ (dashed) partition the positive orthant of the plane into three regions based on what $\cR$ is:
  whole disk ($\alpha\le\Delta-\varepsilon$), lens/annular cap ($\Delta-\varepsilon<\alpha\le\Delta+\varepsilon$), and empty ($\alpha>\Delta+\varepsilon$).}
  \label{fig:alpha-eps-feasibleM}
\end{figure}
\textbf{Interpretation of the result :} This feasible region $\cR$ quantifies a natural cost of distributional unlearning: a minimal preservation loss is incurred  ($\ve \leftrightarrow f_c$) for any given removal level ($\alpha \leftrightarrow f_d$), governed by the baseline trade-off functions.  Finally, the result also highlights the suboptimality of keeping only the retained data $q_1$. While this strategy achieves perfect preservation $\ve = 0 \leftrightarrow f_c(x) = 1-x$  for $x \in [0,1]$, the region $\cR$ in Proposition \ref{thm:FFR} shows it is possible to attain a significantly higher level of removal by accepting a  small preservation loss $\ve >0 \leftrightarrow f_c(x) < 1- x$ for $x \in [0,1]$. By dimension invariance of Gaussian TOFs (Lemma \ref{lem: GTOF}) one dimensional $f_d,f_c$ are not restrictive.
\vspace{-1mm}
\begin{lemma}[Pareto Frontier] \label{lem:PFGF}
Consider $p_1, q_1, \cP$  as in Proposition \ref{thm:FFR}  with $\Sigma \succ 0$, the Gaussian forms of $f_d$ and $f_c$, one obtains a curve of $(\alpha, \varepsilon)$ pairs, capturing for a fixed preservation level of $\varepsilon \geq 0$, the largest level of removal $\alpha \geq 0$  that is jointly achievable or, equivalently, allows $\cR \neq \emptyset$.
\begin{equation} \label{eq:Pareto}
    \text{PF}(p_1, q_1, \cP) := \{(\Delta +\varepsilon, \varepsilon): \varepsilon \geq 0\} \text{ with } \Delta = ||\Sigma^{-1/2}(\mu_1- \nu_1)||_2.
\end{equation}
\end{lemma}
\begin{remark} We can take $\Sigma= \II_d$ and work in $(\mu_1', \nu_1', \mu')= (\Sigma{^{-1/2} \mu_1}, \Sigma{^{-1/2} \nu_1}, \Sigma{^{-1/2} \mu})$  coordinates. In our unlearning framework (Definition \ref{def:dist-unlearning}), the geometric description (Figure \ref{fig:alpha_e_plot_1}) as well as the Pareto frontiers (Figure \ref{fig:alpha-eps-feasibleM}) for finite-dimensional Gaussians in Proposition \ref{thm:FFR} and Lemma \ref{lem:PFGF} extends naturally to any one-parameter location family with a symmetric log-concave base noise (Proposition \ref{prop:FFRLC}), and even to non-parametric families \citep{gine2016mathematical}, \citep{castillo2024bayesian} with infinite dimensional parameter spaces such as the  Gaussian white noise model (Proposition \ref{thm:GWN}). 
\end{remark}
\vspace{-3mm}
\subsection{Feasible region of the Gaussian white noise model}
\vspace{-2mm}
We now describe the additive Gaussian white noise model, one of the canonical models of non-parametric statistics ~\citet{gine2016mathematical, johnstone2019gaussian, castillo2024bayesian}, where the parametric family $\{p_{\mu}: \mu\in \R^d\}$ of shifted Gaussians is replaced by a nonparametric family $\{p_f: f\in  H\}$ for  
\vspace{-2 mm}
\begin{equation}
H:= L^{2}([0,1], \cB_{[0,1]}, \lambda)
\text{ with }\langle f,g\rangle_{ H}:=\int_0^1 f(t)g(t)\,dt.
\vspace{-2 mm}
\end{equation}
\vspace{-2 mm}

Let \(W:= \{W(\varphi):\varphi \in H\} \) be a centered Gaussian process over \(H= L^2[0,1]\). More precisely,  for each \(\varphi\in H\), \(W(\varphi)\) is a centered Gaussian on $(\R, \cB_{\R})$  with  covariance $\mathbb E[W(\varphi)W(\psi)] =
    \langle \varphi,\psi\rangle_{H}$ for $\varphi,\psi\in H.$ Informally, $W= \{W(t): t\in [0,1]\}$ viewed as a random function, is `derivative' of Brownian motion $B$ and Gaussian white noise $W(\varphi)=\int_{0}^1 \varphi(t)dB(t)$ \cite{gine2016mathematical} is the Wiener integral. Now, for a signal \(f\in H\), the Gaussian white-noise observation is an SDE.
\begin{equation} \label{eq:AGWN}
dY_t=f(t)\,dt+\,dB_t \text{ for } 0\le t\le 1 \leftrightarrow Y(t)= \int_{0}^t f(s)ds + B(t) \text{ for } 0\le t\le 1, 
\end{equation}
 where $\{B(t): t\in [0,1]\}$ is the standard Brownian motion.
More precisely, for each test function \(\varphi\in H\), one has the observation  $Y_f(\varphi) =\int_0^1 \varphi(t)f(t)\,dt + W(\varphi) \overset{d}{=} N(\langle f, \varphi \rangle_{H}, \langle \varphi, \varphi\rangle_H ),$ as we have $W(\varphi)= \int_{0}^1 \varphi(t) dB(t)\overset{d}{=} N(0, \langle \varphi, \varphi \rangle_H)$.
Let \(p_f\) denote the distribution of the process  $\{Y_f(\varphi):\varphi\in H\}.$
For \(f,g\in H\), the likelihood ratio is   \text{ ~\citet{gine2016mathematical}[Proposition 6.1.1]}
\begin{equation}
    \log\frac{dp_g}{dp_f}(Y)
    =
    \int_0^1 (g-f)(t)\,dY_t
    -
    \left(
        \frac{\|g\|_{H}^2-\|f\|_{H}^2}{2}
    \right).
\end{equation}
Under \(p_f\), the SDE path $\{Y_t:t \in [0,1]\}$ has the form  $dY_t=f(t)\,dt+ dB_t.$ Therefore, we have
\[
\log\frac{dp_g}{dp_f}(Y)
    =
    \int_0^1 (g-f)(t)\,dB_t
    -
    \frac{\Delta^2}{2}, \text{ where } \Delta^2= \|f-g\|_{H}^2 =\langle f-g, f-g\rangle_H .
\]
\begin{equation}
 \text{Consequently, }   \log\frac{dp_g}{dp_f}
    \sim
    N\left(- \frac{\Delta^2}{2},
        \Delta^2
    \right) \quad \text{under } p_f.
\end{equation}
Under \(p_g\), the SDE path $\{Y_t:t \in [0,1]\}$ has the form  $dY_t=g(t)\,dt+ dB_t.$ Therefore,
\begin{equation}
    \text{Similarly, }
    \log\frac{dp_g}{dp_f}
    \sim
    N\left( \frac{\Delta^2}{2},
        \Delta^2
    \right)
    \quad \text{under } p_g
\end{equation}
From ~\citet{pandey2025infinitely}[Lemma 4]  $T(P,Q)= T\left(P \circ \log \left(\frac{dQ}{dP}\right)^{-1}, Q \circ \log \left( \frac{dQ}{dP}\right)^{-1}\right)$.

\begin{equation}
   \text{So, } T(p_f,p_{g})
    = 
    T\left(N\left(-\frac{\Delta^2}{2}, \Delta^2\right), N\left(\frac{\Delta^2}{2}, \Delta^2\right)\right)
=T\bigl(N(0,1),N(\Delta,1)\bigr).
\end{equation}

As a consequence, we have the following result in infinite-dimensional non-parametric families.

\begin{restatable}[Feasible region of the Gaussian White Noise Model]{proposition}{GWN}
\label{thm:GWN}

Following the notations from above, for $f,g\in H$, let $p_f, p_g \in   \mathcal{P}:= \{p_h: h \in H\}$ be the class of additive Gaussian white noise models \eqref{eq:AGWN}. Take $f_d= T(N(0,1), N(\alpha, 1))$ and $f_c= T(N(0,1), N(\ve,1))$ for $\alpha, \ve \geq 0$. Then  
\vspace{-2mm}
\begin{equation} \label{eq:GWNFR}
\cR(\cP, (p_f,p_g), (f_d,f_c)) =
\cR= \{p_h: h\in H: \alpha'(h) \geq \alpha, \ve'(h) \leq \ve\} 
\end{equation}
\begin{equation} \label{eq:GWNFRparameters0}
\text{ where } \alpha'^2(h)= \langle f-h, f-h \rangle \text{ and }\ve'^2(h)= \langle g-h, g-h \rangle_{H}. 
\end{equation}
Let $\Delta= \|f- g\|_{H}$  then $\cR$ is empty $\leftrightarrow$ $ \Delta + \epsilon <\alpha $. Moreover, $\cR= \overline{B}(\nu_1, \ve) \leftrightarrow \Delta - \ve \geq \alpha $, and $\cR$ is annular cap otherwise. Moreover, the Pareto frontier of removal-preservation pairs is given by
\begin{equation} \label{eq:ParetoGWN}
    \text{PF}(p_f, p_g, \cP) := \{(\Delta +\varepsilon, \varepsilon): \varepsilon \geq 0\} \text{ with } \Delta = ||f-g||_H.
\end{equation}
\end{restatable}
\vspace{-2 mm}
We prove a generalization of the above in the appendix for shifted Gaussians on Hilbert spaces \ref{ssec:GaussiansonHS}.
\vspace{-1mm}
\begin{remark}
    Definition~\ref{def:dist-unlearning} yields the same geometric form of the fundamental region for additive Gaussian models, in both finite and infinite dimensions. In particular, for location models with unwanted $p_1\overset{d}{=} \mu_1+X$ and desired $q_1\overset{d}{=}\nu_1+X$, our unlearning framework ~\eqref{eq:dist-unlearning} formalizes the intuition in a clear geometric way (Figure \ref{fig:alpha_e_plot_1}) that the edited data should be close to $\nu_1$ and far from $\mu_1$, with the precise tradeoff governed by the allowable removal and preservation levels $\alpha,\varepsilon$.
\end{remark}
\vspace{-5mm}
\subsection{Feasible region for Poisson family}
\vspace{-2mm}
We consider a parametric, but a non-location Poisson family $\cP= \{P(\mu): \mu >0\}$ on $(\mathbb{Z}_{\geq 0}, 2^{\mathbb{Z}_{\geq 0}})$ with baselines $f_d=T(P(1), P(\alpha))$ and $f_c=T(P(1), P(\varepsilon))$ for some $\alpha, \varepsilon >0$. The intuition is that when the data $\{X_i\}_i$, or a relevant summary statistic $\{f(X_i)\}_i$ of the data, is count-valued, for example, the number of occurrences $f(X_i)$ of a specific word in a text $X_i$, the Poisson family provides a natural model. In this setting, distributional unlearning aims to carefully remove a subset of the text corpora $\{X_i\}_i$ coming from the unwanted domain with \(f(X) \sim P(\mu_1)\) so that the edited data remain close to the desired domain with \(f(X)\sim P(\nu_1)\).  Moreover, it is natural to adapt the removal-preservation baselines $f_d,f_c$ as trade-off functions between Poisson pairs.

\begin{restatable}[Feasible region of Poissons]{proposition}{FRP}
\label{thm:FRP}

Let $p_1 = P(\mu_1), q_1= P(\nu_1)\in \mathcal{P}:= \{P(\mu): \mu >0\}$. Take $f_d= T(P(1), P(\alpha))$ and $f_c= T(P(1), P(\ve))$ for some $\alpha>1,\ve<1, \mu_1, \nu_1 > 0$. Then  
\vspace{-2mm}
\begin{equation} \label{eq:FRP}
\cR(\cP, p_1,q_1,f_d,f_c)=\cR_P= 
\left[ \nu_1 ,
\min\left\{
\frac{\mu_1}{\alpha},
\mu_1-\alpha+1, \frac{\nu_1}{\varepsilon},
\nu_1-\varepsilon+1
\right\}
\right],
\end{equation}
\begin{equation} \label{eq:FRPparameters0}
\text{ provided we have } \nu_1\leq \min\left\{
\frac{\mu_1}{\alpha},
\mu_1-\alpha+1
\right\}.
\end{equation}
Otherwise, $\cR_{P}$ is empty.  See Table \ref{table:Poissonregion} for a description of the feasible region $\cR_P$ for other pairs of  $(\alpha, \ve)$ different from $\alpha >1,\varepsilon <1$. Moreover, see Table \ref{table:Paretofrontiers} for the Pareto frontier of $(\alpha, \varepsilon)$ pairs.
\end{restatable}
\begin{remark}[Intuition]
The description of the feasible region for the Poisson family is more involved compared to a location family. For a given $\mu_1 >\nu_1>0$, it is natural from Lemma \ref{lem:Poisson-experiments} to consider $\alpha >1$ and $\ve <1$ as the baseline parameters. Moreover, the Pareto frontier of $(\alpha, \ve)$ pairs (\ref{table:Paretofrontiers}) involves several  parameters including  $\frac{\mu_1}{\nu_1}, \mu_1- \nu_1$.  The proof is a consequence of a characterization result (Lemma \ref{lem:Poisson-experiments}) of when exactly do we have $f=T(P(a), P(b)) \leq T(P(c), P(d))$ for constants $a,b,c,d>0$.
In the appendix, we provide the feasible region of the Binomial and the Bernoulli family (Lemma \ref{lem:binomregion}), along with a similar characterization result (Lemma \ref{lem:Binom}) required in the proof.
\end{remark}
\vspace{-2mm}
\textbf{Consistency with Poisson baselines:}  For $f_d=T(P(1),P(\alpha))$ and  $f_c=T(P(1),P(\ve))$, with $\alpha >1$ and $\ve <1$ if $p$ satisfies $(f_d \leftrightarrow \alpha ,f_c \leftrightarrow \ve)$ unlearning, then as a consequence of Lemma \ref{lem:Poisson-experiments} it also satisfies  $(f_{d'} \leftrightarrow \alpha' ,f_{c'} \leftrightarrow \ve')$ unlearning for any $1\leq \alpha' \leq \alpha$ and $\ve' \leq \ve < 1$.

\vspace{-4mm}
\section{Algorithms and information-computation gaps}
\vspace{-3mm}
\label{sec:theory}
The preceding sections developed our framework at the level of distributions. We now shift focus to the practically relevant finite-sample regime, where achieving $(f_d, f_c)$ unlearning must be accomplished using only the finite but practically large collections of drawn samples introduced in Section~\ref{sec:problem}. To develop rigorous analytical intuition about finite-sample behavior, we restrict attention to an idealized yet foundational setting of parametric families $\cP= \{p_{\mu}: \mu\in \mathbb{R}^d\}$ on some space $(\cW, \cF_{\cW})$. Our theoretical results are for Gaussian distributions in $d$ dimensions with variance $\sigma^2 \mathbb{I}_d$,  and removal-preservation baselines $f_d = T(\mathcal{N}(0,1), \mathcal{N}(\alpha, 1))$, $f_c = T(\mathcal{N}(0,1), \mathcal{N}(\varepsilon, 1))$. While the proofs extend to arbitrary covariances $\Sigma \succ 0$, we focus on other aspects of the problem. The analytical tractability of the Gaussian setting enables us to obtain closed-form sample complexity bounds, which yield essential intuition regarding the comparative efficiency of distinct removal strategies.  What follows presents and analyzes two distinct deletion strategies: a random baseline and a selective, distance-based approach. Our analysis concentrates on the distribution class $\mathcal{P} \coloneqq \left\{ \mathcal{N}(\mu, \sigma^2\mathbb{I}_d) \colon \mu \in \mathbb{R}^d \right\}$, with $\sigma > 0$ treated as known. Starting from $n_1$ IID samples drawn from the unwanted distribution $p_1 \in \mathcal{P}$ and $n_2$ IID samples drawn from the retained distribution $q_1 \in \mathcal{P}$, together with a deletion budget $0 \leq n_r \leq n_1$, we establish high-probability bounds on the $(f_d \leftrightarrow \alpha, f_c \leftrightarrow \varepsilon)$-distributional unlearning guarantees that result from each strategy. 

\vspace{-2mm}
\subsection{Random removal: behavior on Gaussian and log-concave families}
\vspace{-2mm}
\textbf{Random removal} We begin with a baseline deletion strategy that treats every sample equally, deleting $n_r$ points chosen uniformly at random from the $n_1$ samples of $p_1$. The formal procedure is:

\textbf{Algorithm (Random removal by deleting points from unwanted data) \citep{allouah2025distributionalmachineunlearningselective}.} 
\begin{enumerate}[nosep] \label{algo:RR}
\item Randomly select $n_r$ out of the $n_1$ samples of $p_1$ without replacement, and remove them.
\item  Fit MLE (maximum likelihood estimator) on the remaining data,  take weighted average.
\end{enumerate}

The following result provides in our framework \ref{def:dist-unlearning} a finite-sample guarantee for achieving $(f_d \leftrightarrow \alpha,f_c \leftrightarrow \varepsilon)$-distributional unlearning using random removal with a deletion budget $n_r$.

\begin{restatable}[Random Removal]{proposition}{samples}
\label{th:random}
Let $p_1=N(\mu_1, \sigma^2 \II_d), q_1= N(\nu_1, \sigma^2 \II_d) \in \mathcal{P}= \{N(\mu, \sigma^2\II_d): \mu \in \R^d\}$ and $\delta \in (0,1)$. We observe $n_1$ IID samples from $p_1$ and $n_2$ IID samples from $q_1$, and randomly remove $n_r$ samples from $p_1$ before fitting a weighted MLE according to \ref{algo:RR}. With probability at least $1 - \delta$, the resulting $p$ satisfies $(f_d \leftrightarrow \alpha,f_c\leftrightarrow  \varepsilon)$- unlearning \ref{def:dist-unlearning} with:
\begin{align*}
\alpha & \leq \alpha_M(R):= \Delta -  \ve_m(R),
\varepsilon \geq  \ve_m(R):= \frac{n_1'  \Delta }{n_1' + n_2} + \frac{n_1' \gamma(n_1',d,\delta) + n_2 \gamma(n_2, d,\delta) }{n_1' + n_2}.
\end{align*}
where $\sigma \Delta := \|\mu_1- \nu_1\|_2$, $n_1'= n_1- n_r$ and $\sqrt{n} \gamma(n, d, \delta):= \sqrt{d+2\sqrt{d \log (1/\delta)} + 2 \log (1/\delta)}$.  
\end{restatable}

\begin{remark}[Information-Computation gap] \label{rem:ICGRRM}
The feasible region $\cR$ allows any $\ve \geq 0$ and requires $\alpha \leq \Delta +\ve$ \ref{thm:FFR}. However, in the finite sample case, where we do not know the locations $\mu_1, \nu_1$, the baseline random removal with weighted MLE has a  lower bound on $\ve \geq \ve_m(R)>0$ and an upper bound $\alpha \leq \alpha_M(R)= \Delta -\ve_m(R)$ (see Remark \ref{rem:ICGRR} in the appendix for a more detailed analysis of the Information-computation gap). In both $\alpha_M(R)$ and  $\ve_m(R)$, the latter two terms involves  $\gamma(n, d, \delta)$ with weights. The quantity $\gamma(n, d, \delta)$ comes from the tail of $\|\sum_{i=1}^n X_i\|_2$ for $X_i$ IID $N(0,\II_d)$, and allow the upper bounds to hold with probability $1-\delta$. So, this is the cost in finite sample to avoid cases where samples are far from their mean. The first term involving $\Delta$ appears because weighted MLE involves the unwanted samples, and the associated weights $\frac{n_2}{n_1' + n_2}$ and $\frac{n_1'}{n_1' + n_2}$ capture that. A generalization of the above to  $\{\mu+X:\mu\in \R\}$ with symmetric log-concave $X$ (Proposition \ref{th:randomLC}), including the analysis of weighted median for Laplace (Remark \ref{rem:RRLaplace}), is in the appendix.
\end{remark}

\vspace{-4mm}
\subsection{Selective removal: behavior on Gaussian and log-concave families}
\vspace{-2mm}

For a family of distributions $\cP=\{p_{\mu}: \mu\in \Theta\}$, indexed by a metric space  $(\Theta, d(\cdot, \cdot))$  including the shifted Gaussian $\{N(\mu, \sigma^2 \II_d): \mu \in \R^d\}$ or the symmetric log-concave family $\{\mu+X: \mu\in \R\}$, a more effective strategy than random removal decides which samples to delete, based on relative separation between the distributions $p_1$ and $q_1$. For location families, our intuition suggests shifting the dataset's empirical mean away from the unwanted center $\mu_1$ and towards the retained center $\nu_1$. The most impactful samples to remove are those from $p_1$ that are furthest from $\nu_1$. Within our $(f_d,f_c)$ unlearning framework \ref{def:dist-unlearning} this intuition extends to families $\cP=\{\theta+X: \theta\in \Theta\}$ if $T(p_{\theta_1}, p_{\theta_2}) \uparrow$ whenever $d(\theta_1,\theta_2)\downarrow $. The families $\{N(\mu, \sigma^2 \II_d): \mu \in \R^d\}$ and $\{\mu+X: \mu\in \R\}$ satisfy that.
This intuition leads to the proposed selective removal strategy of ~\citep{allouah2025distributionalmachineunlearningselective}, using the empirical mean of the retained data $\hat{\mu}_2$ as a reference point for selection. 

\textbf{Algorithm (Distance based selective removal algorithm of ~\citep{allouah2025distributionalmachineunlearningselective}).}
\begin{enumerate}[nosep] \label{algo:SR}
\item Compute  MLE (maximum likelihood estimator) $\hat\nu_1$ of $n_2$ samples $x_2^{(1)}, \ldots, x_2^{(n_2)}$  from $q_1$.
\item For each of the $n_1$ samples $x_1^{(i)}$ from $p_1$, compute the score $s_i = ||x_1^{(i)} - \hat\nu_1||_2$.
\item Delete  $n_r$ samples with  largest scores $s_i$ and fit by a weighted mean on the remaining data.
\end{enumerate}
\begin{remark}[Empirical mean versus MLE]
   For Gaussians with unknown mean and known variance \ref{th:selective}, the empirical mean $\frac{1}{n} \sum_{i=1}^n X_i$ is the MLE. However, it is not  MLE for shifted log-concave families in general, such as shifted Laplace family, for which the median$(X_1, \cdots, X_n)$ is the MLE.
\end{remark}
\vspace{-2 mm}
The following result provides in our framework \ref{def:dist-unlearning} a finite-sample guarantee for achieving $(f_d \leftrightarrow \alpha,f_c \leftrightarrow \varepsilon)$-distributional unlearning using selective removal with a deletion budget $n_r$.

\begin{restatable}[Selective Removal]{proposition}{samplesselection}
\label{th:selective}
Let $p_1=N(\mu_1, \sigma^2 \II_d), q_1= N(\nu_1, \sigma^2 \II_d) \in \mathcal{P}= \{N(\mu, \sigma^2\II_d): \mu \in \R^d\}$ and $\delta \in (0,1)$. We observe $n_1$ IID samples from $p_1$ and $n_2$ IID samples from $q_1$, and selectively remove $n_r$ samples from $p_1$. With probability atleast $1 - \delta$, the resulting $p$ satisfies $(f_d= T(N(0,1), N(\alpha, 1)),f_c= T(N(0,1), N(\ve, 1)))$ unlearning \ref{def:dist-unlearning} with:
\begin{align*}
    \alpha
    \leq 
   \alpha_M(S):= \Delta - \ve_m(S),
    \varepsilon
    \geq  
    \ve_m(S):= 
    \frac{1}{\sigma}\frac{n_1'}{n_1' + n_2} 
    F_1^{-1}\left(\frac{n_1'}{n_1} +d(n_1, \delta)\right)
    + 
    2\gamma(n_2,d,\delta)
\end{align*}
where $\sigma \Delta := \|\mu_1- \nu_1\|_2$, $n_1'= n_1- n_r$, $\sqrt{n} \gamma(n, d, \delta):= \sqrt{d+2\sqrt{d \log (1/\delta)} + 2 \log (1/\delta)}$, $d(n,\delta)=\sqrt{\frac{\ln(4/\delta)}{2n}}$ and $F_1^{-1}$ is the inverse of $F_1(t):= \mathbb{P}[\|\sigma Z + \mu_1- \nu_1\|_2 \leq t] \text{ with } Z \sim N(0, \II_d).$ 
\end{restatable}
\begin{remark}[Information-Computation gap] \label{rem:ICGSR}
The feasible region $\cR$ allows any $\ve \geq 0$ and  $ 0\leq \alpha \leq \Delta +\ve$ \ref{thm:FFR}. In comparison to the random removal with $\ve \geq \ve_m(R)>0$ and  $\alpha \leq \alpha_M(R)= \Delta  -\ve_m(R)$, the selective removal provides a trade-off, with  $\ve \geq \ve_m(S)>0$ and  $\alpha \leq \alpha_M(S)=  \Delta -\ve_m(S)$. The factor $d(n_1,\delta)$ comes from  approximating $F_1$ with its empirical version $\widehat{F}_1$ (Lemma \ref{lem:DKW}). The CDF $F_1$ appears naturally, since the selective removal is based on the sample version of the ordered list of distances   $\left\{\|x_1^{(i)} - \nu_1\|_2: i\in [n_1]\right\}$ for $x_1^{(i)} \overset{d}{=} p_1\overset{d}{=} \mu_1 + \sigma Z$. A generalization to  $\{\mu+X:\mu\in \R\}$ with log-concave noise $X$ is in the appendix (Proposition \ref{th:randomLCSR}). In the appendix (Proposition \ref{th:CSRRODG}, Remark \ref{rem:CSROGH}), we analyze in detail when it is the case that $\ve_m(S)\leq \ve_m(R)$, and observe that one needs a minimal separation $\Delta \geq \Delta_m$ (Equation \ref{eq:Delta_m}) and $n_r \geq \frac{n_1 (1+ o(1))}{2}$ (Equation \ref{eq:n_r}) for the selective removal to allow a smaller preservation level $\ve_m(S)$ than random removal $\ve_m(R)$.
\end{remark}
\vspace{-2 mm}

\vspace{-4mm}
\section{Conclusion and future directions}
\label{sec:conclusion}
\vspace{-3mm}
 Machine unlearning increasingly requires moving beyond individual record deletion to erase the influence of entire subpopulations. Our theoretical study is focused on analyzing the statistical or information-theoretic tradeoff in this task: full removal is computationally expensive, while random partial removal is statistically inefficient. At the distributional level, we formalize this removal of a carefully chosen subset of the unwanted domain as \emph{distributional unlearning}, a framework that optimally balances forgetting an unwanted distribution while preserving a desired one. Our theoretical analysis describes the information-theoretic feasible regions for the edited population for several parametric and non-parametric families, and provides provable guarantees connecting this approach to downstream model performance. We also provide finite-sample guarantees for a few data selection algorithms. The statistical framework of unlearning itself is model-agnostic. Moreover, since the selection algorithms we analyze are based on a metric on data representations rather than model internals, the algorithms are model-agnostic too.

While our theoretical results establish several benefits of selective data deletion, some limitations remain and suggest natural avenues for future work. One important direction is to characterize the feasible regions for broader parametric families $\{p_{\mu}: \mu \in \Theta\}$, and to develop finite-sample or non-asymptotic concentration guarantees for weighted MLE-based procedures $\hat{\mu}$. A more basic limitation is that our framework assumes access to samples that have already been labeled as coming from the unwanted subpopulation or the desired one. In practice,  identifying all samples from such a domain may itself be a difficult upstream task. It would also be valuable to understand the behavior of the method when the available retain set is incomplete, noisy, or only approximately specified. Finally, from an algorithmic perspective, the success of distance-based removal depends on the quality of the representation space ($\Theta,\|\cdot\|$). This raises the question of how the theoretical guarantees change under alternative feature maps, particularly non-isometric representations.

\bibliographystyle{plainnat}
\bibliography{references}


\appendix
\newpage
\appendix
\section{Appendix: Basics of the unlearning framework}
\label{app:proofs}

\subsection{Statistical unlearning versus sample-level unlearning}
In this section, we further motivate the framework of statistical unlearning of distributions. 
\paragraph{Data selection.} Our distributional unlearning framework in Definition~\ref{def:dist-unlearning} addresses a different part of the entire machine unlearning pipeline than standard sample-level unlearning. We focus on the \emph{data-selection} step: once a collection of candidate samples associated with an unwanted domain or subpopulation has been identified, our goal is to select a small, carefully chosen subset whose removal most effectively erases the influence of that subpopulation at the distributional level. This selection problem is complementary to both upstream detection, which identifies candidate samples from the unwanted population, and downstream model-update procedures such as approximate retraining from scratch, which modify a trained model after a forget set has been specified. 

\paragraph{Statistical framework. } We do not propose an algorithmic method for efficiently removing prescribed samples from a trained model, nor do we claim formal model-level privacy or deletion guarantees. Instead, our framework isolates the statistical question of which data points should be removed, while treating sample-level unlearning methods as downstream black-box procedures on the edited data. Once the carefully chosen forget set is selected, any sample-level unlearning algorithm may consume it as input to update the model. We also emphasize that identifying all samples belonging to an unwanted domain can itself be a challenging upstream problem in practice. Finally, the deletion budget may be chosen according to computational constraints: given a downstream sample-level unlearning method, its cost scaling can be translated into a maximum acceptable forget-set size, and our selection rule can then be applied under this budget.

\paragraph{Sample versus computational efficiency. }
The computational overhead of sample-level unlearning methods grows directly with the cardinality of the forget set. Influence-function-based techniques, such as that of \citet{guo2020certified}, necessitate computing gradients or Hessian-vector products over forget samples in order to approximate or execute a second-order parameter update such as \citet{zou2025certified}, \citet{pandey2026gaussian}. Alternative methods, such as SalUn \citet{fan2024salun}, locate and adjust salient model weights by performing backpropagation through the network for every sample in the forget set. Shrinking the forget set therefore directly reduces the number of gradient computations required, translating into tangible runtime savings. This dependence on forget set size applies equally to the broad family of certified unlearning methods that operate via noise injection \citet{chourasia2023forget, allouah2024utility, waerebeke2025when}.

\paragraph{Computational complexity of data selection. }

The selection strategies considered in this theoretical work require negligible to moderate computational demands. The random removal method imposes negligible computational overhead, with complexity scaling roughly linearly in both the number of forget samples $n_1$ and the desired ones $n_2$. As we compute the weighted empirical mean after the random removal of samples from the forget set, we require an overall cost of approximately $\Theta(n_1 +n_2)$. The selective removal algorithm carries a moderate computational burden of computing the distance for every point in the forget set to the empirical mean of the retained set, and then sorting them before computing the weighted average.

\subsection{Preliminaries on Gaussian trade-off functions}\label{sec:prelimGTOF}
Our \textbf{first} lemma summarizes some of the basic but extremely fruitful properties of the Gaussian trade-off function $f_{G, \ve}(x)=T(N(0,1), N(\ve,1))(x)= \Phi(\Phi^{-1}(1-x) - \ve)$ for all $x \in [0,1]$. 

\begin{lemma} (\cite{dong2022gaussian})\label{lem: GTOF} The Gaussian trade off functions satisfy the following properties:
\begin{itemize}
\item Monotonicity:  For any pair  $\ve_1, \ve_2  \geq0$, $\ve_1 \leq \ve_2$ if and only if  
\begin{equation} \label{eq:GMontone}
   f_{ G,\ve_1}(x)  \geq f_{G,\ve_2}(x) \text{ for all } x \in [0,1]. 
\end{equation}
\item Closure under suprema:  For any collection of $(\ve_i)_{i\in I} \subset 
\R$ with index set $I$
\begin{equation} \label{eq:Gsuprema}
   \inf_{i\in I} f_{ G,\ve_i}(x) = f_{G,\sup_{i\in I}|\ve_i|}(x) \text{ for all } x \in [0,1]
\end{equation}
\item Symmetry: $f_{G,\varepsilon}= f_{G,-\varepsilon}$ for $\varepsilon >0$. For any  $\mu_1, \nu_1 \in \R^p$ and $\sigma >0$ 
\begin{equation} \label{eq:symmetry}
   T(\nu_1 + \sigma N(0, \II_p), \mu_1+ \sigma N(0, \II_p)) =   T(\mu_1+ \sigma N(0, I_p), \nu_1 + \sigma N(0, I_p))
\end{equation}
\item Dimension freeness: Let $\mu_1, \nu_1 \in \R^p$, $\Sigma  \succ 0$ and  $\ve^2 := \langle (\mu_1- \nu_1), \Sigma^{-1} (\mu_1- \nu_1)\rangle $. Then
\begin{equation} \label{eq:dto1a}
    T(\mu_1+ \sqrt{\Sigma} N(0, \II_p), \nu_1 + \sqrt{\Sigma}N(0, \II_p)) \equiv T(N(0,1), N(\ve, 1)), 
\end{equation}
\end{itemize}
 \end{lemma}
\textbf{Intuition and importance:} The proof of this lemma is immediate from the explicit description of $f_{G, \ve}(x)=\Phi(\Phi^{-1}(1-x) - \ve)$ where $\Phi(\cdot)$ is the CDF of a standard Gaussian variable $\Phi(t)= \mathbb{P}[N(0,1)\leq t]$ for all $t\in \mathbb{R}$. It also requires applying the Neyman-Pearson lemma or the likelihood ratio test \cite{polyanskiy_wu_2025}. But, the conclusions that they imply are extremely powerful. 
\par
\textbf{Monotonicity:} The \textit{monotonicity} condition \eqref{eq:GMontone} reduces an apriori difficult functional comparison between two functions $f$ and $g$ at uncountably many points to a comparison of just one parameter $\varepsilon$.
\par
\textbf{Closure under suprema:} The closure under suprema property \eqref{eq:Gsuprema} essentially says two very important things. First, it makes it easy to identify what the suprema of an apriori arbitrary collection of functions $\{f_i\}_i$ is (in fact explicitly). Second, the limiting object is a function of the same kind: it is again a Gaussian trade-off function with a different choice of parameter $\ve$.
\par
\textbf{Symmetry:} The symmetry property  \eqref{eq:symmetry}  (requires Neyman-Pearson lemma) is very interesting because the definition of a trade-off function $T(P,Q)$ ~\ref{def:TOF} is asymmetric in general, between its first and second arguments.  But, for a pair of shifted isotropic Gaussians, they match because of the spherical symmetry (orthogonal invariance) of the standard Gaussian density in any dimension.
\par
\textbf{Dimension freeness:} The dimension freeness property \eqref{eq:dto1a}  (requires Neyman-Pearson lemma)  makes the case for Gaussian certifiability in high dimensions stronger than any other notion of certifiability. This is because, almost all the results of classical statistics that are true in low dimensions, fail to hold in high dimensions, because many of the quantities involved in controlling the errors are dimension dependent and blows up when $p \uparrow \infty$. It is often the case that finding a dimension-free quantity or even an inequality that \textit{tensorizes} \footnote{We are referring dimension free Poincare and Logarithmic Sobolev inequalities of high-dimensional statistics that are extremely important in obtaining concentration bounds in high dimensions \cite{vanhandelAPC550}.} help us resolve high dimensional issues.

\subsection{Preliminaries on log-concave trade-off functions}\label{sec:prelimLCTOF}
Our \textbf{next} lemma says that the some aspects of the Gaussian trade-off functions generalize. 
\begin{lemma}\cite{dong2022gaussian}[Lemma A.2, Proposition A.3] \label{lem:SLCTOF}
Consider a \textit{symmetric} random variable $X\overset{d}{=} -X$ with CDF $F$ having a log-concave Lebesgue density. Consider its  baseline trade-off function $f_{X,\ve}(x)=T(X,\ve+X)(x)$ for all $x \in [0,1]$. 
\begin{enumerate}
    \item \textit{Symmetry:} For any $t_1,t_2 \in \mathbb{R}$ and $\delta = |t_2- t_1|$ the following holds  for all $x \in [0,1]$.
    \begin{equation}
T(t_1+X, t_2+X)(x) = F\left(F^{-1}(1-x)-\delta \right)=T(t_2+X,t_1+X)(x). 
\end{equation}
\begin{equation}
\text{So, }T(t_1+X, t_2+X)(x)= T(X,\delta+X)(x)=T(\delta+X,X)(x) 
\end{equation}
\item \textit{Monotonicity:}
For $\ve_1 ,\ve_2 \geq 0$. Then $\ve_1 \leq \ve_2$ if and only if 
\begin{equation}
    f_{X,\ve_1}(x) \geq f_{X,\ve_2}(x) \text{ for all } x \in [0,1].
\end{equation}
\end{enumerate}
\end{lemma}
\textbf{Intuition and importance:}  The proof of this lemma is along the same lines of the proof given in \cite{dong2022gaussian}[Lemma A.2, Proposition A.3]. Using log-concavity of the density of $X$ and symmetry $X \overset{d}{=}-X$, there is an explicit description of $f_{X, \ve}(x)=F(F^{-1}(1-x) - \ve)$ where $F(\cdot)$ is the CDF of the random variable $X$, $F(t)= \mathbb{P}[X\leq t]$ for all $t\in \mathbb{R}$, and $\ve \geq 0$.

\subsection{Behavior of  trade-off functions under taking inverses}\label{sec:prelimTOFI}
Our \textbf{next} lemma describes the  behavior of pairs of  tradeoff-functions under taking inverses.
\begin{lemma}\label{lem:ITOF}
Let \(f,g:[0,1]\to[0,1]\) be nonincreasing functions (for all $x\geq y \implies f(x) \leq f(y)$) 
\[
  \text{such that }  f(x)\leq g(x) \text{ for all } x\in[0,1].
\]
For a nonincreasing function \(h:[0,1]\to[0,1]\), define its generalized left inverse by
\[
    h^{-1}(t)
    :=
    \inf\{x\in[0,1]: h(x)\leq t\},
    \text{ for } t\in[0,1],
\]
with the convention that \(\inf\varnothing=1\). Then
$  f^{-1}(t)\leq g^{-1}(t),
    \text{ for } t\in[0,1].$
\end{lemma}
\begin{remark} \label{rem:ITOF}
    \citet{dong2022gaussian}[Lemma A.2] shows that if $f=T(P,Q)$ is a trade-off function for distributions $P,Q$ on some $(\cW, \cF_W)$ , then $f^{-1}$ is also a trade-off function, and $f^{-1}= T(Q,P)$.
\end{remark}
\begin{proof}
Fix \(t\in[0,1]\). Define the sublevel sets for the functions $f$ and $g$ as
\[
    A_f(t):=\{x\in[0,1]: f(x)\leq t\},
    \qquad
    A_g(t):=\{x\in[0,1]: g(x)\leq t\}.
\]
We observe that $A_g(t)\subseteq A_f(t)$, because, if \(x\in A_g(t)\), then \(g(x)\leq t\) $\implies  f(x)\leq g(x)\leq t$. So, \(x\in A_f(t)\) and therefore \(A_g(t)\subseteq A_f(t)\). As a consequence we have 
\[
     f^{-1}(t)= \inf A_f(t)\leq \inf A_g(t)= g^{-1}(t) \text{ for any } t\in[0,1].
\]
\end{proof}
\subsection{Blackwell ordering and the data-processing inequality}
Our \textbf{next} theorem describes  that an ordering of trade-off functions $T(P,Q)=f \leq g =T(P',Q')$ is equivalent to data-processing $(R\circ P= P', R\circ Q= Q')$, also known as Blackwell ordering.
\begin{theorem}[Blackwell ordering]  \citet{dong2022gaussian}[Theorem 2] \label{thm:Blackwell}
Let \(P,Q\) be probability distributions on some measurable space $(\cW, \cF_{\cW})$ and let
\(P',Q'\) be probability distributions on some measurable space $(\cW', \cF_{\cW'})$. The following
two statements are equivalent:
\begin{enumerate}
    \item \(f=T(P,Q) \leq T(P',Q')=g\) $\leftrightarrow f(x) \leq g(x)$  for all $x\in [0,1]$.
    \item There exists a Markov kernel \footnote{A Markov kernel  $\mathrm{R}: (\cW, \cF_{\cW})  \to (\cW', \cF_{\cW'})$ is a measurable collection $\{R_w:w\in \cW\}$ of probability measures on $(\cW', \cF_{\cW'})$. We refer the reader to \cite{kallenberg2021foundations}  for more details on Markov kernels.}
        $\mathrm{R}: (\cW, \cF_{\cW})  \to (\cW', \cF_{\cW'})$
    such that
    \[
        R \circ P=P' \text{ and }
        R \circ Q=Q',  \text{ where}
    \]
   for a probability distribution  $p$ on  $(\cW, \cF_{\cW})$, we define the probability distribution $R\circ p$ as   
    \begin{equation}  \label{eq:MKDef}
    R\circ p(A'):= \int_{\cW} R(w, A') dp(w)    \text{ for } A' \in \cF_{\cW'}.
    \end{equation}
\end{enumerate}
\end{theorem}
Our \textbf{next} lemma is a consequence of Blackwell ordering \ref{thm:Blackwell}. It describes that any bivariate functional $D(P,Q)$ of probability measures $P,Q$ satisfying data processing inequality $D(R\circ P, R \circ Q) \leq D(P,Q)$ for all probability distributions $P,Q$ on some $(\cW, \cF_{\cW})$ and Markov kernels $R:=\{R_w:w \in \cW\}$ to some $(\cW', \cF_{\cW'})$ has to be a non-increasing functional $\ell_D(T(P,Q))$.

\begin{lemma} \citet{dong2022gaussian}[Proposition B.1, Lemma B.2, Corollary B.3]
\label{prop:data-processing-functional}
 Let \(D(\cdot\|\cdot)\) be a bivariate functional  that takes two probability distributions $P,Q$ on some measurable space $(\cW,\cF_{\cW})$ and outputs a real number (including $\pm \infty$). If it satisfies the data processing inequality, which means
\begin{equation} \label{eq:DPI}
    D(R\circ P, R \circ Q) \leq D(P,Q)
\end{equation}
for all probability distributions $P,Q$ on some measurable space $(\cW, \cF_{\cW})$ and Markov kernels $R:=\{R_w:w \in \cW\}$ to some measurable space $(\cW', \cF_{\cW'})$. Then  there exists
a functional
\begin{equation}
    \ell_D:\cT \to \mathbb R \cup \{\pm \infty\} \text{ such that }
    D(P, Q)=\ell_D\bigl(T(P,Q)\bigr), \text{ where} 
\end{equation}
\begin{equation}
    \cT= \{T(P,Q): [0,1] \to [0,1]:  (P,Q) \text{ on some } (\cW,\cF_{\cW})\}.
\end{equation}
Moreover, this functional $\ell_D$ is non-increasing. More precisely, for  $f= T(P,Q)$ and $g= T(P',Q') $
\begin{equation}
   \text{if } f \leq g \text{ then } D(P'\|Q') =\ell_D(g)\leq  \ell_D(f) =D(P\|Q).
\end{equation}

In particular, if $T(P,Q)=T(P',Q')$, 
then $ D(P\|Q)=D(P'\|Q')$.
\end{lemma}

Now, we work towards the proof of the composition rule \ref{lem:CDU}.  First, we collect the basic properties of the product operation \(\otimes\). Let $\cP, \cP'$ be collections of probability measures on $(\cW, \cF_{\cW})$ and $(\cW', \cF_{\cW'})$ respectively. We define  $\cP\otimes \cP':= \{p \otimes p': p\in \cP, p'\in \cP'\}$ as the product family, where $p\otimes p'$ is the product distribution on $(\cW \times \cW', \cF_{\cW} \otimes \cF_{\cW'})$ defined as 
\[
\int_{\cW \times \cW'} g(w) g'(w') dp \otimes p'(w,w')= \int_{\cW} g(w) dp(w) \int_{\cW'} g'(w') dp'(w')
\]
for all bounded measurable functions $g: (\cW, \cF_{\cW}) \to (\R, \cB_{\R})$ and $g': (\cW', \cF_{\cW'}) \to (\R, \cB_{\R})$. For trade-off functions $f= T(P,Q)$ and $f'= T(P',Q')$, we define the product  
\begin{equation} \label{eq:TPTOF}
    f\otimes f':= T(P\otimes P', Q\otimes Q')
\end{equation}

\begin{proposition}[Tensor product sturcture and ordering] \citet{dong2022gaussian}[Propositions C.2, D.1]
\label{prop:tensor-basic-properties}
The product \(\otimes\)of trade-off functions  defined in \ref{eq:TPTOF} has the following properties:
\begin{enumerate}
    \item The product \(\otimes\) is well-defined, commutative and associative.
    \item \(f \otimes \mathrm{Id} = \mathrm{Id} \otimes f = f\) for $Id(\alpha)= 1-\alpha$, and \((f \otimes g)^{-1} = f^{-1} \otimes g^{-1}\). 
    \item If \(g_1 \geq g_2\), then \(f \otimes g_1 =g_1\otimes f \geq   g_2\otimes f= f \otimes g_2\).
    \item For Gaussian trade-off functions $G_{\mu}= T(N(0,1), N(\mu, 1))$ we have
    \begin{equation}
        G_{\mu_1} \otimes G_{\mu_2} \otimes \cdots \otimes G_{\mu_n}
        =
        G_{\mu},
        \qquad
        \text{where }
        \mu = \sqrt{\mu_1^2+\cdots+\mu_n^2}.
    \end{equation}
\end{enumerate}
\end{proposition}
\section{Appendix: compositional rules and downstream guarantees}
\subsection{Proofs for compositional rules of statistical unlearning}
\CDU*
\begin{proof}
    It follows from the preservation of Blackwell ordering under the product of trade-off functions \ref{prop:tensor-basic-properties}. More precisely,  if $T(p,p_1)\leq f_d$, then for any  $g=T(P,Q)\in \cT$ $T(p,p_1)\otimes g \leq f_d\otimes g$. 
    
    Now, observe that if $p \in \cP$ satisfy $(f_d=T(P_d,Q_d),f_c=T(P_c,Q_c))$ unlearning with respect to  probability distributions $p_1,q_1$ in the family $\cP$, then we have
    \[
    T(p,p_1) \leq f_d= T(P_d,Q_d) \text{ and } T(p,q_1) \geq f_c = T(P_c,Q_c)
    \]
    Moreover, the probability distribution $p' \in \cP'$ satisfy $(f_d'=T(P_d',Q_d'),f_c'=T(P_c',Q_c'))$ unlearning with respect to  probability distributions $p_1',q_1'\in \cP'$ implies that 
    \[
    T(p',p_1') \leq f_d'= T(P_d',Q_d') \text{ and } T(p',q_1') \geq f_c'= T(P_c',Q_c')
    \]
    Combining the two assumptions on $p$ and $p'$ and the fact stated above, we have the following 
    \begin{equation}
       T(p\otimes p', p_1\otimes p_1')
       =
       T(p,p_1) \otimes  T(p',p_1') 
       \leq
       f_d \otimes T(p',p_1') 
       \leq 
       f_d\otimes f_d'
    \end{equation}
    \begin{equation}
        T(p\otimes p', q_1\otimes q_1') 
        =
        T(p,q_1) \otimes  T(p',q_1') 
        \geq 
        f_c \otimes  T(p',q_1') 
        \geq f_c\otimes f_c'
    \end{equation}
\end{proof}
\subsection{Comparison with \citet{allouah2025distributionalmachineunlearningselective}: the unlearning framework}
\CoD* 
\begin{proof} \label{proof:CoD}
    The proof follows from two observations from \cite{dong2022gaussian}[Lemma A.2] and from \ref{lem:ITOF}. The first one along with \ref{lem:ITOF} says if the trade-off functions $f=T(P, Q)$ and $g=T(P',Q') $ satisfy $f\leq g$ for probability distributions $P,Q$ on some $(\cW, \cF_W)$ and $P',Q'$ on some $(\cW', \cF_W')$, then we have  $f^{-1}=T(Q, P) \leq T(Q',P')=  g^{-1}$. As a consequence, we have the following
    \begin{equation}
        T(p,p_1) \leq f_d=T(P_d,Q_d) \implies   T(p_1,p) \leq T(Q_d,P_d),
    \end{equation}
    \begin{equation}
        T(p,q_1) \leq f_c= T(P_c,Q_c) \implies   T(q_1,p) \leq T(Q_c,P_c).
    \end{equation}
    Now, since KL divergence satisfy data-processing inequality $KL(R\circ P,R\circ Q) \leq KL(P,Q)$ \cite{polyanskiy_wu_2025}[Theorem 2.17], we have the following immediately from \ref{prop:data-processing-functional}.
    \begin{align*}
KL(p_1, p) \geq \alpha =KL(Q_d,P_d) \quad\text{(removal)},
\\
\quad 
KL(q_1, p) \leq \varepsilon= KL(Q_c, P_c) \quad\text{(preservation)}.
\end{align*}
\end{proof}

\begin{remark}
    We observe that in comparison with \ref{thm:FFR}, if $f_d= T(N(0,1), N(\alpha, 1))$ and $f_c= T(N(0,1), N(\varepsilon, 1))$, then a quick computation reveals $KL(Q_d, P_d)= \frac{\alpha^2}{2}$ and $KL(Q_c, P_c)= \frac{\ve^2}{2}$.
\end{remark}
\subsection{Comparison with \citet{allouah2025distributionalmachineunlearningselective}: downstream guarantees}
 Now, we show that under our $(f_d,f_c)$ distributional unlearning framework \ref{def:dist-unlearning}, we have the same downstream guarantees as in \citet[Proposition 2]{allouah2025distributionalmachineunlearningselective} with $\alpha= KL(Q_d,P_d)$ and $\varepsilon= KL(Q_c,P_c)$. First,  recall that for proving downstream guarantees, \citet{allouah2025distributionalmachineunlearningselective} considers a predictor (Markov kernel)  $h:(\cX, \cF_{\cX}) \to \cP(\cY, \cF_{\cY}),$ with \(\mathcal X\) is the input space and \(\cP(\cY, \cF_{\cY})\) the probability simplex over label space \(\mathcal Y\), trained on a
distribution \(p\) over \(\mathcal X\times\mathcal Y\) that satisfies \((f_d,f_c)\)-distributional unlearning with respect to $p_1,q_1\in \cP \subset \cP(\cX \otimes \cY, \cF_{\cX}\otimes \cF_{\cY})$. \citet{allouah2025distributionalmachineunlearningselective} measures the performance of \(h\) under the true data-generating distributions \(p_1\) (unwanted) and \(q_1\) (desired) with the expected log-loss 
\begin{equation}
    \mathcal L(h;p)
    :=
    \mathbb 
    E_{(x,y)\sim p}\bigl[\ell(y,h(x))\bigr]
    =
    \mathbb 
    E_{(x,y)\sim p}\bigl[-\log h(x)(y)\bigr]
    .
\end{equation}

\begin{proposition}\cite[Proposition 2]{allouah2025distributionalmachineunlearningselective} \label{lem:DG} Consider input space $\cX,$ prediction space $\cY,$ and distributions $ p_1,q_1 \in \cP \subset \cP(\cX\times \cY, \cF_{\cX}\otimes \cF_{\cY})$ as above.  Let \(h\) minimize \(\mathcal L(h;p)\) over $\cP$, and let \(h_1,h_2\) be optimal
predictors under \(p_1,p_2\in\mathcal P\), respectively. If \(p\)
satisfies \((\alpha,\varepsilon)\)-distributional unlearning in the sense of \citet{allouah2025distributionalmachineunlearningselective} with respect
to \((p_1,p_2)\), then
\begin{equation}
    \mathcal L(h;p_1)-\mathcal L(h_1;p_1)
    \geq
    \alpha-\delta_1,
    \text{ and }
    \mathcal L(h;p_2)-\mathcal L(h_2;p_2)
    \leq
    \varepsilon-\delta_2, \text{ where}
\end{equation}
$\delta_1 := \mathrm{KL}(p_{1,X}\,\|\,p_X),$ and $ \delta_2 := \mathrm{KL}(p_{2,X}\,\|\,p_X),$
denote the marginal KL divergence over inputs.
\end{proposition}
\begin{remark}
    From \ref{thm:CoD}, it is immediate that with the same setup as in the proposition above, if a distribution $p\in \cP$ satisfies $(f_d=T(P_d,Q_d),f_c=T(P_c,Q_c))$ unlearning \ref{def:dist-unlearning} with respect to $p_1,q_1\in \cP$, then it satisfies $(\alpha, \varepsilon)$ unlearning of \citet{allouah2025distributionalmachineunlearningselective} with $\alpha = KL(Q_d,P_d)$ and $\varepsilon =KL(Q_c,P_c)$, and hence the conclusion of the above proposition. 
\end{remark}
\subsection{Robust predictive performance and distributional guarantees}

\predictive*

\begin{proof}
    We observe that for the class of functions $\cF= \{h: (\cW, \cF_{\cW}) \to ([-1,1], \cB_{[-1,1]})\}$ the scaled total variation distance defined as $2TV(p,p_1)= \sup_{h\in \cF} |p[h]-p_1[h]|$ is a bivariate divergence functional that satisfies the data-processing inequality  $TV(R\circ P, R\circ Q) \leq TV(P,Q)$ \citet[Theorems 7.4 and 7.7]{polyanskiy_wu_2025}, and therefore, from \ref{prop:data-processing-functional} we have 
    \begin{equation*}
        T(p,p_1) \leq T(P_d, Q_d) \implies TV(p,p_1) \geq TV(P_d,Q_d)
    \end{equation*}
    \begin{equation*}
        T(p,q_1) \geq T(P_c, Q_c) \implies TV(p,q_1) \leq TV(P_c,Q_c).
    \end{equation*}
\end{proof}

\begin{remark}[Different classes of $\cF$]
    Depending on requirements, one can consider a different class $\cG$ of functions on $(\cW, \cF_{\cW})$ than $\cF$ to have $\cG(p, p_1):=\sup_{h \in \cG}|\E_p[h]- \E_{p_1}[h]|$. However, 
    \begin{equation*}
    T(p,p_1) \leq T(P_d,Q_d) \text{ for probability  distributions } P_d, Q_d \text{ on }(\cW_d, \cF_{\cW_d}) \text{ implies  there exists}
    \end{equation*}
    a Markov kernel $R_d:(\cW, \cF_{\cW}) \to (\cW_d, \cF_{\cW_d})$ such that $R_d\circ p= P_d$ and $R_d\circ p_1= Q_d$. So\footnote{For a Markov kernel $(\cW, \cF_{\cW}) \overset{R_d}{\to}(\cW_d, \cF_{\cW_d})$ and a function $(\cW_d, \cF_{\cW_d} ) \overset{h}{\to} (\R, \cB_{\R})$, we define a function $R_d \circ h : (\cW, \cF_{\cW}) \to (\R, \cB_{\R})$ as $R_d \circ h (w):= \int_{\cW_d} R(w, dw_d) h(w_d)$. Then, we have \citet{kallenberg2021foundations} $R_d \circ p [h]=  \int_{\cW_d} h(w_d) \int_{\cW} R(w, dw_d) dp(w) =  \int_{\cW} dp(w) \int_{\cW_d} h(w_d)  R(w, dw_d)  =p [R_d \circ h]$.} ,
     \begin{align*}
    \cG_d(P_d,Q_d)
    &
    := \sup_{h \in \cG_d}|P_d[h]- Q_d[h]|
    = \sup_{h \in \cG_d}|p[R_d\circ h]- p_1[R_d\circ h]|
    \\
    &
     = \sup_{h \in R_d\circ \cG_d}|p[h]- p_1[h]| 
     =: R_d\circ \cG_d (p,p_1) \leq \cG(p,p_1)
     \end{align*}
     whenever $R_d \circ \cG_d \subset \cG$. Similarly, $T(p,q_1) \geq T(P_c,Q_c)$ for some $P_c, Q_c$ on $(\cW_c, \cF_{\cW_c})$ implies there exists  Markov kernel $R_c:(\cW_c, \cF_{\cW_c}) \to(\cW, \cF_{\cW})$ with $R_c\circ P_c= p$ and $R_c\circ Q_c= q_1$. So, 
     \begin{align*}
     \cG(p,q_1) 
     := \sup_{h \in \cG}|p[h]- q_1[h]|
      = \sup_{h \in \cG}|P_c[R_c \circ h]-  Q_c[R_c \circ h]|
     \\
     = \sup_{h \in R_c\circ \cG}|P_c[h]-  Q_c[h]|.
     =: R_c\circ \cG (P_c, Q_c) \geq \cG_c(P_c, Q_c)
     \end{align*}
     whenever $\cG_c \subset R_c\circ \cG$. Combining all the above, first, we have an implicit existence of the Markov kernels $R_d$ and $ R_c$ for a given $\cP \ni p, p_1, q_1$, with $f_d=T(P_d,Q_d)$ and $ f_c=T (P_c,Q_c)$. Now, we have the robust downstream guarantees for a class of functions $h\in \cG$ from $(\cW, \cF_W) \to (\R, \cB_{\cR})$
     \begin{equation}
         \cG(p,p_1) = R_d^{-1} \circ  \cG (P_d, Q_d) \text{ and } \cG(p,q_1) = R_c \circ \cG (P_c,Q_c).
     \end{equation}
     However, the description of the Markov kernels $R_d, R_c$ is often implicit given the quantities $\cP, p,p_1, q_1, P_d, Q_d, P_c, Q_c$, and therefore finding appropriate bounds on the quantities $R_d^{-1} \circ  \cG (P_d, Q_d)$ and $R_c \circ \cG (P_c,Q_c)$ remains an interesting future direction of research.
\end{remark}

\section{Appendix: information theoretic feasible regions}
\subsection{Feasible region of Gaussian family}
\FFR*
\begin{proof} \label{proof:FFRG}
Let $p = N(\mu, \Sigma) \in \mathcal{P}$. From the properties of Gaussian trade-off functions \eqref{lem: GTOF} we have  
\begin{equation}
T(p,p_1)= T(N(\mu, \Sigma), N(\mu_1, \Sigma)) = T(N(0,\Sigma), N(\mu_1- \mu, \Sigma)) = T(N(0,1), N(\alpha',1 )) 
\end{equation}
\begin{equation}
    T(p,q_1)= T(N(\mu, \Sigma), N(\nu_1, \Sigma)) = T(N(0,\Sigma), N(\nu_1- \mu, \Sigma)) = T(N(0,1), N(\ve',1 )) 
\end{equation}
\begin{equation} \label{eq:GFRparameters}
\text{ where } \alpha'^2= \langle \mu_1 -\mu, \Sigma^{-1}(\mu_1-\mu) \rangle \text{ and }\ve'^2= \langle \nu_1 -\mu, \Sigma^{-1}(\nu_1-\mu) \rangle. \text{ Now,}
\end{equation}
\begin{equation}
    T(p,p_1) = T(N(0,1), N(\alpha',1 ))  \leq T(N(0,1), N(\alpha,1 )) =f_d \leftrightarrow \alpha' \geq \alpha
\end{equation}
\begin{equation}
    T(p,q_1) = T(N(0,1), N(\ve',1 )) \geq T(N(0,1), N(\ve,1 )) =f_c \leftrightarrow \ve' \leq \ve
\end{equation}

So, given the class $\cP$ of shifted Gaussian distributions with a common covariance $\Sigma$,  $p_1,q_1$, and $f_d,f_c$ as above, the feasible region $\cR$ has the following explicit form with $\alpha', \ve'$ as above.
\begin{equation}
\cR= \{p= N(\mu, \Sigma): \mu \in \mathbb{R}^d: \alpha'(\mu) \geq \alpha, \ve'(\mu) \leq \ve\}
\end{equation}
Now, we determine the boundary of this region $\cR$ and for interpretability take $\Sigma= \II_d$ \footnote{We could  work in the transformed coordinates $(\mu_1', \nu_1', \mu')= (\Sigma{^{-1/2} \mu_1}, \Sigma{^{-1/2} \nu_1}, \Sigma{^{-1/2} \mu})$ in general.}  so that 
\begin{equation}
\cR= \{\mu \in \mathbb{R}^d: \|\mu_1- \mu\| \geq \alpha, \|\nu_1 - \mu\| \leq \ve\}= \overline{B}(\nu_1,\varepsilon)\ \cap\ \big(\mathbb{R}^d \setminus B(\mu_1,\alpha)\big).
\end{equation}
In words, the region is the closed $\varepsilon$-disk around $\nu_1$ (denoted as $\overline{B}(\nu_1,\varepsilon)$) but outside the open $\alpha$-ball around $\mu_1$ $\big(\text{denoted as }\mathbb{R}^d \setminus B(\mu_1,\alpha)\big)$. Let \(\Delta=\|\mu_1-\nu_1\|\). 
\begin{figure}[h] 
  \centering
  \includegraphics[width=\linewidth]{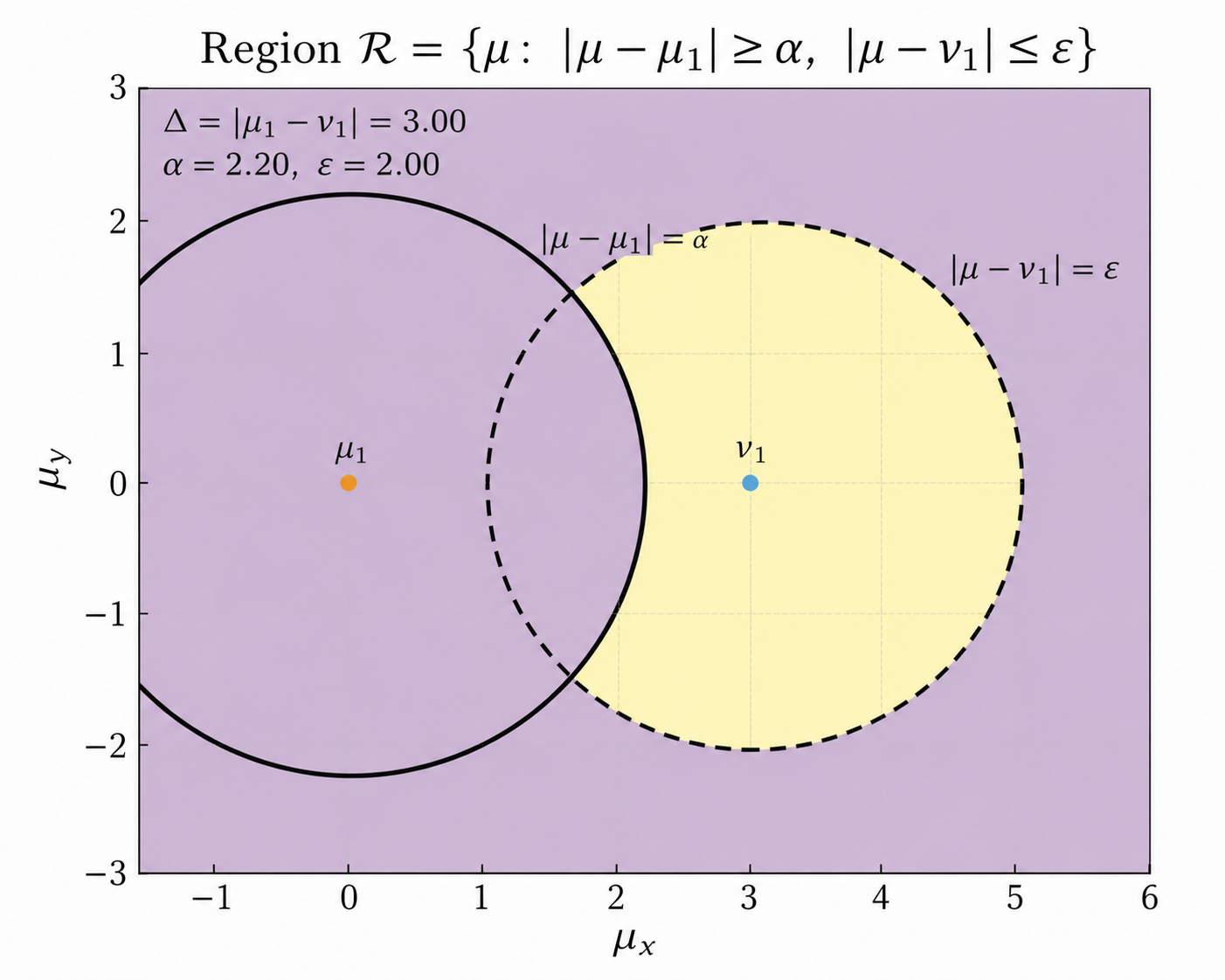}
  \caption{Region $\mathcal{R}=\{\mu:\ \|\mu-\mu_1\|\geq \alpha,\ \|\mu-\nu_1\|\leq \varepsilon\}$ for sample parameters.}
  \label{fig:alpha_e_plot_1}
\end{figure}
In the image \ref{fig:alpha_e_plot_1}, it is the yellow region. Now, it is immediate that $\cR$ is an \textbf{empty} set $\leftrightarrow \overline{B}(\nu_1, \ve)\subset B(\mu_1, \alpha) \leftrightarrow \Delta+\varepsilon<\alpha$ (the whole \(\varepsilon\)-disk of $\nu_1$ is too small and the radius $\alpha$ of the disk at $\mu_1$ is large enough). It follows that for a given $p_1 \leftrightarrow \mu_1, q_1 \leftrightarrow \nu_1$ and for a fixed  $f_c \leftrightarrow \ve $, the largest choice of $\alpha$ allowed for feasibility of such a $p\in \cR$ is given by $\Delta+ \ve$. This generalizes the concept of Pareto frontier in \cite{allouah2025distributionalmachineunlearningselective}. Moreover,  the region is the \textbf{whole closed disk:}   \(\mathcal{R}=\overline{B}(\nu_1,\varepsilon) \leftrightarrow \overline{B}(\nu_1,\varepsilon) \subset \mathbb{R}^d \setminus B(\mu_1,\alpha) \leftrightarrow \Delta-\varepsilon\ge \alpha\) (the \(\varepsilon\)-disk around $\nu_1$ lies entirely outside the \(\alpha\)-ball around $\mu_1$). In this feasible case, we again observe a clean trade-off between $ f_d\leftrightarrow \alpha$ and $f_c \leftrightarrow\ve$. For a given $p_1 \leftrightarrow \mu_1, q_1 \leftrightarrow \nu_1$ and for a fixed  $f_c \leftrightarrow \ve $, the largest choice of $\alpha$ allowed for feasibility of such a $p\in \cR$ is given by $\Delta- \ve$. Further, If \(\Delta-\varepsilon<\alpha\le \Delta+\varepsilon\), then \(\mathcal{R}\) is a \textbf{lens/annular cap} (circular segment)  bounded by the circle \(\|x-\nu_1\|=\varepsilon\) and the circle \(\|x-\mu_1\|=\alpha\), and once again for a given $p_1 \leftrightarrow \mu_1, q_1 \leftrightarrow \nu_1$ the trade-off between $f_d$ and $f_c$ is immediate. See the image  \ref{fig:alpha-eps-feasible} for an illustration of this.

\begin{figure}[t]\label{alpha_e_plot_2A}
  \centering
  \includegraphics[width=0.9\linewidth]{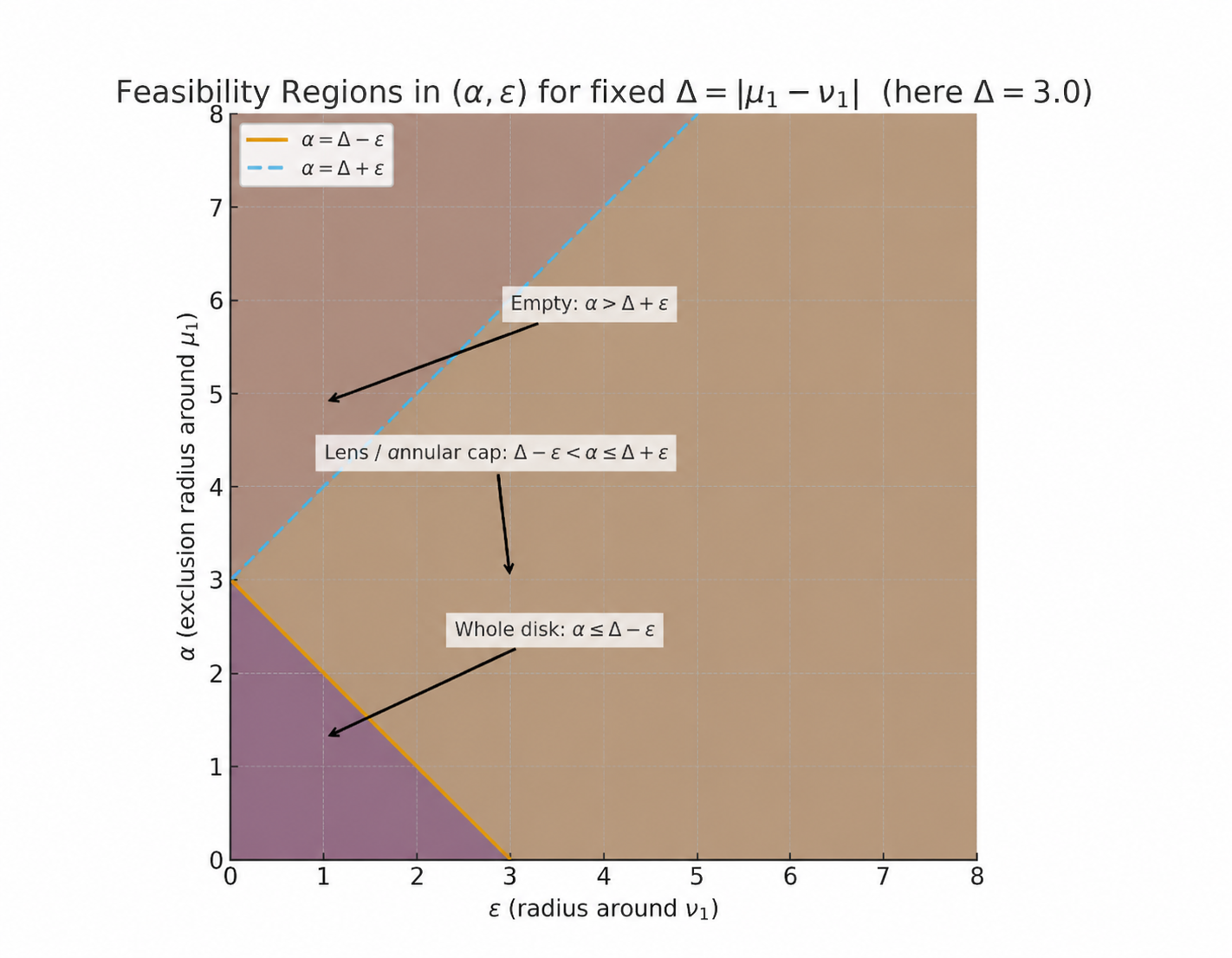}
  \caption{Feasibility regions in $(\alpha,\varepsilon)$ for fixed $\Delta=\|\mu_1-\nu_1\|$. 
  The lines $\alpha=\Delta-\varepsilon$ (solid) and $\alpha=\Delta+\varepsilon$ (dashed) partition the positive orthant of the plane into three regions based on what $\cR$ is:
  whole disk ($\alpha\le\Delta-\varepsilon$), lens/annular cap ($\Delta-\varepsilon<\alpha\le\Delta+\varepsilon$), and empty ($\alpha>\Delta+\varepsilon$).}
  \label{fig:alpha-eps-feasible}
\end{figure}
\end{proof}

\subsection{Feasible region for shifted log-concave family}
\textbf{Generalization to the one parameter location model:} Next, we show that a similar result holds more generally for any one parameter location family $\cP=\{p' \overset{d}{=} \mu+ X: \mu \in \mathbb{R},  X \sim F\}$ such that $X$ is symmetric $X \overset{d}{=} -X$ and the density of $F$ is a log-concave function. Observe that this captures the one-dimensional Laplace case, and many other family of (shifted) symmetric log-concave distributions. Moreover, notationally, $T(X,Y)= T(P,Q)$ such that $X\overset{d}{=} P$ and $Y \overset{d}{=}Q$

\begin{proposition}[Feasible region of shifted one parameter family] \label{prop:FFRLC}
For $\mu_1, \nu_1 \in \mathbb{R}$ let $p_1 \overset{d}{=} \mu_1+X , q_1 \overset{d}{=} \nu_1+X\in \mathcal{P}:= \{p' \overset{d}{=} \mu+ X: \mu \in \mathbb{R},  X \sim F\}$ be class of shifted log concave distributions with  CDF $F$. Take $f_d= T(X,\alpha+X))$ and $f_c= T(X, \ve +X)$ for $\alpha, \ve \geq 0$. Then 

\begin{equation} \label{eq:XFR}
\cR(\cP, (p_1,q_1), (f_d,f_c)) =
\cR= \{p= \mu+X: \mu \in \mathbb{R}: \alpha'(\mu) \geq \alpha, \ve'(\mu) \leq \ve\} 
\end{equation}
\begin{equation} \label{eq:XFRparameters0}
\text{ where } \alpha'= |\mu_1 -\mu|  \text{ and }\ve'= |\nu_1 -\mu|. 
\end{equation}
Let $\Delta= |\mu_1- \nu_1|$ then $\cR$ is empty $\leftrightarrow$ $ \Delta + \epsilon <\alpha $. Moreover, $\cR= [\nu_1- \ve, \nu_1+\ve]\leftrightarrow \Delta - \ve \geq \alpha $, and $\cR$ is an interval (see the proof and the plots for an explicit description of this set) otherwise. 
Moreover, the Pareto frontier of $(\alpha, \ve)$ pairs, capturing for a fixed preservation level of $\varepsilon \geq 0$, the largest level of removal $\alpha \geq 0$  that is jointly achievable or, equivalently, allows $\cR \neq \emptyset$.
\begin{equation} \label{eq:ParetoLC}
    \text{PF}(p_1, q_1, \cP) := \{(\Delta +\varepsilon, \varepsilon): \varepsilon \geq 0\} \text{ with } \Delta = |(\mu_1- \nu_1)|.
\end{equation}
\end{proposition}

\begin{proof} \label{proof:FFR}
Let $p \overset{d}{=} \mu +X \in \mathcal{P}$. From the properties of  trade-off functions \ref{lem:SLCTOF} we have  
\begin{equation}
T(p,p_1)= T(\mu+X, \mu_1+X) = T(X, |\mu_1- \mu|+ X) = f_{X,\alpha'} 
\end{equation}
\begin{equation}
    T(p,q_1)= T(\mu+X, \nu_1+X) = T(X, |\nu_1- \mu|+X) = f_{X, \ve'} 
\end{equation}
\begin{equation} \label{eq:LCFRparameters}
\text{ where } \alpha'= |\mu_1 -\mu|  \text{ and }\ve'= |\nu_1 -\mu|. \text{ Now,}
\end{equation}
\begin{equation}
    T(p,p_1) = f_{X,\alpha'}  \leq f_{X,\alpha} =f_d \leftrightarrow \alpha' \geq \alpha
\end{equation}
\begin{equation}
    T(p,q_1) = f_{X,\ve'} \geq f_{X,\ve} =f_c \leftrightarrow \ve' \leq \ve
\end{equation}
So, given the class $\cP$ of shifted symmetric log concave distributions with a common CDF $F$,  $p_1,q_1$, and $f_d,f_c$ as above, the feasible region $\cR$ has the following explicit form with $\alpha', \ve'$ as above.
\begin{equation}
\cR= \{p\overset{d}{=} \mu+X: \mu \in \mathbb{R}: \alpha'(\mu) \geq \alpha, \ve'(\mu) \leq \ve\}
\end{equation}
Now, we determine the boundary of this region $\cR$  by identifying it in the following way.
\begin{equation}
\cR= \{\mu \in \mathbb{R}: |\mu_1- \mu| \geq \alpha, |\nu_1 - \mu| \leq \ve\}= [\nu_1- \ve, \nu_1 +\ve]\ \cap  \Big((-\infty, \mu_1- \alpha] \cup  [\mu_1+ \alpha, \infty) \Big).
\end{equation}

\begin{figure}[t]
  \centering
  \includegraphics[width=0.9\linewidth,keepaspectratio]{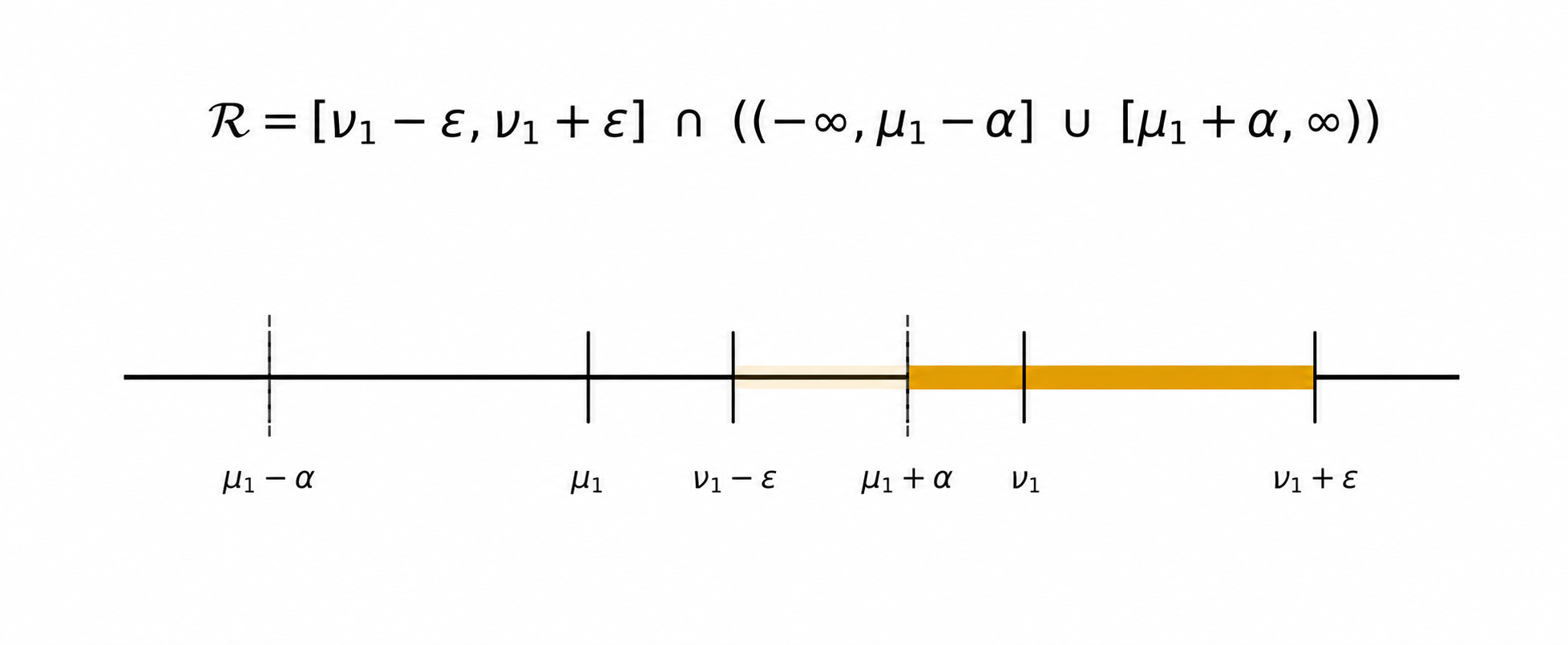}
  \caption{Feasible set on the line:
  $\mathcal R=[\nu_1-\varepsilon,\nu_1+\varepsilon]\ \cap\ \big((-\infty,\mu_1-\alpha]\ \cup\ [\mu_1+\alpha,\infty)\big)$.}
  \label{fig:alpha_e_plot_3}
\end{figure}

In the image \ref{fig:alpha_e_plot_3}, it is the orange region. Now, it is immediate that $\cR$ is an \textbf{empty} set $\leftrightarrow  [\nu_1 -\ve, \nu_1 +\ve]\subset (\mu_1-\alpha, \mu_1+\alpha)\leftrightarrow \Delta+\varepsilon<\alpha$ (the whole \(\varepsilon\)-interval of $\nu_1$ is too small and the radius $\alpha$ of the interval at $\mu_1$ is large enough). It follows that for a given $p_1 \leftrightarrow \mu_1, q_1 \leftrightarrow \nu_1$ and for a fixed  $f_c \leftrightarrow \ve $, the largest choice of $\alpha$ allowed for feasibility of such a $p\in \cR$ is given by $\Delta+ \ve$. This generalizes the concept of Pareto frontier in \cite{allouah2025distributionalmachineunlearningselective}. Moreover,  the region is the \textbf{whole closed interval:}   \(\mathcal{R}=[\nu_1- \ve, \nu_1 + \ve] \leftrightarrow [\nu_1- \ve, \nu_1 + \ve] \subset \Big((-\infty, \mu_1- \alpha] \cup  [\mu_1+ \alpha, \infty) \Big) \leftrightarrow \Delta-\varepsilon\ge \alpha\) (the \(\varepsilon\)-interval around $\nu_1$ lies entirely outside the \(\alpha\)-interval around $\mu_1$). In this feasible case, we again observe a clean trade-off between $ f_d\leftrightarrow \alpha$ and $f_c \leftrightarrow\ve$. For a given $p_1 \leftrightarrow \mu_1, q_1 \leftrightarrow \nu_1$ and for a fixed  $f_c \leftrightarrow \ve $, the largest choice of $\alpha$ allowed for feasibility of such a $p\in \cR$ is given by $\Delta- \ve$. Further, If \(\Delta-\varepsilon<\alpha\le \Delta+\varepsilon\), then \(\mathcal{R}\) is a \textbf{subinterval}  $[\mu_1+\alpha, \nu_1+ \ve]$\footnote{Here, we  assume $\nu_1\geq \mu_1$, but the other case $\mu_1\geq \nu_1$ can be written with appropriate modifications.} of $[\nu_1-\ve, \nu_1+ \ve]$, and once again for a given $p_1 \leftrightarrow \mu_1, q_1 \leftrightarrow \nu_1$ the trade-off between $f_d$ and $f_c$ is immediate. See the image  \ref{fig:alpha_e_plot_4} for an illustration of this.

\begin{figure}[t]
  \centering
  \includegraphics[width=0.9\linewidth,keepaspectratio]{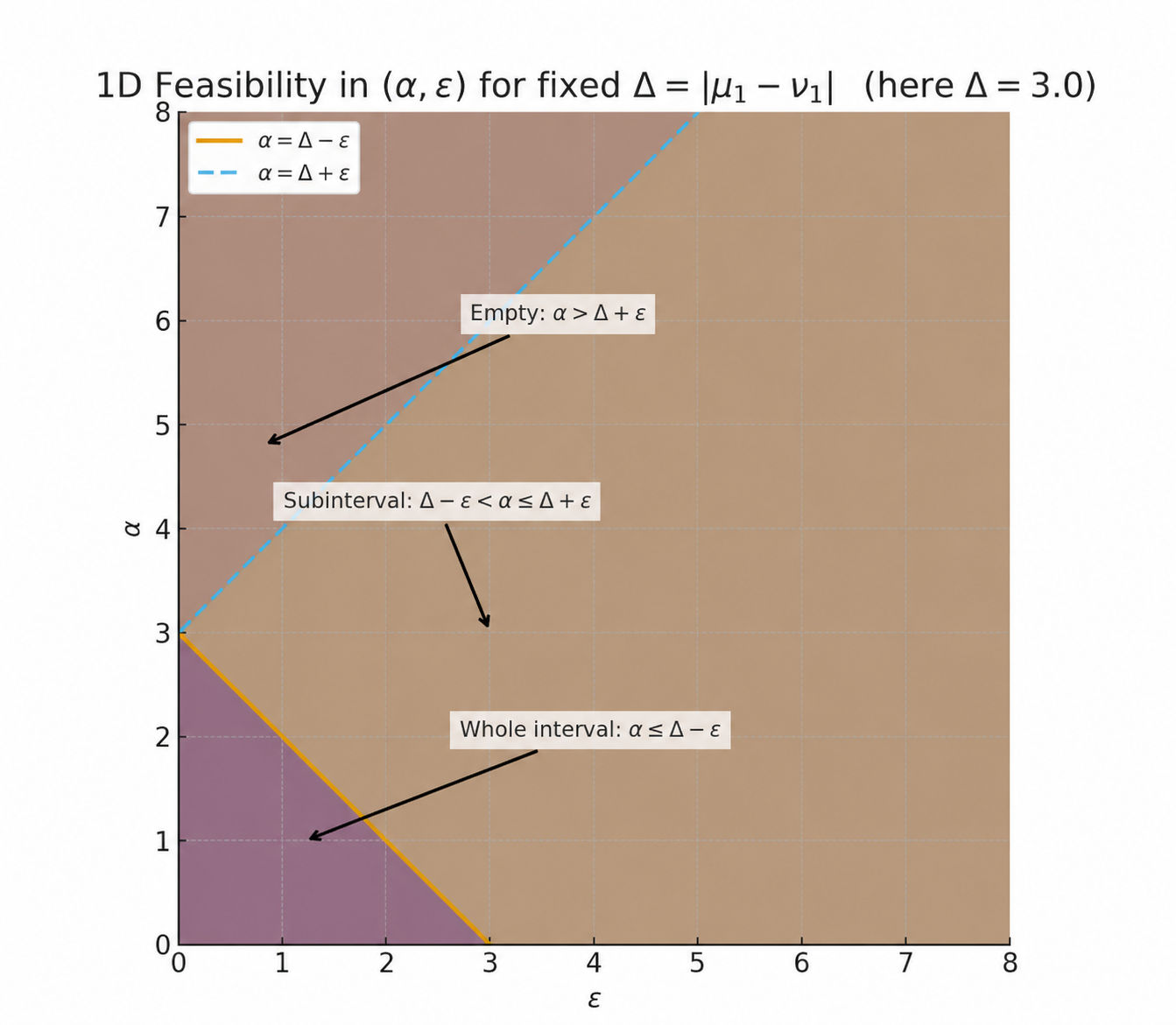}
  \caption{Feasibility regions in $(\alpha,\varepsilon)$ for fixed $\Delta=|\mu_1-\nu_1|$.
  The boundary  lines $\alpha=\Delta-\varepsilon$ (solid) and $\alpha=\Delta+\varepsilon$ (dashed) partition the positive orthant of the plane into three regions based on what $\cR$ is:
 \emph{Whole interval} ($\alpha\le\Delta-\varepsilon$), \emph{Subinterval}
  ($\Delta-\varepsilon<\alpha\le\Delta+\varepsilon$), and \emph{empty} ($\alpha>\Delta+\varepsilon$).}
  \label{fig:alpha_e_plot_4}
\end{figure}

\end{proof}
\begin{remark}
    Given the prior information about $\cP=\{\mu+X: \mu \in \R\}$ and $X\sim F$, it is  natural to compare with the baseline trade-off functions $f_d=T(X,\alpha+X),f_c=T(X, \ve+ X)$ as chosen above in Proposition \ref{prop:FFRLC}. Our unlearning framework (Definition \ref{def:dist-unlearning}) allows this flexibility of utilizing the prior information in a suitable way in comparison to the $(\alpha, \ve)$ distributional unlearning framework of \cite{allouah2025distributionalmachineunlearningselective}. Moreover, if $p$ satisfies $(f_d,f_c)$ unlearning with respect to $p_1,q_1\in \cP$, then $p$ also satisfies $(f_{d'},f_{c'})$ unlearning for TOFs $f_{d'} \geq f_d$ and $f_c \geq f_{c'}$. For $f_d=T(X, \alpha +X)$ and  $f_c=T(X, \ve+X)$, if $p$ satisfies $(f_d \leftrightarrow \alpha ,f_c \leftrightarrow \ve)$ unlearning, then as a consequence of Lemma \ref{lem:SLCTOF} it also satisfies  $(f_{d'} \leftrightarrow \alpha' ,f_{c'} \leftrightarrow \ve')$ unlearning for any $0\leq \alpha' \leq \alpha$ and $\ve' \geq \ve$.
\end{remark}
\subsection{Multiple unwanted populations and retained ones} \label{ssec:multiple}

We first generalize the definition of $(f_d,f_c)$ distributional unlearning  \ref{def:dist-unlearning} to $k$ unwanted populations $p_1, \cdots, p_k$ and $l$ retained populations $q_1, \cdots, q_l$ with their respective baseline trade-off functions.
\begin{definition}[$(f_{\mathbf{d}},f_{\mathbf{c}})$ distributional unlearning]
\label{def:dist-unlearning_multi}
For trade off functions $f_{\mathbf{d}} = \{f_{d1}, \cdots, f_{dk}\},f_{\mathbf{c}} = \{f_{c1}, \cdots,f_{cl}\} \in \cT$, a distribution $p \in \mathcal{P}$ (a class of distributions on some measurable space $(\cW, \cF_{\cW})$) satisfy $(f_{\mathbf{d}},f_{\mathbf{c}})$ unlearning with respect to unwanted populations $\mathbf{p}$ and retained populations $ \mathbf{q} $ if $\mathbf{p}= (p_1, \cdots, p_k)$ with each $p_i \in\cP$ and $\mathbf{q}= (q_1, \cdots, q_l)$ with each $q_j \in \cP$ and
\begin{align}
&
T(p, \mathbf{p}) \leq f_{\mathbf{d}}  \leftrightarrow T(p, p_i) \leq f_{di} \forall i \in [k]  \text{ (removal)},
\\
&
T(p, \mathbf{q}) \geq f_{\mathbf{c}} \leftrightarrow T(p, q_j) \geq f_{cj}  \forall j \in [l] \text{ (preservation)}.
\end{align}
\end{definition}
Now, we generalize the definition of the feasible region \ref{def:FR} for a multiple-population situation.
\begin{definition}[Feasible region]
   For families of probability distributions $\mathbf{p},\mathbf{q} \in \cP$ as above  and TOFs $f_{\mathbf{d}},f_{\mathbf{c}}$ we define the feasible region  $\cR=\{p \in \cP: T(p, \mathbf{p}) \leq f_{\mathbf{d}}, \text{ and } T(p, \mathbf{q}) \geq f_{\mathbf{c}}\}$.
\end{definition}
Given the definitions above, we now describe the feasible regions for a shifted Gaussian family in any dimensions and then for a shifted one-parameter family with symmetric log-concave noise.
\begin{restatable}[Feasible region of shifted Gaussians]{proposition}{FFR}
\label{thm:FFRkl}
Let $p_i = N(\mu_i, \Sigma), q_j= N(\nu_j, \Sigma)\in \mathcal{P}:= \{N(\mu, \Sigma): \mu \in \mathbb{R}^d\}$ be class of shifted Gaussian distributions with common covariance $\Sigma$. Take $f_{di}= T(N(0,1), N(\alpha_i, 1))$ and $f_{cj}= T(N(0,1), N(\ve_j,1))$\footnote{By dimension invariance of Gaussian trade off functions one dimensional $f_d,f_c$ are not restrictive.} for $\alpha_i, \ve_j \geq 0$. Then  

\begin{equation} \label{eq:GFRkl}
\cR(\cP, (\mathbf{p},\mathbf{q}), (f_\mathbf{d},f_\mathbf{c})) =
\cR= \{p= N(\mu, \Sigma): \mu \in \mathbb{R}^d: \alpha_i'(\mu) \geq \alpha_i, \ve_j'(\mu) \leq \ve_j \forall i,j\} 
\end{equation}
\begin{equation} \label{eq:GFRparameters0kl}
\text{ where } \alpha_i'^2= \langle \mu_i -\mu, \Sigma^{-1}(\mu_i-\mu) \rangle \text{ and }\ve_j'^2= \langle \nu_j -\mu, \Sigma^{-1}(\nu_j-\mu) \rangle. 
\end{equation}
 Assume $\Sigma=\II_d$ then $\cR = \bigcap_{j=1}^l\overline{B}(\nu_j, \ve_j) \bigcap_{i=1}^k B(\mu_i, \alpha_i)^c$. Moreover,  if dimension $d=1$ then $\cR= [\max_{j} (\nu_j -\ve_j), \min_{j} (\nu_j + \ve_j)] \bigcap \left((-\infty, \min_{i}(\mu_i -\alpha_i)] \cup [ \max_{i} (\mu_i+\alpha_i), \infty)\right)$\footnote{For $a<b$  emptiness of a set $[L,R] \cap \left((-\infty, a] \cup [b, \infty)\right)$ is equivalent to $L>R$ or $[L,R] \subset (a,b)$.}.  is empty $\leftrightarrow$ $ \max_{j} (\nu_j -\ve_j) > \min_{j} (\nu_j + \ve_j)$ or $[\max_{j} (\nu_j -\ve_j), \min_{j} (\nu_j + \ve_j)] \subset (\min_{i}(\mu_i -\alpha_i), \max_{i} (\mu_i+\alpha_i))$.
\end{restatable}
\begin{proof}
    The proof is exactly the same as in the case of $k=l=1$ \ref{proof:FFRG}, and therefore we skip it.
\end{proof}
\begin{remark}
   Determining the non-emptiness of the non-convex set $\cR$ in higher dimensions is computationally hard \cite{MurtyKabadi1987}. It would be interesting to develop algorithms for this.
\end{remark}
\begin{proposition}[Feasible region of shifted one parameter family] \label{prop:FFRLCk}
For $\mu = (\mu_1, \cdots, \mu_k), \nu= (\nu_1, \cdots, \nu_l)$ with $ \mu_i, \nu_j\in \mathbb{R}$ consider  $p_i \overset{d}{=} \mu_i+X , q_j \overset{d}{=} \nu_j+X\in \mathcal{P}:= \{p' \overset{d}{=} \mu+ X: \mu \in \mathbb{R},  X \sim F\}$ be class of shifted log concave distributions with common CDF $F$. Take the baseline trade off functions $f_{di}= T(X,\alpha_i+X))$ and $f_c= T(X, \ve_j +X)$\footnote{Given the prior information about $\cP \leftrightarrow F$, it is indeed appropriate to compare with the baseline trade-off functions $f_{\mathbf{d}},f_{\mathbf{c}}$ as chosen above. It is the flexibility of our framework of utilizing the prior information in an \textit{optimal} way in comparison to the $(\alpha, \ve)$ distributional unlearning framework of \cite{allouah2025distributionalmachineunlearningselective}.} for $\alpha_i, \ve_j \geq 0$. Then 

\begin{equation} \label{eq:XFRk}
\cR(\cP, (\mathbf{p},\mathbf{q}), (f_{\mathbf{d}},f_{\mathbf{c}})) =
\cR= \{p= \mu+X: \mu \in \mathbb{R}: \alpha_i'(\mu) \geq \alpha_i, \ve_j'(\mu) \leq \ve_j\} 
\end{equation}
\begin{equation} \label{eq:XFRkparameters0}
\text{ where } \alpha_i'= |\mu_i -\mu|  \text{ and }\ve_j'= |\nu_j -\mu|. \text{ Moreover,}
\end{equation}
  $\cR= [\max_{j} (\nu_j -\ve_j), \min_{j} (\nu_j + \ve_j)] \bigcap \left((-\infty, \min_{i}(\mu_i -\alpha_i)] \cup [ \max_{i} (\mu_i+\alpha_i), \infty)\right)$ is empty $\leftrightarrow$ $ \max_{j} (\nu_j -\ve_j) > \min_{j} (\nu_j + \ve_j)$ or $[\max_{j} (\nu_j -\ve_j), \min_{j} (\nu_j + \ve_j)] \subset (\min_{i}(\mu_i -\alpha_i), \max_{i} (\mu_i+\alpha_i))$.
\end{proposition}

\begin{proof}
    The proof is exactly the same as in the case of $k=l=1$ \ref{proof:FFR}, and therefore we skip it.
\end{proof}

\subsection{Feasible region for 
shifted Gaussian measures on a Hilbert space} \label{ssec:GaussiansonHS}

We now describe the Gaussian shift experiment in a Hilbert-space formulation. This includes the finite-dimensional
Gaussian models described earlier \ref{thm:FFR} and Gaussian white noise described later (to be precise, this is true only after embedding the observation in a sufficiently large Hilbert space).

Let \(\mathsf H\) be a real separable Hilbert space with inner product $\langle \cdot,\cdot\rangle_{\mathsf H}$ and norm $\|x\|_{\mathsf H} = \sqrt{\langle x,x\rangle_{\mathsf H}}.$ Consider a centered Gaussian probability measure $P_0$ on $(\mathsf H, \cB(\mathsf H))$ \cite{gine2016mathematical, Bogachev1998GaussianMeasures}. Equivalently, if \(X\sim P_0\), then for every \(u\in\mathsf H\), $\langle X,u\rangle_{\mathsf H}$ is a centered real Gaussian random variable\footnote{This is also an equivalent definition of a centered Gaussian distribution in finite dimensions by requiring that for every \(u\in\R^d\), $\langle X,u\rangle$ is a centered real Gaussian random variable on $\R$ \citet{krishnapur2025gaussian}.}. The covariance operator of \(P_0\) is the unique non-negative, self-adjoint, trace-class operator $K:\mathsf H\to\mathsf H $
such that  $\mathbb E_{P_0}
    \bigl[
        \langle X,u\rangle_{\mathsf H}
        \langle X,v\rangle_{\mathsf H}
    \bigr]
    =
    \langle Ku,v\rangle_{\mathsf H}$  for $u,v\in\mathsf H$ \citet{Bogachev1998GaussianMeasures}[Theorem 2.3.1]. We denote $P_0=N(0,K).$ Since \(K\) is  non-negative  and self-adjoint, it has a unique  non-negative square root \(K^{1/2}\). We define the Cameron--Martin space associated with
\(P_0\)  as  $\mathcal H_K :=\operatorname{Range}(K^{1/2}) = (\ker K)^\perp$ \citet{lunardi2016infinite}[Theorem 4.2.7] with the inner product 
\[
    \langle h_1,h_2\rangle_{\mathcal H_K}
    :=
    \left\langle
        K^{-1/2}h_1,
        K^{-1/2}h_2
    \right\rangle_{\mathsf H},
   \text{ for }h_1,h_2\in\mathcal H_K,
\]
where \(K^{-1/2}h\) denotes the unique element in
\(\overline{\operatorname{Range}(K^{1/2})}\) mapped to \(h\) by
\(K^{1/2}\). The corresponding norm is $\|h\|_{\mathcal H_K} = \|K^{-1/2}h\|_{\mathsf H}.$ An equivalent way to describe  $\mathcal H_K$ is to consider an orthonormal basis  \(\{e_j\}_{j\ge 1}\)  of \(\mathsf H\) so that $K e_j=\lambda_j e_j,$ for $ j\ge 1,$  with 
\(\lambda_j\ge 0\), and \(\sum_j\lambda_j<\infty\), then
\begin{equation}
    \mathcal H_K
    =
    \left\{
        h=\sum_{j\ge 1} h_j e_j:
       \|h\|_{\mathcal H_K}^2:=  \sum_{j:\lambda_j>0}\frac{h_j^2}{\lambda_j}<\infty,
        h_j=0 \text{ whenever } \lambda_j=0
        \right\}.
\end{equation}

For \(h\in\mathcal H_K\), define \(P_h\) to be the law of $X+h$ where  $X\sim P_0$, denoted as $P_h=N(h,K)$. Now, the Gaussian shift family on $\cH_K$ is a collection of shifted Gaussian measures on $(\mathsf H, \cB(\mathsf H))$
\[
    \cP
    =
    \{P_h:h\in\mathcal H_K\}.
\]

Cameron-Martin theorem says for all \(h\in\mathcal H_K\) we have  $P_h\ll P_0$ with  Radon--Nikodym derivative
\[
    \frac{dP_h}{dP_0}(x)
    =
    \exp\left\{
        \langle K^{-1}h,x\rangle_{\mathsf H}
        -
        \frac12
        \langle h,h\rangle_{\cH_{K}}
    \right\}, \frac{dP_h}{dP_0} : (\mathsf H , \cB(\mathsf H)) \to \mathbb{R}_{\geq 0}.
\]

Now, for \(h,h'\in\mathcal H_K\), we want the likelihood ratio of $P_{h'}$ with respect to $P_h$ for computing the TOF $T(P_h, P_{h'})$. Let
$v=h'-h.$ Since \(\mathcal H_K\) is a vector space, \(v\in\mathcal H_K\). The log-likelihood ratio is
\[
    \log \frac{dP_{h'}}{dP_h}(x)
    =
    \langle K^{-1}v,x-h\rangle_{\mathsf H}- \frac12
        \Delta^2 , \text{ where } \Delta =\|v\|_{\mathcal H_K}
\]
Under \(P_h\), we have   \(X-h \overset{d}{=}P_0=N(0,K)\).
So, $\langle K^{-1/2} v, K^{-1/2}(X-h)\rangle_{\mathsf H}  \overset{d}{=}
    N\bigl(0,\|v\|_{\mathcal H_K}^2\bigr).$

\[
   \text{Consequently, } \log\frac{dP_{h'}}{dP_h}(X)
    \sim
    N\left(
        -\frac12\|v\|_{\mathcal H_K}^2,
        \|v\|_{\mathcal H_K}^2
    \right)
    \qquad\text{under }P_h.
\]

\[
    \text{Similarly, }\log\frac{dP_{h'}}{dP_h}(X)
    \sim
    N\left(
        \frac12\|v\|_{\mathcal H_K}^2,
        \|v\|_{\mathcal H_K}^2
    \right)
    \qquad\text{under }P_{h'}.
\]
From \citet{pandey2025infinitely}[Lemma 4]  $T(P,Q)= T\left(P \circ \log \left(\frac{dQ}{dP}\right)^{-1}, Q \circ \log \left( \frac{dQ}{dP}\right)^{-1}\right)$.
\[
   \text{So, } T(P_h,P_{h'})(\alpha)
    = 
    T\left(N\left(-\frac{\Delta^2}{2}, \Delta^2\right), N\left(\frac{\Delta^2}{2}, \Delta^2\right)\right)(\alpha)
    =
    T\bigl(N(0,1),N(\Delta,1)\bigr)(\alpha).
\]

Our \textbf{next} lemma summarizes the results proven above for shifted Gaussians on a Hilbert space.

\begin{lemma} \label{lem: TOFH}
    Let $K:\mathsf H \to \mathsf H$ be a non-negative, self-adjoint, trace-class operator on a real separable Hilbert space $(\mathsf H, \langle \cdot, \cdot \rangle_{\mathsf H})$ with its associated Cameron-Martin space $\cH_K:= \text{Range} \left(K^{1/2}\right)$ with $\langle \cdot, \cdot\rangle_{\cH_{K}}= \langle K^{-1/2} \cdot, K^{-1/2} \cdot \rangle_{\mathsf H} $. Consider the associated shifted Gaussian family of measures $\cP:= \{N(h, K): h \in \cH_K\}$ on $(\mathsf H, \cB(\mathsf H))$. Then for $h,h'\in \cH_{K}$, we have the trade-off functions $T(P_h,P_{h'})=  T\bigl(N(0,1),N(\Delta,1)\bigr)$ with $\Delta^2:=   \|v\|_{\mathcal H_K}^2= \langle K^{-1/2}v, K^{-1/2} v \rangle_{\mathsf H}  $ and $v=h'-h$.
\end{lemma}
Our \textbf{next} lemma extends the finite dimensional Gaussian result \ref{thm:FFR} with covariance matrix $\Sigma \succ 0$ replaced by an appropriate non-negative, self-adjoint, trace-class covariance operator $K$, and instead of parameterizing the family over all the elements $h\in \mathsf H$, we parametrize over the (restricted) Cameron-Martin subspace $\cH_K$ as otherwise for $h \notin \cH_K$,  the probability measure $P_h= N(h,K)$ is  singular to the centered Gaussian measure $P_0=N(0,K)$ \citet{lunardi2016infinite}[Theorem 3.1.5]. 
\begin{restatable}[Feasible region of shifted Gaussians on a Hilbert space $\mathsf H$]{proposition}{FRSGH}
\label{thm:FRSGH}
Under the conditions of lemma \ref{lem: TOFH} let $p_1 = N(\mu_1, K), q_1= N(\nu_1, K)\in \mathcal{P}:= \{N(\mu, \Sigma): \mu \in \cH_{K}\}$ be class of shifted Gaussian measures on a Hilbert space $(\mathsf H, \cB(\mathsf H))$ with common covariance $K:\mathsf H \to \mathsf H$. Take $f_d= T(N(0,1), N(\alpha, 1))$ and $f_c= T(N(0,1), N(\ve,1))$\footnote{By dimension invariance of Gaussian trade off functions one dimensional $f_d,f_c$ are not restrictive.} for $\alpha, \ve \geq 0$. Then  
\vspace{-2mm}
\begin{equation} \label{eq:GFRSGH}
\cR(\cP, (p_1,q_1), (f_d,f_c)) =
\cR= \{p= N(\mu, K): \mu \in \cH_{K}: \alpha'(\mu) \geq \alpha, \ve'(\mu) \leq \ve\} 
\end{equation}
\begin{equation} \label{eq:GFRparameters0SGH}
\text{ where } \alpha'^2= \langle K^{-1/2}(\mu_1 -\mu), K^{-1/2}(\mu_1-\mu) \rangle \text{ and }\ve'^2= \langle K^{-1/2}(\nu_1 -\mu), K^{-1/2}(\nu_1-\mu) \rangle. 
\end{equation}
Let $\Delta= \|K^{-1/2}(\mu_1- \nu_1)\|$, then $\cR$ is empty $\leftrightarrow$ $ \Delta + \epsilon <\alpha $. Moreover, $\cR= \overline{B}(\nu_1, \ve) \leftrightarrow \Delta - \ve \geq \alpha $, and $\cR$ is annular cap otherwise.  Moreover, the Pareto frontier matches as well.
\begin{equation} \label{eq:ParetoSGH}
    \text{PF}(p_1, q_1, \cP) := \{(\Delta +\varepsilon, \varepsilon): \varepsilon \geq 0\} \text{ with } \Delta = ||K^{-1/2}(\mu_1- \nu_1)||_{\mathsf H}.
\end{equation}
\end{restatable}

\begin{proof}
    The proof is exactly the same as that of \ref{thm:FFR}, once we observe that from lemma \ref{lem: TOFH} that for $\mu_1, \nu_1 \in \cH_{K}$ we have $T(P_{\mu_1}, P_{\nu_1}) = T(N(0,1), N(\Delta,1))$ with  $\Delta= \|K^{-1/2}(\mu_1- \nu_1)\|$.
\end{proof}

\begin{remark}[The problem with different covariance operators] In infinite dimensional separable Hilbert spaces $\mathsf H$, such as $\ell^2(\mathbb{N}):=\{(a_n)_{n\geq 0}: a_n\in \R, \sum_{n\geq 0} a_n^2 < \infty\}$, in many generic situations with two different non-negative, self-adjoint, trace-class operators $K,K' : \mathsf H \to \mathsf H$, we have $P_0= N(0,K)$ and $P_{0'}= N(0,K')$ are singular measures with respect to each other \cite{Bogachev1998GaussianMeasures}, \citet{krishnapur2025gaussian}[Theorem 41]. Therefore, the trade-off curve $T(P_0,P_{0'}) \equiv  0$.     
\end{remark}

\begin{table}[t]
\centering
\tiny
\setlength{\tabcolsep}{2.5pt}
\renewcommand{\arraystretch}{1.12}
\begin{tabularx}{\textwidth}{
@{}p{0.105\textwidth}
p{0.275\textwidth}
p{0.30\textwidth}
X@{}}
\toprule
\(\mathsf H\)
& Covariance operator \(K\)
& Cameron--Martin space \(\mathcal H_K\)
& Gaussian measure or process \\
\midrule

\((\mathbb R^d, \langle \cdot, \cdot \rangle)\)
&
\(K=\Sigma\succ0\)
&
\(\R^d\), 
\(\|h\|_{K}^2= \langle \Sigma^{-1/2} h, \Sigma^{-1/2} h \rangle \)
&
Finite-dimensional Gaussian \(N(0,\Sigma)\)
\\

\addlinespace[1.5pt]

\(L^2(Y,\cF,\rho)\)
&
\((Kf)(x)=\int k(x,y)f(y)\,d\rho(y)\),
\(k\succeq0\) and \(\int k(x,x)d\rho(x)<\infty\).
&
If \(k(x,y)=\sum_j\lambda_j e_j(x)e_j(y)\), then $\mathcal H_K$
\(=\{h=\sum_jh_je_j:\sum_{\lambda_j>0}\frac{h_j^2}{\lambda_j}<\infty\}\).
&
Centered Gaussian process on $(Y, \cF)$ 
\\

\addlinespace[1.5pt]

\(L^2([0,1])\)
&
\((Kf)(t)=\int_0^1(s\wedge t)f(s)\,ds\).
&
\(\mathcal H_K=\{h: h(t)= \int_{0}^{t}h'(s)ds\, h, h'\in L^2[0,1]\}\),
\(\|h\|_{\mathcal H_K}^2=\int_0^1|h'(t)|^2dt\).
&
Standard Brownian motion \(B_t\), viewed as an \(L^2[0,1]\)-valued  random element.
\\

\addlinespace[1.5pt]

\(\ell^2(\mathbb N)\)
&
\(K e_j=\lambda_j e_j\), \(\lambda_j >0\), \(\sum_j\lambda_j<\infty\).
&
\(\{h=(h_j)_{j\geq 0}:\ \sum_{j\geq 0} \frac{h_j^2}{\lambda_j}<\infty\}\).
&
\(X_j=\sqrt{\lambda_j}\xi_j\), \(\xi_j\stackrel{iid}{\sim}N(0,1)\).
\\

\bottomrule
\end{tabularx}
\caption{Covariance operators \(K: \mathsf H \to \mathsf H\),  Cameron--Martin spaces $\cH_{K}$, and  Gaussian measures.}
\label{tab:gaussian-covariance-examples}
\end{table}

\subsection{Feasible region for Poisson family}
\FRP*
\begin{proof} \label{proof:FRP}

$\mathcal R
:=
\left\{
\mu>0:
T(P(\mu),P(\mu_1))\le T(P(1),P(\alpha))
\ \text{and}\
T(P(\mu),P(\nu_1))\ge T(P(1),P(\varepsilon))
\right\}.$

\textbf{Proof of $\cR \supseteq\cR_P$:} First, we prove that the feasible region $\cR$ contains $\cR_P$. More precisely, for $\alpha > 1, \varepsilon< 1$ if $\mu$ satisfy the conditions of \ref{eq:FRP}, then we have removal bound $T(P(\mu), P(\mu_1)) \leq T(P(1), P(\alpha))$ as well as the preservation bound $T(P(\mu), P(\nu_1)) \geq T(P(1),P(\varepsilon))$.  

Now, recall the Poisson thinning kernel $K_s$ that takes a sample from $P(\theta)$ and outputs a sample from $P(s\theta)$ for some $s\in [0,1]$. One can implement this by first considering the output  $X\overset{d}{=} P(\theta)\in \mathbb{Z}_{\geq 0}$, and draw  a sample from  $B(X,s)$\footnote{$B(X,s)$ is the binomial distribution with number of trials $X$ and with success probability  s.}. The marginal distribution of  $B(X,s)$ is given by $P(s\theta)$\footnote{See \citet{pandey2025infinitely}[Lemma 21] for a complete description of the subsampling kernel $K_s$.}. Moreover, recall the Poisson superposition kernel $K^{r}$ that takes a sample from $P(\theta)$ and outputs a sample from $P(\theta +r)$ for some $r\geq 0$. One can implement this by first considering the output  $X\overset{d}{=} P(\theta)\in \mathbb{Z}_{\geq 0}$, and draw  an independent  sample from  $Z \overset{d}{=}P(r)$ and output  $X+Z$. The marginal distribution of  $X+Z$ is given by $P(\theta +r)$\footnote{See \citet{pandey2025infinitely}[Lemma 21] for a complete description of the superposition kernel $K^r$.}.  A composition of these two Markov kernels would yield $K^r \circ K_s \circ (P(\theta)) =K^r \circ P(s\theta)= P(s\theta +r)$ for $s\in [0,1], r\geq 0$. Equivalently
\[
\text{ If }
X\sim P(\theta)
\text{ then }
Y= B(X,s) + Z
\overset{d}{=}
P(s\theta+r), 
\text{ where }
X \perp Z \overset{d}{=} P(r).
\]
Therefore, if there exist \(s\in[0,1]\) and \(r\ge0\) such that with $a, b>0$ $c=s a+r, d=s b+r,$
then
\[
T(P(a),P(b))\le T(P(sa), P(sb)) \leq T(P(sa+r), P(sb+r)) = T(P(c),P(d)).
\]
The above inequalities follow from the definition of $T(P,Q)$ \ref{def:TOF} or Blackwell's theorem \ref{thm:Blackwell} that says that $T(P,Q) \leq T(K\circ P, K \circ Q)$ for any two 
probability distributions $P$ and $Q$ and Markov kernel $K$.
Now, we will construct Markov kernels of the kind $K^r\circ K^s$ so that we have the following
\[
P(\mu),P(\mu_1)
\overset{K^{r_1}\circ K_{s_1}}{\longmapsto}
P(1),P(\alpha), \text{ and also }
P(1),P(\varepsilon)
\overset{K^{r_2}\circ K_{s_2}}{\longmapsto}
P(\mu),P(\nu_1).
\]

\noindent\textbf{First.}
We require \(s_1\in[0,1]\), \(r_1\ge0\) so that $1=s_1\mu+r_1,
\alpha=s_1\mu_1+r_1.$ So, when $\mu_1 \neq \mu$
\begin{equation}
0\leq s_1=\frac{\alpha-1}{\mu_1-\mu} \leq 1,
\qquad
r_1=1-\mu\frac{\alpha-1}{\mu_1-\mu} \geq 0.
\end{equation}

\textbf{Second.} We require \(s_2\in[0,1]\), \(r_2\ge0\) so that $\mu=s_2+r_2,
\nu_1=s_2\varepsilon+r_2.$ So, when \(\varepsilon\neq1\),
\begin{equation}
0\leq s_2=\frac{\nu_1-\mu}{\varepsilon-1} \leq 1,
\qquad
r_2=\mu-\frac{\nu_1-\mu}{\varepsilon-1} \geq 0.
\end{equation}
\medskip

Therefore the thinning--superposition construction yields the subset $\mathcal R_{P} \subseteq \mathcal R $ where
\[
\mathcal R_{P}
:=
\left\{
\mu>0:
0\le \frac{\alpha-1}{\mu_1-\mu}\le1,
1-\mu\frac{\alpha-1}{\mu_1-\mu}\ge0,
0\le \frac{\nu_1-\mu}{\varepsilon-1}\le1,
\mu-\frac{\nu_1-\mu}{\varepsilon-1}\ge0
\right\}.
\]

\begin{align} \label{eq:R_P}
\mathcal R_{P}(\alpha, \varepsilon, \mu_1, \nu_1)
&=
\Bigl\{
\mu>0:
\exists\, s_1,s_2\in[0,1],\ r_1,r_2\ge0 \ \text{ such that we have } \\
&\qquad
1 = s_1\mu + r_1,
\alpha = s_1\mu_1 + r_1, 
\mu = s_2 + r_2,
\nu_1 = s_2\varepsilon + r_2
\Bigr\}.
\end{align}
Now, in remark \ref{rem:Poissonregion} and \ref{table:Poissonregion} we determine the representation of the region $\cR_{P}$ as mentioned in \ref{eq:FRP}.

\textbf{Proof of $\cR \subseteq \cR_P$:} For $\alpha>1, \varepsilon <1$ and $\mu_1,\nu_1 >0$, the statement that $T(P(\mu),P(\mu_1))\le T(P(1),P(\alpha))
\ \text{and}\
T(P(\mu),P(\nu_1))\ge T(P(1),P(\varepsilon))$ implies $\mu \in \cR_{P}$ is inspired from \citet[Complement 17, Chapter 10]{torgersen1991comparison}.\footnote{We refer to \citet{torgersen1991comparison} for more connections to statistics and Information theory} More generally, we have the following characterization of when exactly do we have $f:= T(P(a), P(b)) \leq g:=T(P(c), P(d))$ for constants $a,b,c,d >0$. 
\begin{lemma}[Poisson characterization lemma] \label{lem:Poisson-experiments}
    Consider $a,b,c,d>0$. Then  we have 
    \begin{enumerate}
        \item $f:= T(P(a), P(b)) \leq g:=T(P(c), P(d))$ if and only if 
        \label{condition1}
        \item $\exists s \in [0,1], r\geq 0$ such that $c= sa+r$ and $d=sb+r$ if and only if  \label{condition2}
        \item  \label{condition3}
        $(a,b)$ and $(c,d)$ are similarly ordered\footnote{It means either $a\leq b$ and $c\leq d$  or $a\geq b$ and $c\geq d$.} and they satisfy 
        \begin{equation}  \label{eq:condition3}
        |c-d|\leq |a-b|  \text{ and }
        \frac{\max(a,b)}{\min(a,b)} \geq \frac{\max(c,d)}{\min(c,d)}.
        \end{equation}
    \end{enumerate}
\end{lemma}
Using the lemma, and applying condition \ref{condition2} for removal $T(P(\mu),P(\mu_1))\le T(P(1),P(\alpha))$ \text{and} preservation bounds   $T(P(\mu),P(\nu_1))\ge T(P(1),P(\varepsilon))$ implies $\mu \in \cR_{P}$.
\end{proof} 
\begin{proof}[Proof of characterization lemma \ref{lem:Poisson-experiments}]
    First, with some algebra, one can show that conditions \ref{condition2} and \ref{condition3} are equivalent\footnote{\ref{condition2} immediately implies \ref{condition3}. For the converse one can take $s= \frac{|c-d|}{|a-b|} \in [0,1]$ and $r= d- \frac{|c-d|}{|a-b|} b \geq 0$.}. So, it remains to prove that condition \ref{condition1} implies condition \ref{condition3}\footnote{We already showed that condition \ref{condition2} shows condition \ref{condition1} using the Kernels $K^r\circ K^s$ in the proof of \ref{proof:FRP}.}. We now give a proof of this.  First, one observes that as a consequence of Lemma \ref{lem:ordering} we have  if $f=T(P(a),P(b))\leq g=T(P(c), P(d))$, then $(a,b)$ and $(c,d)$ are similarly ordered and $|c-d|\leq |a-b|$. It remains to show the max min condition in \ref{eq:condition3}.
    Now, we compute the Hellinger transform 
    \begin{equation}
    H_t(P,Q)= \int dP^{1-t} dQ^t \text{ for } t \in \R \text{ and } Q \ll P.
    \end{equation}
    An explicit computation of the Hellinger transforms for the Poisson distributions yields for $a,b>0$
    \begin{equation}
        H_t(P(a), P(b)) = \exp \{ -(1-t)a -tb +b^ta^{1-t}\} \text{ for } t\in \R. 
    \end{equation}
    A consequence of Blackwell's theorem \ref{thm:Blackwell} and data processing inequality for Hellinger transforms \ref{lem:Hellinger} shows that $T(P(a), P(b)) \leq T(P(c), P(d))$ implies 
    \begin{equation}
        \exp \{ -(1-t)a -tb +b^ta^{1-t}\} \leq \exp \{ -(1-t)c -td +d^tc^{1-t}\} \text{ for } 0\leq t\leq 1
    \end{equation}
    \begin{equation}
        \exp \{ -(1-t)c -td +d^tc^{1-t}\} \leq \exp \{ -(1-t)a -tb +b^ta^{1-t}\} \text{ for } t < 0 \text{ or } t > 1.
    \end{equation}
    An equivalent description of the inequalities gives rise to the following with $r= \frac{d}{c}$ and $s= \frac{b}{a}$.
    \begin{equation}
      (1-t)(a-c +t(b-d) + c r^t  - a s^t \geq    0 \text{ for } 0\leq t\leq 1
    \end{equation}
    \begin{equation} \label{eq:ttoinfty}
         (1-t)(c-a) +t(d-b) + a s^t - c r^t \geq 0  \text{ for } t < 0 \text{ or } t > 1.
    \end{equation}

    Consider $0<c<d$ and $0<a<b$. Then proving the max min condition in \ref{eq:condition3} is equivalent to proving $s =\frac{b}{a} \geq \frac{d}{c} =r$. If not, then assuming  $s <r$ in equation \ref{eq:ttoinfty} and once we let $t \uparrow \infty$ we have  
    \begin{equation}
      0\leq   (1-t)(c-a) +t(d-b) + a s^t - c r^t \to - \infty  \text{ as } t \uparrow \infty, \text{ a contradiction}.
    \end{equation}
   If  $0<d<c$ and $0<b<a$, one can show the max min condition \ref{eq:condition3} by letting $t\downarrow -\infty$  in \ref{eq:ttoinfty}.
\end{proof}
\begin{lemma}(Supports of likelihood ratios) \label{lem:ordering}
Let $f(u):=T(P(a),P(b))(u)$ for $0\le u\le 1$ and  \(a,b>0\) with \(a\neq b\). Then
\[
\begin{aligned}
\lim_{u\uparrow 1}\frac{f(u)}{1-u}
&=
\begin{cases}
e^{-(b-a)}, & 0<a<b,\\
0, & a>b>0,
\end{cases}
\qquad
\lim_{u\downarrow 0}\frac{1-f(u)}{u}
=
\begin{cases}
+\infty, & 0<a<b,\\
e^{a-b}, & a>b>0.
\end{cases}
\end{aligned}
\]
\end{lemma}

\begin{proof}
For \(P=P(a)\) and \(Q=P(b)\),  the likelihood ratio is given by 
\[
\frac{dQ}{dP}(k)
=
\frac{e^{-b}b^k/k!}{e^{-a}a^k/k!}
=
e^{-(b-a)}
\left(\frac ba\right)^k,
\qquad k\in\mathbb{Z}_{\geq 0} .
\]
From \cite{torgersen1991comparison}[Complement 15, Chapter 10]  the trade-off function $f=T(P,Q)$  satisfies
\[
\lim_{u\uparrow 1}\frac{T(P,Q)(u)}{1-u}
=
\operatorname*{ess\,inf}_{P}\frac{dQ}{dP} := \sup\{z: P\left(z: z\leq \frac{dQ}{dP}\right)=1 \}, \text{ and}
\]
\[
\lim_{u\downarrow 0}\frac{1-T(P,Q)(u)}{u}
=
\operatorname*{ess\,sup}_{P}\frac{dQ}{dP} := \inf\{z: P\left(z: z\geq \frac{dQ}{dP}\right)=1 \}.
\]

Since both \(P(a)\) and \(P(b)\) have full support on \( \mathbb{Z}_{\geq 0}\), the essential infimum and the essential supremum are simply the infimum and supremum over \(k\in \mathbb{Z}_{\geq 0}\) of the likelihood ratio.

If \(0<a<b\), then \(b/a>1\). So, we have 
\[
\lim_{u\uparrow 1}\frac{f(u)}{1-u}= \operatorname*{ess\,inf}_{P(a)}
\frac{dP(b)}{dP(a)}
=
e^{-(b-a)},
\text{ and }
\lim_{u\downarrow 0}\frac{1-f(u)}{u}
=
\operatorname*{ess\,sup}_{P(a)}
\frac{dP(b)}{dP(a)}
=
+\infty.
\]

If \(a>b>0\), then \(b/a<1\). So, we have

\[
\lim_{u\downarrow 0}\frac{1-f(u)}{u}
=
\operatorname*{ess\,sup}_{P(a)}
\frac{dP(b)}{dP(a)}
=
e^{a-b},
\text{ and }
\lim_{u\uparrow 1}\frac{f(u)}{1-u}
=
\operatorname*{ess\,inf}_{P(a)}
\frac{dP(b)}{dP(a)}
=
0.
\]
\end{proof}
\begin{remark}
    It is immediate from the above lemma  that if $T(P(a), P(b))=f \leq g= T(P(c), P(d))$, then $(a,b)$ and $(c,d)$ are similarly ordered and $|c-d|\leq |a-b|$.
 \end{remark}

\begin{lemma}[Data processing and reverse data processing for Hellinger transforms] \label{lem:Hellinger}
Let \(P,Q\) be probability measures on a measurable space \((\mathcal X,\mathcal A)\) so that $Q\ll P$, and let
\(R\) be a Markov kernel from \((\mathcal X,\mathcal A)\) to another measurable space
\((\mathcal Y,\mathcal B)\). For \(t\in\mathbb R\), define
\[
H_t(P,Q):=\int dP^{1-t}\,dQ^t.
\]
Then (as $x \to x^t$ is concave for $0\leq t\leq 1$) we have the following 
\begin{equation}
H_t(P,Q)\le H_t(R\circ P,R\circ Q) \text{ for } 0\le t\le 1.
\end{equation}
Moreover, (as $x \to x^t$ is convex for $t<0$ or $t> 1$) we have the following 
\begin{equation}
H_t(P,Q)\ge H_t(R\circ P,R\circ Q)
\text{ for }t<0 \ \text{or}\ t>1.
\end{equation}
\end{lemma}

\begin{proof}
Let \(\mu\) be a measure dominating both \(P\) and \(Q\), and write $p:=\frac{dP}{d\mu},$ and $
q:=\frac{dQ}{d\mu}.$
Then
\[
H_t(P,Q)
=
\int p^{1-t}q^t\,d\mu
=
\int_{p>0} \left(\frac{q}{p}\right)^t\,dP, 
\]

For \(0<t<1\) define  $f_t(u):=1-u^t,  u\ge0.$ Since \(u\mapsto u^t\) is concave on \([0,\infty)\), \(f_t\) is convex. So,
\[
D_{f_t}(Q\|P)
:=
\int f_t\!\left(\frac{dQ}{dP}\right)dP
=
1-H_t(P,Q)
\]
is an \(f\)-divergence \citet[Definition 7.1]{polyanskiy_wu_2025}. By the data-processing inequality for \(f\)-divergences \citet[Theorem 7.4]{polyanskiy_wu_2025}, we have 
\[
D_{f_t}(R\circ Q\|R\circ P)\le D_{f_t}(Q\|P).
\]

\[
\text{So, } 1-H_t(R\circ P,R\circ Q)\le 1-H_t(P,Q), \leftrightarrow 
H_t(P,Q)\le H_t(R\circ P,R\circ Q).
\]

The cases \(t=0\) and \(t=1\) follow, since $H_0(P,Q)=H_1(P,Q)=1$ and likewise $H_0(R\circ P,R\circ Q)=H_1(R\circ P,R\circ Q)=1.$ Now suppose \(t<0\) or \(t>1\). Then \(u\mapsto u^t\) is convex on
\((0,\infty)\). Thus
\[
D_{g_t}(Q\|P)
:=
\int g_t\!\left(\frac{dQ}{dP}\right)dP= H_t(P,Q) -1,
\quad \text{for }
g_t(u):=u^t -1,
\]
is an \(f\)-divergence.  Applying data processing again on $D_{g_{t}}(Q||P)$ yields for  $t<0$ \text{or} $t>1$
\[
H_t(R\circ P,R \circ Q)
=
D_{g_t}(R\circ Q\|R\circ P) + 1
\le
D_{g_t}(Q\|P) +1
=
H_t(P,Q).
\]

\end{proof}

\begin{figure}[t]
    \centering
    \includegraphics[width=0.65\textwidth]{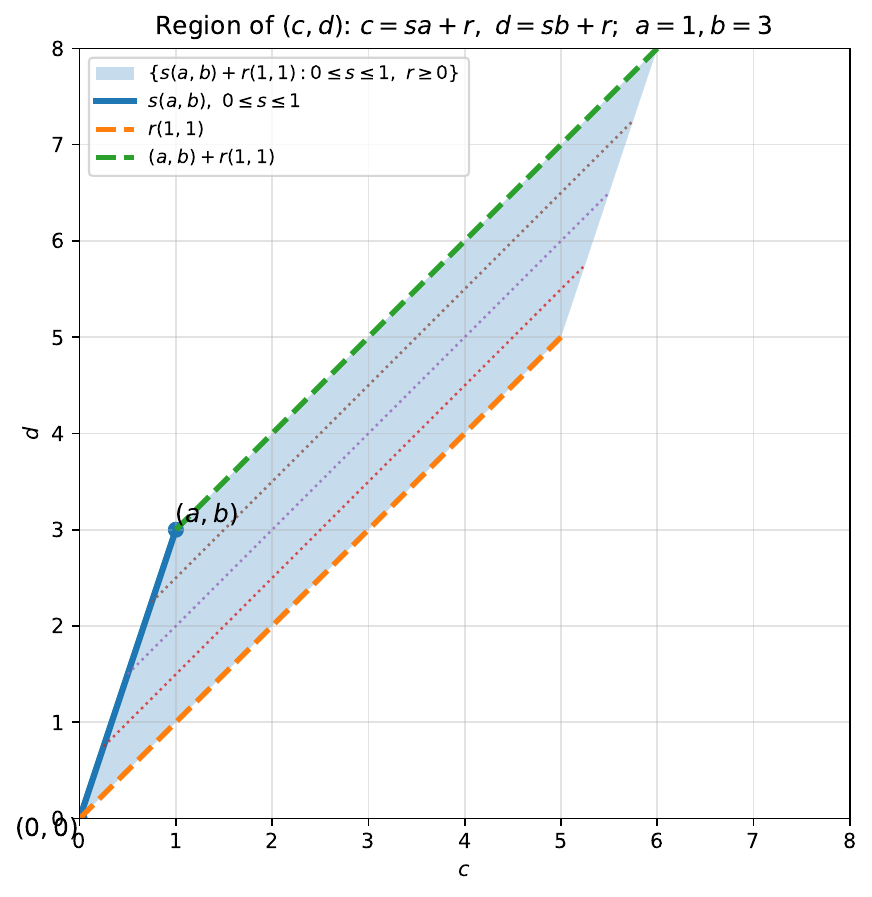}
    \caption{Region of \((c,d)>0\) such that
    \(c=sa+r\) and \(d=sb+r\) for some \(s\in[0,1]\) and \(r\ge 0\).}
    \label{fig:poisson-thinning-region}
\end{figure}
\begin{remark} \label{rem:Poissonregion}
    In the special case when $\alpha >1, \varepsilon <1$ and $\mu_1 , \nu_1 >0$, we have \footnote{Observe that this representation of $\cR_P$ still holds for other choices of $(\alpha, \ve) \in (0, \infty)^2$ too.}
\[
\mathcal R_P
=
\left\{
\mu>0:
0\le \frac{\alpha-1}{\mu_1-\mu}\le1,
1-\mu\frac{\alpha-1}{\mu_1-\mu}\ge0,
0\le \frac{\nu_1-\mu}{\varepsilon-1}\le1,
\mu-\frac{\nu_1-\mu}{\varepsilon-1}\ge0
\right\}.
\]

Since \(\alpha>1\), the condition $0\le \frac{\alpha-1}{\mu_1-\mu}\le1$
implies \(\mu_1-\mu>0  \leftrightarrow \mu<\mu_1\).  The upper bound gives $\alpha-1\le \mu_1-\mu, \leftrightarrow 
\mu\le \mu_1-\alpha+1.$
The condition $1-\mu\frac{\alpha-1}{\mu_1-\mu}\ge0$
is equivalent, since \(\mu_1-\mu>0\), to $\mu(\alpha-1)\le \mu_1-\mu \leftrightarrow 
\mu\le \frac{\mu_1}{\alpha}.$
Therefore, the  constraints are equivalent to\footnote{Observe that an analogous computation for $\alpha <1$ yields that $\mu\ge \max\left\{
\frac{\mu_1}{\alpha},
\mu_1-\alpha+1
\right\}$ }
\[
\mu\le \min\left\{
\frac{\mu_1}{\alpha},
\mu_1-\alpha+1
\right\}.
\]

Next, since \(\varepsilon< 1\), the condition $0\le \frac{\nu_1-\mu}{\varepsilon-1}\le1$ is equivalent to $ \nu_1-\varepsilon+1\ge \mu\ge \nu_1.$
The remaining condition $\mu-\frac{\nu_1-\mu}{\varepsilon-1}\ge0$ is equivalent to
$ 
\mu\le \frac{\nu_1}{\varepsilon}.$
Combining these lower bounds gives\footnote{Observe that $\ve>1$ yields that $\nu_1 \geq \mu\ge
\max\left\{
\frac{\nu_1}{\varepsilon},
\nu_1-\varepsilon+1
\right\}$.}
\[
\nu_1 \le \mu\le
\min\left\{
\frac{\nu_1}{\varepsilon},
\nu_1-\varepsilon+1
\right\}.
\]

\begin{equation} \label{eq:R_P_explicit}
\text{So, }\mathcal R_P(\alpha,\varepsilon,\mu_1,\nu_1)
=
\left[
\nu_1,
\min\left\{
\frac{\mu_1}{\alpha},
\mu_1-\alpha+1, \frac{\nu_1}{\varepsilon},
\nu_1-\varepsilon+1
\right\}
\right],
\end{equation}

\text{provided we have }
$\nu_1
\le
\min\left\{
\frac{\mu_1}{\alpha},
\mu_1-\alpha+1
\right\}.$   Otherwise $\mathcal R_P(\alpha,\varepsilon,\mu_1,\nu_1)=\varnothing.$ 

In general, one can obtain the following table of regions (potentially empty) depending on  $(\alpha, \varepsilon)$.
\[
\begin{array}{c|l}  \label{table:Poissonregion}
\textbf{Case} & \textbf{Region } \mathcal R_P(\alpha,\varepsilon,\mu_1,\nu_1)
\\ \hline
\alpha>1,\ \varepsilon>1
&
\left[
\max\left\{
\dfrac{\nu_1}{\varepsilon},\,
\nu_1-\varepsilon+1
\right\},
\;
\min\left\{
\nu_1,\,
\dfrac{\mu_1}{\alpha},\,
\mu_1-\alpha+1
\right\}
\right]
\\[2ex]

\alpha>1,\ \varepsilon<1
&
\left[
\nu_1,\;
\min\left\{
\dfrac{\mu_1}{\alpha},\,
\mu_1-\alpha+1,\,
\dfrac{\nu_1}{\varepsilon},\,
\nu_1+1-\varepsilon
\right\}
\right]
\\[2ex]

\alpha<1,\ \varepsilon>1
&
\left[
\max\left\{
\dfrac{\mu_1}{\alpha},\,
\mu_1+1-\alpha,\,
\dfrac{\nu_1}{\varepsilon},\,
\nu_1-\varepsilon+1
\right\},
\;
\nu_1
\right]
\\[2ex]

\alpha<1,\ \varepsilon<1
&
\left[
\max\left\{
\dfrac{\mu_1}{\alpha},\,
\mu_1+1-\alpha,\,
\nu_1
\right\},
\;
\min\left\{
\dfrac{\nu_1}{\varepsilon},\,
\nu_1+1-\varepsilon
\right\}
\right]
\end{array}
\]
Moreover, for a fixed $\mu_1, \nu_1>0$, a computation similar to the $\alpha>1 , \varepsilon<1$ considered above yields the following region of allowable removal, preservation pairs $(\alpha >0, \varepsilon>0)$ (Pareto frontiers) so that statistically or information theoretically there is an edited data distribution $p$ that satisfy $(f_d =T(P(1), P(\alpha)) \leftrightarrow \alpha, f_c= T(P(1), P(\varepsilon)))\leftrightarrow \varepsilon$ distributional unlearning \ref{def:dist-unlearning} with respect to $p_1=P(\mu_1),q_1=P(\nu_1)\in \cP= \{P(\mu): \mu >0\}$ or equivalently $\cR_P$ is not empty.
\[
\begin{array}{c|c} \label{table:Paretofrontiers}
\textbf{Case for }\varepsilon
&
\textbf{Allowed values of }\alpha
\\ \hline
\varepsilon>1
&
\displaystyle
\alpha\in
\left[
\max\left\{
\dfrac{\mu_1}{\nu_1},
\mu_1+1-\nu_1
\right\},
1
\right)
\cup
\left(
1,
\min\left\{
\dfrac{\mu_1}{L_\varepsilon},
\mu_1+1-L_\varepsilon
\right\}
\right]
\\[4ex]
0<\varepsilon<1
&
\displaystyle
\alpha\in
\left[
\max\left\{
\dfrac{\mu_1}{U_\varepsilon},
\mu_1+1-U_\varepsilon
\right\},
1
\right)
\cup
\left(
1,
\min\left\{
\dfrac{\mu_1}{\nu_1},
\mu_1+1-\nu_1
\right\}
\right]
\end{array}
\]
\[
\text{ where }
L_\varepsilon
:=
\max\left\{
\dfrac{\nu_1}{\varepsilon},
\nu_1-\varepsilon+1
\right\},
\qquad
U_\varepsilon
:=
\min\left\{
\dfrac{\nu_1}{\varepsilon},
\nu_1+1-\varepsilon
\right\}.
\]

\[
\text{Each interval in the above is interpreted as empty if its left endpoint exceeds its right endpoint.}
\]
\end{remark}
\subsection{Feasible region for Binomial and the Bernoulli family}
We now describe the feasible region for a few more non-location families, more precisely the one-dimensional Binomial and Bernoulli families with appropriate removal preservation baselines.
\begin{remark}[Feasible region for Bernoulli and Binomial family]
    Using \citet[Complement 16, Chapter 10]{torgersen1991comparison}, we can similarly obtain a geometric description of the feasible region $\cR_B$ (see figure \ref{fig:binomial-region-cd}) for the Binomial family $\cP_n= \{B(n,p): p\in (0,1)\}$\footnote{Observe that the binomial family also includes the Bernoulli family $\cP_1$ as a special case.} with $p_1= B(n, \mu_1)$ and $q_1= B(n,\nu_1)$ for some $\mu_1, \nu_1 \in  (0,1)$. In this case, it is natural to consider removal-preservation baselines  $f_{dn}= T(B(n,\frac{1}{2}), B(n, \alpha))$ and $f_{cn}= T(B(n,\frac{1}{2}), B(n, \varepsilon))$  for some $\alpha, \varepsilon \in  (0,1)$.
\end{remark}
\begin{lemma}[Feasible region for Binomial family] \label{lem:binomregion}
    Consider $\cP_n= \{B(n,p): p\in (0,1)\}$ with $p_1= B(n, \mu_1)$ and $q_1= B(n,\nu_1)$ for some $\mu_1, \nu_1 \in  (0,1)$ with removal-preservation baselines  $f_{dn}= T\left(B\left(n,\frac{1}{2}\right), B(n, \alpha)\right)$ and $f_{cn}= T\left(B\left(n,\frac{1}{2}\right), B(n, \varepsilon)\right)$  for some $\alpha, \varepsilon \in  (0,1)$. Then
    \begin{align}
       & \cR_{B}= \{p\in (0,1): (p,\mu_1) \sim \left(\frac{1}{2}, \alpha\right),  (p,\nu_1) \sim \left(\frac{1}{2}, \varepsilon \right)  \text{ and } p \text{ satisfy } 
        \\
        &
        \frac{1-\max(p,\mu_1)}{1-\min(p,\mu_1)} \leq \frac{1-\max\left(\frac{1}{2},\alpha\right)}{1-\min\left(\frac{1}{2},\alpha\right)} \leq \frac{\max\left(\frac{1}{2},\alpha\right)}{\min\left(\frac{1}{2},\alpha\right)} \leq \frac{\max(p,\mu_1)}{\min(p,\mu_1)},
        \\
        &
        \frac{1-\max\left(\frac{1}{2},\varepsilon\right)}{1-\min\left(\frac{1}{2},\varepsilon\right)} \leq \frac{1-\max(p,\nu_1)}{1-\min(p,\nu_1)} \leq \frac{\max(p,\nu_1)}{\min(p,\nu_1)} \leq \frac{\max\left(\frac{1}{2},\varepsilon\right)}{\min\left(\frac{1}{2},\varepsilon\right)}.\}, 
    \end{align}
    where we denote $(a,b) \sim (c,d)$ to mean that $(a,b)$ and $(c,d)$ are similarly ordered.
\end{lemma}
\begin{proof}
    The proof follows immediately from the following Binomial characterization lemma. 
\end{proof}
\begin{lemma}[Binomial characterization lemma] \label{lem:Binom}
    Consider $a,b,c,d \in (0,1)$. Then 
    \begin{enumerate}
        \item For $n\geq 1$, $f_n = T(B(n,a), B(n,b)) \leq g_n=T(B(n,c), T(B(n,d)))$ if and only if  \label{Bcondition1}
        \item  $f= T(B(1,a), B(1,b)) \leq g= T(B(1,c), T(B(1,d)))$ if and only if  \label{Bcondition2}
        \item $(a,b)$ and $(c,d)$ are similarly ordered and we have \label{Bcondition3}
        \begin{equation} \label{eq:Binomcondition}
            \frac{1-\max(a,b)}{1-\min(a,b)} \leq \frac{1-\max(c,d)}{1-\min(c,d)} \leq \frac{\max(c,d)}{\min(c,d)} \leq \frac{\max(a,b)}{\min(a,b)}.  
        \end{equation}
    \end{enumerate}
\end{lemma}
\begin{proof}
    The proof  is inspired from \cite[Complement 16, Chapter 10]{torgersen1991comparison}. First, as in Lemma \ref{lem:ordering}, one analyses the behavior of $f_n$ and $g_n$ around $\alpha \in \{0,1\} $ to show that condition \ref{Bcondition1} is equivalent to \ref{Bcondition2}. Moreover, an explicit computation of the Bernoulli trade-off curve $T(B(1,a), B(1,b))$ from lemma \ref{lem:BernoulliTOF} yields the equivalence of condition \ref{Bcondition1} and condition \ref{Bcondition3}.
\end{proof}
\begin{lemma}[Bernoulli trade-off curve] \label{lem:BernoulliTOF} Let  $P=B(1,a),$ and $ Q=B(1,b),$ for $a,b\in(0,1).$ Then

\noindent\textbf{Case 1: \(b>a\).}
\[
T(B(1,a),B(1, b))(\alpha)
=
\begin{cases}
1-\dfrac{b}{a}\,\alpha,
& 0\le \alpha\le a,\\[1em]
\dfrac{1-b}{1-a}\,(1-\alpha),
& a\le \alpha\le 1.
\end{cases}
\]

\medskip

\noindent\textbf{Case 2: \(b<a\).}
\[
T(B(1,a),B(1,b))(\alpha)
=
\begin{cases}
1-\dfrac{1-b}{1-a}\,\alpha,
& 0\le \alpha\le 1-a,\\[1em]
\dfrac{b}{a}\,(1-\alpha),
& 1-a\le \alpha\le 1.
\end{cases}
\]

\medskip

\noindent\textbf{Case 3: \(a=b\).}
\[
T(B(1,a),B(1,b))(\alpha)
=
1-\alpha \text{ for } 0\leq \alpha \leq 1.
\]
\end{lemma}
\begin{figure}[t]
    \centering
    \includegraphics[width=0.65\textwidth]{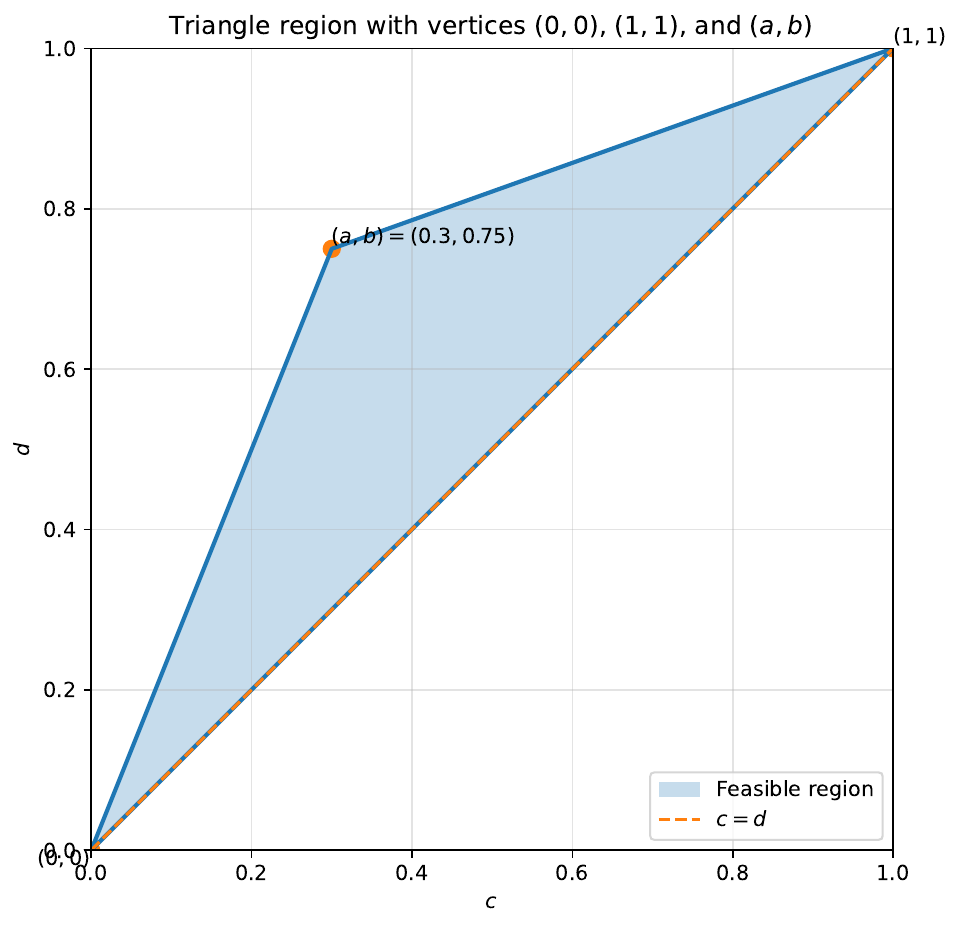}
    \caption{Pairs of \((c,d)\in(0,1)^2\) lying inside the triangle with vertices \((0,0)\), \((1,1)\), and \((a,b)\).}
    \label{fig:binomial-region-cd}
\end{figure}

\section{Appendix: algorithms and information-computation Gap}
\subsection{Random removal for Gaussian families}
\begin{lemma}[Finite-sample concentration of mean] \citet[Lemma 1]{laurent2000adaptive}
\label{lem:hoeffding}
Let $\hat{\mu} =\frac{1}{n} \sum_{i=1}^n X_i$ be the empirical mean of $n$ IID samples drawn from $\mathcal{N}(\mu,\sigma^2\II_d)$. For any $\delta \in (0,1)$, with probability at least $1-\delta$, we have for $\gamma (n,d,2\delta)= \frac{1}{\sqrt{n}} \left( \sqrt{d+2\sqrt{d \log (1/\delta)} + 2 \log (1/\delta)}\right)$
\begin{equation} 
||\hat{\mu} - \mu||_2\leq  \sigma \gamma (n,d,2\delta)
\leftrightarrow 
\mathbb{P}\left[||\hat{\mu} - \mu||_2  \geq \sigma \gamma (n,d,2\delta)\right] \leq \delta
\end{equation} 
\end{lemma}

\samples*
\begin{proof}
    We recall that $p_1 = \cN(\mu_1 , \sigma^2\II_d), q_1 = \cN(\nu_1 , \sigma^2\II_d) \in \cP= \{\cN(\mu, \sigma^2\II_d): \mu \in \R^d\}$.  
    From our assumption, we are given $n_1$ IID samples $x_1^{(1)}, \ldots, x_1^{(n_1)}$ from $p_1$ and $n_2$ IID samples $x_2^{(1)}, \ldots, x_2^{(n_2)}$  from $q_1$. Let $\sigma \Delta:=\|\nu_1-\mu_1\|$. Upon removing $n_r \leq n_1$ randomly chosen samples $x_1^{(1)}, \ldots, x_1^{(n_1 - n_r)}$ (which we can just fix to be the first $n_r$ samples from the unwanted data, since they are IID) from the target distribution $p_1$, we set the center $\mu$ of the unlearned distribution $p\in \cP$ to be the weighted MLE (maximum likelihood estimator) with $n_1':= n_1-n_r$:
    \begin{equation}
        \mu = \frac{n_1' \hat{\mu}_1' + n_2 \hat{\nu}_1}{n_1' + n_2}, \text{ where }
        \hat{\mu}_1' := \tfrac{1}{n_1'} \sum_{i=1}^{n_1'} x_1^{(i)} \text{ and } \hat{\nu}_1 := \tfrac{1}{n_2} \sum_{i=1}^{n_2} x_2^{(i)}.
    \end{equation}
 
    From the Laurent-Massart  lemma~\ref{lem:hoeffding} and union bound\footnote{If the $n_1$ samples from $p_1$ are independently sampled from those of the $n_2$ samples from $p_2$, we can use the independence of the two events $A=\{||\hat{\mu}_1' - \mu_1||_2 \leq \sigma \gamma(n_1', d, \delta)\}, B=\{||\hat{\nu}_1 - \nu_1||_2 \leq \sigma \gamma(n_2, d,\delta)\}$ to conclude that $\mathbb{P}[A\cap B]= \mathbb{P}[A] \mathbb{P}[B] \geq 1- \delta + \frac{\delta^2}{2}$, better than $1-\delta$ obtained using union bound.},  with probability at least $1- \delta$, we  have 
    \begin{align}
       \text{simultaneously } ||\hat{\mu}_1' - \mu_1||_2 \leq \sigma \gamma(n_1', d, \delta), \quad ||\hat{\nu}_1 - \nu_1||_2 \leq \sigma \gamma(n_2, d,\delta)
        \label{eq:hoeffding}
    \end{align}
    
Now, we are in search of $(\alpha, \ve)$ pairs so that $T(p,p_1) \leq f_d \leftrightarrow ||\mu_1-\mu||_2 \geq \alpha \sigma  $ and $T(p,q_1) \geq f_c \leftrightarrow ||\nu_1-\mu||_2 \leq \ve \sigma $. Remember, if $\alpha >\Delta +\ve$, then the feasible region is itself empty \ref{thm:FFR}.

\paragraph{Preservation bound.} Using triangle inequality for norms, we have 
\begin{align*}
    ||\mu - \nu_1||_2
    &
    = 
    \left|\left|\frac{n_1' \hat{\mu}_1' + n_2 \hat{\nu}_1}{n_1' + n_2}
    -
    \nu_1 \right|\right|_2 
    = 
    \left|\left|\frac{n_1'}{n_1' + n_2} (\hat{\mu}_1' - \nu_1) 
    +
    \frac{n_2}{n_1' + n_2} (\hat{\nu}_1 - \nu_1) \right|\right|_2
    \\
    &
    = 
    \left|\left|\frac{n_1'}{n_1' + n_2} (\mu_1 - \nu_1) 
    +
    \frac{n_1'}{n_1' + n_2} (\hat{\mu}_1' - \mu_1) 
    + 
    \frac{n_2}{n_1' + n_2} (\hat{\nu}_1 - \nu_1) \right|\right|_2
    \\
    &
    \leq \frac{n_1'}{n_1' + n_2} ||\mu_1 - \nu_1||_2 +
    \frac{n_1'}{n_1' + n_2} ||\hat{\mu}_1' - \mu_1||_2
    +
    \frac{n_2}{n_1' + n_2} ||\hat{\nu}_1 - \nu_1||_2.
\end{align*}
Therefore, using the Laurent Massart bound ~\eqref{eq:hoeffding} we have with probability atleast $1-\delta$:
\begin{align*}
    ||\mu - \nu_1||_2
    &\leq \frac{n_1'}{n_1' + n_2} \sigma \Delta + \frac{n_1'}{n_1' + n_2} \sigma \gamma(n_1',d,\delta) + \frac{n_2}{n_1' + n_2} \sigma \gamma(n_2, d,\delta).
\end{align*}
\begin{remark}[Lower bound] \label{rem:lowerboundRR}
    We also observe a triangle inequality $||a+b||_2 \geq ||a||_2 -||b||_2$ yields with probability at least $1-\delta$ we also have the following lower bound
    \begin{align*}
    ||\mu - \nu_1||_2
    &\geq
    \frac{n_1'}{n_1' + n_2} \sigma \Delta 
    -
    \frac{n_1'}{n_1' + n_2} \sigma \gamma(n_1',d,\delta) 
    -
    \frac{n_2}{n_1' + n_2} \sigma \gamma(n_2, d,\delta).
\end{align*}
\end{remark}
\paragraph{Removal bound.}
Using the Laurent Massart bound ~\eqref{eq:hoeffding} with probability $1-\delta$, we have
\begin{align*}
   &
   \sigma \Delta= ||\mu_1 - \nu_1||_2 
     \leq  ||\mu_1 - \mu||_2 + ||\mu - \nu_1||_2
     \\
     &
    \leq  ||\mu_1 - \mu||_2  +  \frac{n_1'}{n_1' + n_2} \sigma \Delta + \frac{n_1'}{n_1' + n_2} \sigma \gamma(n_1',d,\delta) + \frac{n_2}{n_1' + n_2} \sigma \gamma(n_2, d,\delta).
\end{align*}
Rearranging the terms yields that with probability $1-\delta$ we have the following 
\begin{align*} 
    ||\mu_1 - \mu||_2  \geq 
    \frac{n_2}{n_1' + n_2} \sigma \Delta 
    - \frac{n_1'}{n_1' + n_2} \sigma \gamma(n_1',d,\delta) - \frac{n_2}{n_1' + n_2} \sigma \gamma(n_2, d,\delta)
\end{align*}
\begin{remark}[Upper bound]  \label{rem:upperboundRR}
    We also observe a triangle inequality $||a+b||_2 \leq ||a||_2 +||b||_2$ yields with probability at least $1-\delta$ we also have the following upper bound
    \begin{align*}
    ||\mu - \mu_1||_2
    &
    \leq
    \frac{n_2}{n_1' + n_2} \sigma \Delta 
    +
    \frac{n_1'}{n_1' + n_2} \sigma \gamma(n_1',d,\delta) 
    +
    \frac{n_2}{n_1' + n_2} \sigma \gamma(n_2, d,\delta).
\end{align*}
\end{remark}

\paragraph{Conclusion.}
With probability $1-\delta$, we have $(f_d \leftrightarrow \alpha, f_c \leftrightarrow  \varepsilon)$-distributional unlearning \ref{def:dist-unlearning} with
\begin{align*}
    \alpha 
    &
    \leq
    \alpha_M(R)
    :=
    \frac{n_2}{n_1' + n_2} \Delta 
    -
    \frac{n_2}{n_1' + n_2} \gamma(n_2, d,\delta)
    - 
    \frac{n_1'}{n_1' + n_2} \gamma(n_1',d,\delta) 
    ,\\
    \varepsilon &\geq 
    \ve_m(R)
    :=
    \frac{n_1'}{n_1' + n_2}  \Delta + \frac{n_1'}{n_1' + n_2}  \gamma(n_1',d,\delta) + \frac{n_2}{n_1' + n_2}  \gamma(n_2, d,\delta).
\end{align*}
\end{proof}

\begin{remark}[Information-Computation gap of random removal] \label{rem:ICGRR}
    The information-theoretic feasible region for the shifted Gaussian family (Proposition \ref{thm:FFR}) allows any $\ve \geq 0$, and requires $\alpha \leq \Delta + \ve$. However, for the random removal to output a distribution (with high probability), we need a lower bound $\ve \geq \ve_m(R) >0$. To be very precise, the bound $\ve_m(R)$ is a high probability upper bound. However, from Remark \ref{rem:lowerboundRR}, with high probability (at least $1-\delta$) the random removal algorithm fails to satisfy the preservation bound $(T(p,q_1)\geq T(N(0,1), N(\ve,1)))$ as soon as we take \[
    \ve= \frac{n_1'}{n_1' + n_2}  \Delta - \frac{n_1'}{n_1' + n_2}  \gamma(n_1',d,\delta) - \frac{n_2}{n_1' + n_2}  \gamma(n_2, d,\delta) 
    \text { or lower}.
    \]
    Similarly, from Remark \ref{rem:upperboundRR}, with high probability (at least $1-\delta$) the random removal algorithm fails to satisfy the removal bound $(T(p,p_1)\leq T(N(0,1), N(\alpha,1)))$ as soon as we take \[
    \alpha = \frac{n_2}{n_1' + n_2}  \Delta + \frac{n_1'}{n_1' + n_2}  \gamma(n_1',d,\delta) + \frac{n_2}{n_1' + n_2}  \gamma(n_2, d,\delta) 
    \text { or larger}.
    \]
    In summary, the high probability behavior of the random removal algorithm for shifted Gaussians has a smaller Pareto frontier, or, said differently, for a given pair of removal-preservation parameters $(\alpha, \ve) \in (0,\infty)^2$, the non-emptiness of the information theoretic feasible required $\Delta \geq \alpha- \ve$. However, ignoring the lower-order finite sample approximation terms such as $\frac{n_1'}{n_1' + n_2}  \gamma(n_1',d,\delta)$ or  $\frac{n_2}{n_1' + n_2}  \gamma(n_2, d,\delta)$, the high probability success of the random removal requires 
    \[
    \alpha \leq \frac{n_2}{n_1'+n_2} \Delta 
    \text{ and }
    \frac{n_1'}{n_1'+n_2}\Delta 
    \leq 
    \ve.
    \]
    Regarding ignoring the finite sample approximation terms, we are assuming that the following holds.
    \[
    \min(n_2\Delta, n_1'\Delta) >> \max(n_2 \gamma(n_2,d,\delta), n_1'\gamma(n_1',d,\delta)).
    \]
   Now, once we ignore the $\delta$ term in $\gamma(n, d, \delta)$ \ref{lem:hoeffding} (assume this is user specified and independent of every other parameter such as $d,n_1,n_2,n_r$) and take noise variance $\sigma=1$, we have\footnote{We have $ \sqrt{n}\gamma(n,d,\delta) \sim \sqrt{d}$ or equivalently $\lim_{n,d \to  \infty} \frac{\sqrt{n}\gamma(n,d,\delta)}{\sqrt{d}}= C$ for some absolute $C>0$.}
   \[
   \max(n_2 \gamma(n_2,d,\delta), n_1'\gamma(n_1',d,\delta)) \sim \max( \sqrt{n_2d}, \sqrt{n_1'd}).
   \]
   So, the assumption holds (recall with $\sigma= 1$, we have $\Delta= \|\mu_1- \nu_1\|_2$) as soon as
   \[
   \min(n_2\Delta, n_1'\Delta) >> \max( \sqrt{n_2d}, \sqrt{n_1'd})  \text{ or equivalently }
   \]
   \[
   \Delta = \|\mu_1- \nu_1\|_2
   >>
   \max
   \left(\sqrt{\frac{d}{n_2}}, \sqrt{\frac{dn_1'}{n_2^2}}, \sqrt{\frac{d}{n_1'}}, \sqrt{\frac{n_2 d}{n_1'^2}} \right).
   \]
   This assumption holds in most situations since for any pair of $\mu_1, \nu_1 \in \R^d$, if every coordinate differs by a constant $\Theta(1)$, then $\|\mu_1- \nu_1\|_2 \sim \Theta(\sqrt{d})$, and we need
   \[
   1 >> \max
   \left(\frac{1}{n_2}, \frac{n_1'}{n_2^2}, \frac{1}{n_1'}, \frac{n_2}{n_1'^2} \right) \text{or equivalently }
   \]
   \[
   n_2 >> \max(1, \sqrt{n_1'}) \text{ and }
   n_1' >> \max(1, \sqrt{n_2}).
   \]
\end{remark}
\subsection{Random removal for log-concave families with weighted mean}
Now, we provide a generalization of the finite sample Gaussian result $\{N(\mu, \sigma^2\II_d): \mu\in \R^d\}$ with removal-preservation baselines $f_d= T(N(0,1), N(\alpha, 1)), f_c=T(N(0,1), N(\ve,1))$ \ref{th:random} to a location family $\{p_{\mu} \overset{d}{=} \mu+X: \mu \in \R\}$ in $d=1$ with a symmetric log-concave noise $p_0 \overset{d}{=}X$ and removal-preservation baselines $f_d= T(X, \alpha+X), f_c=T(X,\ve+X)$.
\begin{restatable}[Random Removal]{proposition}{samplesLCRR}
\label{th:randomLC}
Let $p_1\overset{d}{=}\mu_1+X, q_1\overset{d}{=} \nu_1+ X \in \mathcal{P}= \{p_{\mu} \overset{d}{=} \mu+X: \mu \in \R\}$ and $\delta \in (0,1)$, where we assume that $p_0\overset{d}{=}X$ is a symmetric log-concave distribution on $\R$, and consider removal-preservation baselines as $f_d= T(X, \alpha+X), f_c=T(X,\ve+X)$. We observe $n_1$ IID samples from $p_1$ and $n_2$ IID samples from $q_1$, and randomly remove $n_r$ samples from $p_1$ before fitting a weighted mean $\mu$ `according' to the baseline random removal algorithm \ref{algo:RR}\footnote{It is important to mention that for the more general shifted log-concave family $\{\mu+X:\mu\in \R\}$, we do not fit the MLE, since the weighted mean $\mu$ is a weighted average of empirical means $\hat{\mu}_1'$ and $\hat{\nu}_1$, and both the empirical means  $\hat{\mu}_1'$ and $\hat{\nu}_1$ are not MLEs, since empirical mean is not an MLE for log-concave families in general (such as shifted Laplace family), as it is the case for Gaussians. For mathematical tractability, and trying to bring out the structure of a reasonable baseline  in the more general log-concave situation, we still consider weighted empirical  means and not weighted MLEs, as required in the baseline random removal algorithm.}. With probability atleast $1 - \delta$, the resulting  distribution $p_{\mu}$ satisfies $(f_d \leftrightarrow \alpha,f_c\leftrightarrow  \varepsilon)$ unlearning \ref{def:dist-unlearning} with:
\begin{align*}
\alpha & \leq \alpha_M(R):=\frac{n_2}{n_1' + n_2} \Delta - \frac{n_2}{n_1' + n_2} \gamma_X(n_2, 1,\delta) -\frac{n_1'}{n_1' + n_2} \gamma_X(n_1',1,\delta) , \\
\varepsilon &\geq  \ve_m(R):= \frac{n_1'}{n_1' + n_2}  \Delta + \frac{n_1'}{n_1' + n_2}  \gamma_X(n_1',1,\delta) + \frac{n_2}{n_1' + n_2}  \gamma_X(n_2, 1,\delta).
\end{align*}
where $\Delta = |\mu_1- \nu_1|$, $n_1'= n_1- n_r$ and $\gamma_X(n, 1, \delta)>0$ is a number dependent on the distribution $p_0$ (symmetric log-concave) such that for $n$ samples $X_1, \cdots, X_n$   IID from $p_0$ we have 
\begin{equation}
    \mathbb{P}\left[|\sum_{i=1}^n X_i| \geq n \gamma_X(n,1, 2\delta)\right] \leq \delta \leftrightarrow  2 \mathbb{P}\left[\sum_{i=1}^n X_i \geq n \gamma_X(n,1, 2\delta)\right] \leq \delta
\end{equation}
\end{restatable}

\begin{proof}
    The proof is the same as that of the finite sample Gaussian result \ref{th:random}, since from \ref{prop:FFRLC}, we are in search of $(\alpha, \ve)$ pairs so that $T(p,p_1) \leq f_d \leftrightarrow |\mu_1-\mu|_2 \geq \alpha  $ and $T(p,q_1) \geq f_c \leftrightarrow |\nu_1-\mu|_2 \leq \ve $. Similarly, if $\alpha >\Delta +\ve$, then the feasible region is itself empty \ref{prop:FFRLC}.
    Now, we recall that $p_1 \overset{d}{=} \mu_1+X, q_1 \overset{d}{=} \nu_1+ X \in \cP= \{p_{\mu} \overset{d}{=}\mu+X: \mu \in \R\}$.  
    From our assumption, we are given $n_1$ IID samples $x_1^{(1)}, \ldots, x_1^{(n_1)}$ from $p_1$ and $n_2$ IID samples $x_2^{(1)}, \ldots, x_2^{(n_2)}$  from $q_1$. Let $\Delta:=\|\nu_1-\mu_1\|$. Upon removing $n_r \leq n_1$ randomly chosen samples $x_1^{(1)}, \ldots, x_1^{(n_1 - n_r)}$ (which we can just fix to be the first $n_r$ samples from the unwanted data, since they are IID) from the target distribution $p_1$, we set the center $\mu$ of the unlearned distribution $p\in \cP$ to be the weighted mean with $n_1':= n_1-n_r$:
    \begin{equation} \label{eq:empmeanLCRR}
        \mu = \frac{n_1' \hat{\mu}_1' + n_2 \hat{\nu}_1}{n_1' + n_2}, \text{ where }
        \hat{\mu}_1' := \tfrac{1}{n_1'} \sum_{i=1}^{n_1'} x_1^{(i)} \text{ and } \hat{\nu}_1 := \tfrac{1}{n_2} \sum_{i=1}^{n_2} x_2^{(i)}.
    \end{equation}

    From the assumption on the quantity $\gamma_X(n,1,\delta)$ analogous to the Laurent-Massart  lemma~\ref{lem:hoeffding} and union bound,  with probability at least $1- \delta$, we  have  simultaneously 
    \begin{align}
        |\hat{\mu}_1' - \mu_1| \leq \gamma_X(n_1', 1, \delta), \quad |\hat{\nu}_1 - \nu_1| \leq  \gamma_X(n_2, 1,\delta)
        \label{eq:hoeffdingLCRR}
    \end{align}

    Now, one can repeat the arguments of the \textbf{Preservation}  and the  \textbf{Removal} bounds \ref{th:random} to conclude.
    \end{proof} 
\begin{remark}[Table of $\gamma_X(n, 1,\delta)$]
    The finite sample results in the symmetric log-concave case in $d=1$ are very similar to $\gamma(n,d,\delta)$ replaced by $\gamma_X(n,1, \delta)$, and for a given noise $p_0$, one can find a valid choice of $\gamma_X(n,1,2\delta)$. Moreover, the extension of this symmetric log-concave location family to higher dimensions would be interesting, since it requires a generalization of the result \ref{prop:FFRLC}.
\end{remark}
\subsection{Random removal for shifted Laplace family with weighted MLE}
\begin{remark}(Empirical mean to MLE) In  \ref{th:randomLC}, with the same notation  $p_1,q_1, \cP, f_d,f_c$ and the sampling setup, one can replace the empirical mean estimators $\hat{\mu}_1', \hat{\nu}_1$ with any estimator $\hat{\mu}_1'$ of $\mu_1$, and $\hat{\nu}_1$ of $\nu_1$ (for example MLE) and consider $\mu= \frac{n_1' \hat{\mu}_1' + n_2 \hat{\nu}_1}{n_1+n_2}$. Then, the conclusion of \ref{th:randomLC} becomes with probability $\geq 1 - \delta$, the resulting   $p_{\mu}$ satisfies $(f_d \leftrightarrow \alpha,f_c\leftrightarrow  \varepsilon)$-distributional unlearning \ref{def:dist-unlearning} with:
\begin{align*}
\alpha & \leq \alpha_M(R):=\frac{n_2}{n_1' + n_2} \Delta - \frac{n_2}{n_1' + n_2} \gamma_{\nu_1}(n_2, 1,\delta) - \frac{n_1'}{n_1' + n_2} \gamma_{\mu_1}(n_1',1,\delta) , \\
\varepsilon &\geq  \ve_m(R):= \frac{n_1'}{n_1' + n_2}  \Delta + \frac{n_1'}{n_1' + n_2}  \gamma_{\mu_1}(n_1',1,\delta) + \frac{n_2}{n_1' + n_2}  \gamma_{\nu_1}(n_2, 1,\delta).
\end{align*}
where $\Delta = |\mu_1- \nu_1|$, $n_1'= n_1- n_r$ and $\gamma_{\nu}(n, 1, \delta)>0$ for $\nu\in \R$ are numbers (that depends on the estimating procedure $\hat{\nu}$ of $\nu$ too) such that for $n$ samples $X_1, \cdots, X_n$   IID from $p_{\nu}$ we have 
\begin{equation} \label{eq:MLEC}
    \mathbb{P}\left[|\hat{\nu} -\nu| \geq  \gamma_{\nu}(n,1, 2\delta)\right] \leq \delta 
\end{equation}
Values of $\gamma_{\nu}(n,1, 2\delta)$, assuming MLE, can be explicit under smoothness assumptions on the symmetric log-concave  $\cP=\{\mu+X: \mu \in \R\}$ \citet{koltchinskii2023functional, gupta2022finite}. 
\end{remark}

\begin{remark}(MLE in the Laplace case) \label{rem:RRLaplace}
    In \ref{th:randomLC} with the same  $p_1,q_1, \cP, f_d,f_c$ and the sampling setup, once we specialize to $\cP= \{p_{\nu} \overset{d}{=}\nu+X:\nu \in \R\}$ for  $X\overset{d}{=} p_0$ with $2p_0(x)dx = \exp(-|x|) dx$, it is known that  MLE$(X_1, \cdots, X_n)$ for IID $X_i$ from $p_{\nu}$ is given by $\hat{\nu}=$ median$(X_1, \cdots, X_n)$. For simplicity, let's   assume $n=2k+1$ is odd, and  the ordered samples are denoted as
    \begin{equation}
        X_{(1)} \leq  \cdots \leq X_{(k+1)} \leq \cdots  \leq X_{(2k-1)}
    \end{equation}
    Then $\hat{\nu}=$ med$(X_1, \cdots, X_n)= X_{(k+1)}$. Let $Z_i =X_i -\nu \overset{d}{=} p_0$ for $1\leq i\leq n$. Then, we have that $Z_1, \cdots, Z_n$ are IID and
    median$(X_1, \cdots, X_n)= \nu+$ med$(Z_1, \cdots, Z_n)$. Therefore, for any $c>0$ 
    \begin{equation}
        \mathbb{P}[|\hat{\nu}- \nu| \geq c ] = \mathbb{P}[|\text{med}(Z_1, \cdots, Z_n)| \geq c] = 2 \mathbb{P}[\text{med}(Z_1, \cdots, Z_n) \geq c]
    \end{equation}
    where the last equality follows from the symmetry of $p_0$, $(Z_1, \cdots, Z_n)= (-Z_1, \cdots, -Z_n)$ and
    \[
    -\text{med}(Z_1, \cdots, Z_n)= \text{med}(-Z_1, \cdots, -Z_n) \overset{d}{=} \text{med} (Z_1, \cdots, Z_n)=:M_n
    \]
   Define $N_c:=\sum_{i=1}^n \mathbf{1}\{Z_i< c\}\sim  \operatorname{B}(n,F(c))$\footnote{$B(n,p)$ is the standard binomial distribution on $\{0,1,  \cdots, n-1, n\}$ with the success probability $p$.}, where $F(c)= \mathbb P(Z_i\le c) = 1-\frac12 e^{-c}.$ is the CDF of $Z_i$. Since, $M_n=Z_{(k+1)}$, the event \(\{M_n\ge c\}\) implies that at most \(k\) observations are at most \(c\).  
   \[
   \,\mathbb P(M_n\ge c)
   \leq
   \mathbb{P}[N_c \leq k] 
   = \mathbb{P}[B(n, F(c))\leq k]
   \leq 
   \mathbb{P}\left[B(n, F(c)) - nF(c)\leq n \left( \frac{1}{2} - F(c)\right)\right]
   \]

By Hoeffding's inequality\footnote{One can obtain sharper lower tail estimates for Binomial than Hoeffding \citet{boucheron2013concentration}.} applied to Bernoulli  variables
\(\mathbf{1}\{Z_i\le c\}\) and $F(c)-\frac12
    =
    \frac12(1-e^{-c}),$
\[
    2 \mathbb P(M_n\ge c)
    \le
    2\exp\left\{
        -\frac n2(1-e^{-c})^2
    \right\} \leq \delta, \text{ where}
\]

we choose $c = -\log\left(1-\sqrt{\frac{2}{n}\log\frac{2}{\delta}}\right)$, with \(n>2\log(4/\delta)\) to ensure that the quantity inside the logarithm
belongs to \((0,1)\). Therefore, a valid choice of $\gamma_{\nu}(n,1,2\delta)= -\log\left(1-\sqrt{\frac{2}{n}\log\frac{2}{\delta}}\right)$.

We conclude by observing that, as $n \uparrow \infty$, the choices of the tail estimates are $\gamma_{\nu}(n,1, 2\delta) \sim \sqrt{\frac{2}{n}\log\frac{2}{\delta}}$  for MLE, equivalently the median, and $\gamma_{\nu}(n,1, 2\delta) \sim \sqrt{\frac{4}{n}\log\frac{2}{\delta}}$ for the empirical mean.

\end{remark}

\begin{table}[t] \label{table:LMconstants}
\centering
\small
\setlength{\tabcolsep}{6pt}
\renewcommand{\arraystretch}{1.55}
\begin{tabularx}{\textwidth}{
@{}>{\raggedright\arraybackslash}p{0.43\textwidth}
>{\centering\arraybackslash}X@{}}
\toprule
Symmetric log-concave distribution \(p_0\)
&
A valid choice of \(n\gamma_X(n,1,2\delta)\)
\\
\midrule

Standard Gaussian: \(X\sim N(0,1)\)
&
\(n\displaystyle
    \gamma_X(n,1,2\delta)
    =
    \sqrt n\,\Phi^{-1}\!\left(1-\frac{\delta}{2}\right)
\)
\\

Uniform: \(X\sim \operatorname{Unif}[-a,a]\)
&
\(n\displaystyle
    \gamma_X(n,1,2\delta)
    =
    a\sqrt{2n\log\frac{2}{\delta}}
\)
\\

Laplace: \(2p_0(x)=\exp{(-|x|)}\)
&
\(n\displaystyle
    \gamma_X(n,1,2\delta)
    =
    2\sqrt{n\log\frac{2}{\delta}}
    +
    \log\frac{2}{\delta}
\)
\\

Bounded support \(\subseteq[-a,a]\)
&
\(\displaystyle
    n\gamma_X(n,1,2\delta)
    =
    a\sqrt{2n\log\frac{2}{\delta}}
\)
\\

\bottomrule
\end{tabularx}
\caption{Examples of symmetric log-concave  \(p_0\) with  valid \(n\gamma_X(n,1,2\delta)\) \citet{boucheron2013concentration}.}
\label{tab:symmetric-log-concave-tail-bounds}
\end{table}

\subsection{Selective removal for Gaussian family}

\begin{lemma}[Dvoretzky–Kiefer–Wolfowitz  Inequality] \citet{massart1990tight}
\label{lem:DKW}
Let \(X_1, X_2, \dots, X_{n}\) be independent and identically distributed random variables on some probability space $(\Omega, \cF, \mathbb{P})$ with cumulative distribution function \(F:\R \to [0,1]\). Define the empirical distribution function by $\widehat{F}_n(t) = \frac{1}{n} \sum_{i=1}^{n} \mathbf{1}\{X_i \le t\}.$
Then, for any   \(\delta \in (0,1)\),   we have with $d(n, 2\delta):=\sqrt{\frac{\ln(2/\delta)}{2n}}$ 
\[
\mathbb{P}\left[ \sup_{t \in \mathbb{R}} \left|\widehat{F}_n(t) - F(t)\right| \le d(n, 2\delta) \right] \geq 1-\delta \leftrightarrow \mathbb{P}\left[\sup_{t \in \mathbb{R}} \left|\widehat{F}_n(t) - F(t)\right| \geq d(n, 2\delta) \right] \leq \delta 
\]
\end{lemma}
\begin{remark}[Independence of data distribution]
  We observe that in the above inequality, the bound $d(n,\delta)$ is independent of the data distribution $F$, from which the data is sampled. 
\end{remark}

\begin{lemma}[Concentration of the biased empirical mean after selective removal, towards the true mean]
\label{lem:bias_selection}
    Let $\mu_1, \nu_1 \in \R^d$, $\sigma>0$.
    Consider the following sampling setup
    \[
    n_1
    \text{ IID samples }
    x_1^{(1)}, \ldots, x_1^{(n_1)} 
    \text{ from }
    \cN(\mu_1, \sigma^2\II_d),
    \]
    \[
    n_2 \text{ IID  samples }
    x_2^{(1)}, \ldots, x_2^{(n_2)}
    \text{ from }
    \cN(\nu_1, \sigma^2\II_d).
    \]
    We define $\hat{\nu}_1$ the average of the samples from $\cN(\nu_1, \sigma^2)$ given as $\hat{\nu}_1= \frac{1}{n_2} \sum_{j=1}^{n_2}x_2^{(j)}$, and  $\hat{\mu}_1'$ the average of the $n_1'=n_1 - n_r \leq n_1$ closest samples from $x_1^{(1)}, \ldots, x_1^{(n_1)}$ to $\hat{\nu}_1$\footnote{Recall, these closest $n_1'$ samples from the empirical mean $\hat{\nu}_1$ of the desired data, although generated from the unwanted data are used for computing the weighted empirical mean in the selective removal algorithm \ref{algo:SR}.}. We define
    \begin{equation}
    F_1: \R_{\geq 0} \to  [0,1) \text{ as }
    F_1(t):= \mathbb{P}[\|\sigma Z + \mu_1- \nu_1\|_2 \leq t] \text{ with } Z \sim N(0, \II_d). 
    \end{equation}
For $\delta \in (0,1)$, we have with probability at least $1-\delta$ with constants $d(n_1, \delta)$ \ref{lem:DKW} and $\gamma(n_2, d, \delta)$ \ref{lem:hoeffding}
    \begin{equation}
       ||\hat{\mu}_1' - \hat{\nu}_1||_2 \leq F_1^{-1}{\left(\frac{n_1'}{n_1} + d(n_1, \delta)\right)} + \sigma \gamma(n_2,d, \delta).
    \end{equation}
\end{lemma}
\begin{proof}
Recall $\hat{\mu}_1'$ is the average of the $n_1'$ samples, out of $n_1$ IID samples from $p_1=N(\mu_1, \sigma^2\II_d)$, with the closest distance to $\hat{\nu}_1$, the empirical mean of $n_2$ samples from $q_1=N(\nu_1, \sigma^2\II_d)$.  
 We denote by $x_1^{(1:n_1)}, \ldots, x_1^{(n_1:n_1)}$ the original $n_1$ samples from $p_1$ reordered by distance to $\hat{\nu}_1$\footnote{Since we assume the samples are coming from a `continuous' distribution with a Lebesgue density, without loss of generality, we can assume that the distances are all pairwise distinct numbers with probability 1.}:
    \begin{align}
        \|x_1^{(1:n_1)} - \hat{\nu}_1\|_2
        \leq  \ldots \leq \|x_1^{(n_1:n_1)} - \hat{\nu}_1\|_2,
    \end{align}
    with ties broken arbitrarily, then distance-based selection retains only $x_1^{(1:n_1)}, \ldots, x_1^{(n_1':n_1)}$ to obtain
    \begin{equation}
        \label{eq:selectionG}
        \hat{\mu}_1' \coloneqq \tfrac{1}{n_1'} \sum_{i=1}^{n_1'} x_1^{(i:n_1)}.
    \end{equation}
    Denote by $\hat{\tau}_r$ the $n_1'$-th largest distance of the unwanted samples to $\hat{\nu}_1$, defined as 
    \begin{equation} \label{eq:tau}
         \hat{\tau}_r \coloneqq \|x_1^{(n_1':n_1)} - \hat{\nu}_1\|_2
    \end{equation}
    From the triangle inequality,  we have the following upper bound\footnote{We observe this is the step; we only obtain an upper bound on $\|\hat{\mu}_1' - \hat{\nu}_1\|_2$, and therefore obtaining a tighter lower bound on $\|\hat{\mu}_1' - \hat{\nu}_1\|_2$ would be interesting, since that will tell us with high probability when the selective removal algorithm cannot achieve unlearning.}
\begin{align*}
    ||\hat{\mu}_1' - \hat{\nu}_1||_2 
    = 
    \left|\left|\frac{1}{n_1'}\sum_{i=1}^{n_1'} x_1^{(i:n_1)} - \hat{\nu}_1\right|\right|_2 
    \leq
    \frac{1}{n_1'}\sum_{i=1}^{n_1'} \| x_1^{(i:n_1)} - \hat{\nu}_1\|_2  
    \leq
    \hat{\tau}_r
    \label{eq:distance-triangle}
\end{align*}

Let $\widehat{F}_1$ be the CDF of the empirical distribution over the distances $\left \{ \|x_1^{(i)} - \nu_1\|_2 ~\colon~ i \in [n_1] \right\}$, of the unwanted samples $\{x_1^{(i)}: i\in [n_1]\}$ from the population mean $\nu_1$ of the desired population $q_1$ 
\begin{align}
    \widehat{F}_1(t) = \frac{1}{n_1} \sum_{i=1}^{n_1} \mathbf{1}_{\left\{ \|x_1^{(i)} - \nu_1\|_2 \leq t \right\}}.
\end{align}
Similarly, let $\widehat{E}_1$ be the CDF of the empirical distribution over distances $\left \{ \|x_1^{(i)} - \hat{\nu}_1\|_2 ~\colon~ i \in [n_1] \right\}$, of the unwanted samples $\{x_1^{(i)}: i\in [n_1]\}$ from the sample mean $\hat{\nu}_1$ of the desired population $q_1$ 
\begin{align}
    \widehat{E}_1(t) = \frac{1}{n_1} \sum_{i=1}^{n_1} \mathbf{1}_{\left\{ \|x_1^{(i)} - \hat{\nu}_1\|_2 \leq t \right\}}.
\end{align}
Now, we recall the Laurent Massart bound \ref{lem:hoeffding} that  says 
\begin{align}
     \mathbb{P}[||\hat{\nu}_1 - \nu_1||_2 \geq \sigma \gamma(n, d, 2\delta)] \leq  \delta .
\end{align}
Therefore, with probability at least $1-\frac{\delta}{2}$, for $i\in [n_1]$, we simultaneously have the following

\begin{equation}
||x_1^{(i)} - \nu_1||_2 
\leq 
||x_1^{(i)} - \hat{\nu}_1||_2
+
||\hat{\nu}_1 - \nu_1||_2
< 
||x_1^{(i)} - \hat{\nu}_1||_2 
+
\sigma  \gamma(n_2, d, \delta)
\end{equation}
\begin{equation}
||x_1^{(i)} - \hat{\nu}_1||_2 
\leq 
||x_1^{(i)} - \nu_1||_2
+
||\hat{\nu}_1 - \nu_1||_2
<
||x_1^{(i)} - \nu_1||_2 
+
\sigma  \gamma(n_2, d, \delta)
\end{equation}

So, with probability at least $1-\frac{\delta}{2}$,  we simultaneously have the inclusion of events for $i\in [n_1]$
\begin{equation} 
    \{||x_1^{(i)} - \nu_1||_2  < t\} \supseteq 
    \{||x_1^{(i)} - \hat{\nu}_1||_2 + \sigma  \gamma(n_2, d, \delta) \leq t\}
\end{equation}
\begin{equation}
    \{||x_1^{(i)} - \hat{\nu}_1||_2  < t\} \supseteq \{||x_1^{(i)} - \nu_1||_2 + \sigma  \gamma(n_2, d, \delta) \leq t\}
\end{equation}

As a consequence, with probability at least $1-\frac{\delta}{2}$,  we have for all $t\in \R$\footnote{The upper bounds in \eqref{eq:F-to-E} still hold if we replace $\hat{E}_1(t)$ with $\hat{E}_1(t^-):= \frac{1}{n_1} \sum_{i=1}^{n_1} \mathbf{1}_{\left\{\|x_1^{(i)} - \hat{\nu}_1\|_2 < t \right\}} $ }
\begin{equation} \label{eq:F-to-E}
 \widehat{F}_1(t -\sigma \gamma(n_2, d, \delta)) 
 =
 \frac{1}{n_1} \sum_{i=1}^{n_1} \mathbf{1}_{\left\{ \|x_1^{(i)} - \nu_1\|_2 + \sigma \gamma(n_2, d, \delta) \leq t \right\}} 
    \leq
\frac{1}{n_1} \sum_{i=1}^{n_1} \mathbf{1}_{\left\{\|x_1^{(i)} - \hat{\nu}_1\|_2 \leq t \right\}} 
=
\widehat{E}_1(t)
\end{equation}

\begin{align}
 \widehat{E}_1(t -\sigma \gamma(n_2, d, \delta)) 
 =
 \frac{1}{n_1} \sum_{i=1}^{n_1} \mathbf{1}_{\left\{ \|x_1^{(i)} - \hat{\nu}_1\|_2 + \sigma \gamma(n_2, d, \delta) \leq t \right\}} 
    \leq
\frac{1}{n_1} \sum_{i=1}^{n_1} \mathbf{1}_{\left\{\|x_1^{(i)} - \nu_1\|_2 \leq t \right\}} 
=
\widehat{F}_1(t)
\end{align}
In particular, by definition of $\hat{\tau}_r$ \ref{eq:tau}, we have the following\footnote{The equality that $\widehat{E}_1(\hat{\tau}_r) =\frac{n_1'}{n_1}$ requires that the samples are coming from a `continuous' distribution with a Lebesgue density, so that the distances are all pairwise distinct numbers with probability 1. However, since, the inequality $\widehat{F}_1(\hat{\tau}_r - \sigma \gamma(n_2,d,\delta)) 
    \leq 
    \widehat{E}_1(\hat{\tau}_r)$ is true with $\hat{E}_1(\hat{\tau}_r)$ replaced with $\hat{E}_1(\hat{\tau}_r^-)$ because of \ref{eq:F-to-E}. Now, according to the removal strategy of the selective removal algorithm $\hat{E}_1(\hat{\tau}_r^-) \leq \frac{n_1'}{n_1}$ always.} with probability at least $1-\frac{\delta}{2}$
\begin{align} \label{eq:FtoE}
    \widehat{F}_1(\hat{\tau}_r - \sigma \gamma(n_2,d,\delta)) 
    \leq 
    \widehat{E}_1(\hat{\tau}_r)
    =
    \frac{n_1'}{n_1} \leq \widehat{F}_1(\hat{\tau}_r + \sigma \gamma(n_2,d,\delta)) .
\end{align}
Now, observe that $\|x_1^{(i)} - \nu_1\|_2 \overset{d}{=} \|\sigma Z + \mu_1- \nu_1\|_2$\footnote{In the log-concave case \ref{th:randomLCSR} we have  $|x_1^{(i)} - \nu_1| \overset{d}{=} \|X + \mu_1- \nu_1\|$ with $X\sim p_0$.} with $Z \sim N(0, \II_d)$ of location $\mu_1 - \nu_1$ and scale $\sigma^2$, since  $x_1^{(i)}$ follows $p_1 = \cN(\mu_1, \sigma^2\II_d)$ .
Denote by $F_1$ its CDF.
\begin{equation}
    F_1(t):= \mathbb{P}[\|\sigma Z + \mu_1- \nu_1\|_2 \leq t] \text{ with } Z \sim N(0, \II_d). 
\end{equation}
Applying the Dvoretzky–Kiefer–Wolfowitz inequality ~\ref{lem:DKW},\footnote{Observe that the tail estimate $d(n_1, \delta)$ in the Dvoretzky–Kiefer–Wolfowitz inequality is independent of the what the actual CDF $F_1$ is and therefore, the inequality holds   exactly the same in the log-concave case \ref{th:randomLCSR}.} we have with probability at least $1-\tfrac{\delta}{2}$,
\begin{equation}
    \sup_{t\in \R}|\widehat{F}_1(t) - F_1(t)| \leq d(n_1, \delta).
\end{equation}
Plugging it into equation \ref{eq:FtoE} and using a union bound, we get with probability at least $1-\delta$.
\begin{align}
    F_1(\hat{\tau}_r - \sigma \gamma(n_2,d, \delta)) \leq 
    \widehat{F}_1(\hat{\tau}_r - \sigma \gamma(n_2,d, \delta)) + d(n_1, \delta) \leq  \frac{n_1'}{n_1} +  d(n_1, \delta)  .
\end{align}
By taking the inverse $F^{-1}_1$ of  $F_1$\footnote{Since $F_1: \mathbb{R} \to [0,1)$ is the CDF of the norm of a shifted normal distribution $\|\sigma Z + \mu_1- \nu_1\|_2$, it is a strictly increasing function for $\sigma >0$, and therefore admits a functional inverse.}, we obtain with probability at least $1-\delta$ that
\begin{equation} \label{eq:UpperboundSR}
   ||\hat{\mu}_1' - \hat{\nu}_1||_2 \leq  \hat{\tau}_r \leq F_1^{-1}{\left(\frac{n_1'}{n_1} + d(n_1, \delta)\right)} + \sigma \gamma(n_2,d, \delta).
\end{equation}
\end{proof}

\begin{remark}[Lower bound] \label{rem:UBSR}
We observe that the following upper bound can be reversed 
\[
\hat{\tau}_r \leq F_1^{-1}{\left(\frac{n_1'}{n_1} + d(n_1, \delta)\right)} + \sigma \gamma(n_2,d, \delta), \text{ since } 
\frac{n_1'}{n_1} \leq \widehat{F}_1(\hat{\tau}_r + \sigma \gamma(n_2,d,\delta)).
\]
So, applying the Dvoretzky–Kiefer–Wolfowitz inequality  \ref{lem:DKW} we have with probability at least $1-\frac{\delta}{2}$
\[
 \frac{n_1'}{n_1} 
 \leq 
 \widehat{F}_1(\hat{\tau}_r + \sigma \gamma(n_2,d,\delta))
 \leq 
 F_1(\hat{\tau}_r+ \sigma \gamma(n_2,d,\delta))  + d(n_1, \delta)  
\]
\[
\leftrightarrow 
F_1^{-1}\left(\frac{n_1'}{n_1} -d(n_1, \delta)\right) - \sigma \gamma(n_2, d,\delta) \leq \hat{\tau}_r
\]
However, it would be very interesting to be able reverse the inequality $||\hat{\mu}_1' - \hat{\nu}_1||_2 \leq  \hat{\tau}_r$, since that will help up find the a necessary lower bound for $\ve_m(S)$ in Proposition \ref{th:selective}.
\end{remark}

\samplesselection*
\begin{proof}
    We recall that $p_1 = \cN(\mu_1 , \sigma^2 \II_d), q_1 = \cN(\nu_1 , \sigma^2\II_d) \in \cP =\{N(\mu_1, \sigma^2\II_d): \mu_1\in \R^d\}$. We are given $n_1$ IID samples $x_1^{(1)}, \ldots, x_1^{(n_1)}$ from $p_1$ and $n_2$ IID samples $x_2^{(1)}, \ldots, x_2^{(n_2)}$  from $q_1$. The distance-based selection removes $n_r \leq n_1$ selected samples from the unwanted distribution $p_1$ with the $n_r$ largest distances to $\hat{\nu}_1 \coloneqq \tfrac{1}{n_2} \sum_{i=1}^{n_2} x_2^{(i)}$ the empirical estimator of the mean of $q_1$.
    We denote by $x_1^{(1:n_1)}, \ldots, x_1^{(n_1:n_1)}$ the original $n_1$ samples from $p_1$ reordered by increasing distance to $\hat{\nu}_1$:
    \begin{align}
        ||x_1^{(1:n_1)} - \hat{\nu}_1||_2
        \leq  \ldots \leq ||x_1^{(n_1:n_1)} - \hat{\nu}_1||_2,
    \end{align}
    with ties broken arbitrarily, then distance-based selection retains only $x_1^{(1:n_1)}, \ldots, x_1^{(n_1':n_1)}$ to obtain
    \begin{equation}
        \label{eq:selectionSRG}
        \hat{\mu}_1' \coloneqq \tfrac{1}{n_1'} \sum_{i=1}^{n_1'} x_1^{(i:n_1)}.
    \end{equation}
  Subsequently, we set the center $\mu$ of the unlearned distribution $p$ to be the following weighted mean.
    \begin{equation}
        \mu = \frac{n_1' \hat{\mu}_1' + n_2 \hat{\nu}_1}{n_1' + n_2},
    \end{equation}
    where $\hat{\mu}_1 = \tfrac{1}{n_1'} \sum_{i=1}^{n_1'} x_1^{(i:n_1)}$ and $\hat{\nu}_1 = \tfrac{1}{n_2} \sum_{i=1}^{n_2} x_2^{(i)}$ are the empirical mean estimators.
    
    We also observe that the Laurent-Massart bound ~\ref{lem:hoeffding} yields that
    \begin{align}
        \mathbb{P}\left[||\hat{\nu}_1 - \nu_1||_2 \geq \sigma \gamma(n_2,d,2\delta)\right]\leq \delta,
        \label{eq:hoeffding2}
    \end{align}

Now, we are in search of $(\alpha, \ve)$ pairs so that $T(p,p_1) \leq f_d \leftrightarrow ||\mu_1-\mu||_2 \geq \alpha \sigma  $ and $T(p,q_1) \geq f_c \leftrightarrow ||\nu_1-\mu||_2 \leq \ve \sigma $. Remember, if $\alpha >\Delta +\ve$, then the feasible region is itself empty \ref{thm:FFR}.
\paragraph{Preservation bound.}
Using the triangle inequality for norms, we have\footnote{In this step, it would be interesting to provide a lower bound on $\|\hat{\mu}_1' - \hat{\nu}_1\|_2$, so that we can apply the \textit{inverse} triangle inequality $ \|a+b\|_2 \geq \|a\|_2 -\|b\|_2$.}
\begin{align*}
    \|\mu - \nu_1\|_2
    &= \left|\left|\frac{n_1' \hat{\mu}_1' + n_2 \hat{\nu}_1}{n_1' + n_2} - \nu_1 \right|\right| = \left|\left|\frac{n_1'}{n_1' + n_2} (\hat{\mu}_1' - \hat{\nu}_1) +  (\hat{\nu}_1 - \nu_1) \right|\right|\\
    &\leq 
    \frac{n_1'}{n_1' + n_2} \|\hat{\mu}_1' - \hat{\nu}_1\|_2 + \|\hat{\nu}_1 - \nu_1\|_2.
\end{align*}
Therefore, using the Laurent-Massart bound ~\eqref{eq:hoeffding2} we have with probability at least  $1-\tfrac{\delta}{2}$:
\begin{align*}
    ||\mu - \nu_1||_2
    &\leq \frac{n_1'}{n_1' + n_2} ||\hat{\mu}_1' - \hat{\nu}_1||_2+ \sigma \gamma(n_2,d,\delta).
\end{align*}
Moreover, from lemma~\ref{lem:bias_selection}, and more precisely from Equation \ref{eq:UpperboundSR}, we have with probability atleast $1-\delta$ (we have $\{\|\hat{\nu}_1- \nu_1\|_2 \leq \sigma \gamma(n_2, d,\delta)\}$ in common)
\begin{equation*}
    ||\hat{\mu}_1' - \hat{\nu}_1||_2 \leq F_1^{-1}\left(\frac{n_1'}{n_1} +d(n_1, \delta)\right) +  \sigma \gamma(n_2,d,\delta).
\end{equation*}
Combining the above with a union bound, yields that with probability at least $1-\delta$
\begin{align*}
    ||\mu - \nu_1||_2
    &\leq \frac{n_1'}{n_1' + n_2} 
    \left(F_1^{-1}\left(\frac{n_1'}{n_1} +d(n_1, \delta)\right) + \sigma \gamma(n_2,d,\delta)\right)
    + 
   \sigma  \gamma(n_2,d,\delta).
\end{align*}
We can further simplify the above to obtain the simplified expression\footnote{We can also use a tighter bound replacing $2\sigma \gamma(n_2,d,\delta)$ with $\frac{n_1'}{n_1'+n_2} \sigma \gamma(n_2,d,\delta) + \sigma \gamma(n_2,d,\delta) $}.
\begin{equation} \label{eq:tighterbound}
    ||\mu - \nu_1||_2
    \leq \frac{n_1'}{n_1' + n_2} 
    \left(F_1^{-1}\left(\frac{n_1'}{n_1} +d(n_1, \delta)\right)\right)
    + 
    2 \sigma \gamma(n_2,d,\delta).
\end{equation}

\paragraph{Removal bound.}
Using the Laurent Massart bound \ref{eq:hoeffding} with probability atleast $1-\delta$, we have
\begin{align*}
    &
   \sigma \Delta= \|\mu_1 - \nu_1\|_2 
     \leq \|\mu_1 - \mu\|_2 + \|\mu - \nu_1\|_2
    \\
    &\leq \|\mu_1 - \mu\|_2
    + 
    \frac{n_1'}{n_1' + n_2} 
    \left(F_1^{-1}\left(\frac{n_1'}{n_1} +d(n_1, \delta)\right)\right)
    + 
     2\sigma \gamma(n_2,d,\delta).
\end{align*}
Rearranging the terms yields that with probability at least $1-\delta$, we have the following

\begin{equation}
    \|\mu_1 - \mu\|_2\geq  \sigma \Delta - \frac{n_1'}{n_1' + n_2} 
    \left(F_1^{-1}\left(\frac{n_1'}{n_1} +d(n_1, \delta)\right)\right)
    -
     2 \sigma \gamma(n_2,d,\delta).
\end{equation}
\paragraph{Conclusion.}
With probability at least $1-\delta$, $p_{\mu}\in \cP$ satisfy  $(f_d \leftrightarrow \alpha, f_c \leftrightarrow \varepsilon)$ unlearning \ref{def:dist-unlearning} with
\begin{align*}
    \alpha
    &
    \leq 
   \alpha_M(S):= \Delta - 2\gamma(n_2,d,\delta) - \frac{1}{\sigma}\frac{n_1'}{n_1' + n_2} 
    \left(F_1^{-1}\left(\frac{n_1'}{n_1} +d(n_1, \delta)\right)\right)
    \\
    \varepsilon &
    \geq  
    \ve_m(S):= 
    \frac{1}{\sigma}\frac{n_1'}{n_1' + n_2} 
    \left(F_1^{-1}\left(\frac{n_1'}{n_1} +d(n_1, \delta)\right)\right)
    + 
    2\gamma(n_2,d,\delta)
\end{align*}
\end{proof}
\subsection{Selective removal for shifted log-concave family}
Now, we provide a generalization of the finite sample Gaussian result $\{N(\mu, \sigma^2\II_d): \mu\in \R^d\}$ with removal-preservation baselines $f_d= T(N(0,1), N(\alpha, 1)), f_c=T(N(0,1), N(\ve,1))$ \ref{th:random} to a location family $\{p_{\mu} \overset{d}{=} \mu+X: \mu \in \R\}$ in $d=1$ with a symmetric log-concave noise $p_0 \overset{d}{=}X$ and removal-preservation baselines $f_d= T(X, \alpha+X), f_c=T(X,\ve+X)$.
\begin{restatable}[Selective Removal]{proposition}{samplesLCSR}
\label{th:randomLCSR}
Let $p_1\overset{d}{=}\mu_1+X, q_1\overset{d}{=} \nu_1+ X \in \mathcal{P}= \{p_{\mu} \overset{d}{=} \mu+X: \mu \in \R\}$ and $\delta \in (0,1)$, where we assume that $p_0\overset{d}{=}X$ is a symmetric log-concave distribution on $\R$, and consider removal-preservation baselines as $f_d= T(X, \alpha+X), f_c=T(X,\ve+X)$. We observe $n_1$ IID samples from $p_1$ and $n_2$ IID samples from $q_1$, and selectively remove $n_r$ samples from $p_1$ before fitting a weighted mean $\mu$ according to  \ref{algo:SR}\footnote{It is important to mention that for the more general shifted log-concave family $\{\mu+X:\mu\in \R\}$, we do not fit the weighted MLE, since the weighted mean $\mu$ is a weighted average of empirical means $\hat{\mu}_1'$ and $\hat{\nu}_1$, and both the empirical means  $\hat{\mu}_1'$ and $\hat{\nu}_1$ are not MLEs, since empirical mean is not an MLE for log-concave families in general (such as shifted Laplace family), as it is the case for Gaussians. For mathematical tractability, and trying to bring out the structure of a reasonable baseline in the more general log-concave situation, we still consider weighted empirical means and not weighted MLEs, as required in the selective removal algorithm.}. With probability atleast $1 - \delta$, the resulting  distribution $p_{\mu}$ satisfies $(f_d \leftrightarrow \alpha,f_c\leftrightarrow  \varepsilon)$  unlearning \ref{def:dist-unlearning} with $d(n, \delta)= \sqrt{\frac{\log (4/\delta)}{2n}}$ and 
\begin{align*}
\alpha  \leq \alpha_M(S):= \Delta - \ve_m(S),
\varepsilon \geq  \ve_m(S):= \frac{n_1'}{n_1' + n_2} 
   F_X^{-1}\left(\frac{n_1'}{n_1} +d(n_1, \delta)\right) + 2\gamma_X(n_2,1,\delta)
\end{align*}
where $\Delta = |\mu_1- \nu_1|$, $n_1'= n_1- n_r$, $\gamma_X(n, 1, \delta)>0$ is a number dependent on the distribution $p_0$ (symmetric log-concave) such that for $n$ samples $X_1, \cdots, X_n$   IID from $p_0$ we have 
\begin{equation}
    \mathbb{P}\left[|\sum_{i=1}^n X_i| \geq n \gamma_X(n,1, 2\delta)\right] \leq \delta \leftrightarrow  2 \mathbb{P}\left[\sum_{i=1}^n X_i \geq n \gamma_X(n,1, 2\delta)\right] \leq \delta
\end{equation}
Moreover, $F_X^{-1}$ is the functional inverse\footnote{Similar to the Gaussian  $F_1^{-1}$ for any symmetric log-concave noise $p_0$ with a Lebesgue density, $F_X^{-1}$ exists.} of the CDF $ F_X: \R_{\geq 0} \to  [0,1)$ defined as
    \begin{equation}
    F_X(t):= \mathbb{P}[|X + \mu_1- \nu_1|_2 \leq t] \text{ with } X \sim p_0. 
    \end{equation}
\end{restatable}

\begin{proof}
    The proof is the same as that of the finite sample Gaussian result \ref{th:random}, since from \ref{prop:FFRLC}, we are in search of $(\alpha, \ve)$ pairs so that $T(p,p_1) \leq f_d \leftrightarrow |\mu_1-\mu|_2 \geq \alpha  $ and $T(p,q_1) \geq f_c \leftrightarrow |\nu_1-\mu|_2 \leq \ve $. Similarly, if $\alpha >\Delta +\ve$, then the feasible region is itself empty \ref{prop:FFRLC}.
    Now, we recall that 
    \[
    p_1 \overset{d}{=} \mu_1+X, q_1 \overset{d}{=} \nu_1+ X \in \cP= \{p_{\mu} \overset{d}{=}\mu+X: \mu \in \R\}.
    \] 
    From our assumption, we are given $n_1$ IID samples $x_1^{(1)}, \ldots, x_1^{(n_1)}$ from $p_1$ and $n_2$ IID samples $x_2^{(1)}, \ldots, x_2^{(n_2)}$  from $q_1$. Let $\Delta:=|\nu_1-\mu_1|$. The distance-based selection removal algorithm \ref{algo:SR} removes $n_r \leq n_1$ selected samples from the unwanted distribution $p_1$ with the $n_r$ largest distances to $\hat{\nu}_1 \coloneqq \tfrac{1}{n_2} \sum_{i=1}^{n_2} x_2^{(i)}$ the empirical estimator of the mean of $q_1$.
    We denote by $x_1^{(1:n_1)}, \ldots, x_1^{(n_1:n_1)}$ the original $n_1$ samples from $p_1$ reordered by increasing distance to $\hat{\nu}_1$:
    \begin{align}
        ||x_1^{(1:n_1)} - \hat{\nu}_1||_2
        \leq  \ldots \leq ||x_1^{(n_1:n_1)} - \hat{\nu}_1||_2,
    \end{align}
    with ties broken arbitrarily, then distance-based selection retains only $x_1^{(1:n_1)}, \ldots, x_1^{(n_1':n_1)}$ to obtain
    \begin{equation}
        \label{eq:selectionSRLC}
        \hat{\mu}_1' \coloneqq \tfrac{1}{n_1'} \sum_{i=1}^{n_1'} x_1^{(i:n_1)}.
    \end{equation}
  Subsequently, we set the center $\mu$ of the unlearned distribution $p$ to be the following weighted mean.
    \begin{equation}
        \mu = \frac{n_1' \hat{\mu}_1' + n_2 \hat{\nu}_1}{n_1' + n_2},
    \end{equation}
    where $\hat{\mu}_1 = \tfrac{1}{n_1'} \sum_{i=1}^{n_1'} x_1^{(i:n_1)}$ and $\hat{\nu}_1 = \tfrac{1}{n_2} \sum_{i=1}^{n_2} x_2^{(i)}$ are the empirical mean estimators.
    \begin{equation} \label{eq:empmeanLCSR}
        \mu = \frac{n_1' \hat{\mu}_1' + n_2 \hat{\nu}_1}{n_1' + n_2}, \text{ where }
        \hat{\mu}_1' := \tfrac{1}{n_1'} \sum_{i=1}^{n_1'} x_1^{(i)} \text{ and } \hat{\nu}_1 := \tfrac{1}{n_2} \sum_{i=1}^{n_2} x_2^{(i)}.
    \end{equation}

  From the assumption on the quantity $\gamma_X(n,1,\delta)$ analogous to the Laurent-Massart  bound~\ref{lem:hoeffding}, and from the definition of the quantity $F_X$, using  the corresponding Dvoretzky–Kiefer–Wolfowitz inequality \ref{lem:DKW}, we have the following with probability at least $1- \delta$, using the union bound
    \begin{align} \label{eq:hoeffdingLCSR}
        |\hat{\mu}_1' - \nu_1| \leq 
        F_X^{-1}\left(\frac{n_1'}{n_1} +d(n_1, \delta)\right) + 2\gamma_X(n_2, 1,\delta), \quad |\hat{\nu}_1 - \nu_1| \leq  \gamma_X(n_2, 1,\delta)
    \end{align}

    Now, one can repeat the arguments of the \textbf{Preservation}  and the  \textbf{Removal} bounds \ref{th:randomLC} to conclude.
    \end{proof}

\subsection{Random removal versus selective removal for shifted Gaussians}
\label{app:simple-selective}

Now, we analyze when the high-probability upper bounds for the Pareto frontiers of the Selective removal algorithm \ref{algo:SR} perform better than those of the Random removal algorithm \ref{algo:RR}. For simplicity, we focus on the one-dimensional $(d=1)$ shifted Gaussian family.

\samples*
\samplesselection*

\begin{remark}[Comparison of $\ve_m(S)$ and $\ve_m(R)$] \label{rem:CSROGH}
We compare when $\ve_m(S)\leq \ve_m(R)$. This is essentially \footnote{We mention this because we know the high probability behavior of the random removal algorithm (Remark \ref{rem:ICGRR}), and therefore a necessary lower bound for $\ve$ for the random removal to be able to satisfy the preservation bound with high probability. However, we have only been able to obtain a one-sided or sufficient upper bound for the selective removal algorithm to satisfy the preservation bound with high probability. Said differently, it would be interesting to find a tighter lower bound than $\ve_m(S)$ or prove that $\ve_m(S)- o(1)$ is a necessary lower bound for the selective removal algorithm to satisfy the preservation bound with high probability.} the situation, where the selective removal allows a smaller preservation level than the random removal algorithm. For simplicity, we assume $d=1$ and $\sigma=1$. From \eqref{eq:tighterbound} we recall that in $\ve_m(S)$  we could replace $2 \gamma(n_2,d,\delta)$ with a tighter bound $ \gamma(n_2,d,\delta) + \frac{n_1'}{n_1'+n_2} \gamma(n_2,d,\delta)$

\begin{align}
&
\frac{n_1'}{n_1'+n_2}\left(F_1^{-1}\left(\frac{n_1'}{n_1} +d(n_1, \delta)\right)\right)
    + 
 \gamma(n_2,d,\delta) + \frac{n_1'}{n_1'+n_2} \gamma(n_2,d,\delta)
\\
& 
\leq 
\frac{n_1'  \Delta }{n_1' + n_2} + \frac{n_1' \gamma(n_1',d,\delta) + n_2 \gamma(n_2, d,\delta) }{n_1' + n_2} \text{ or equivalently}
\\
&
F_1^{-1}\left(\frac{n_1'}{n_1} +d(n_1, \delta)\right)
\leq \Delta + \gamma(n_1',d,\delta) - 2\gamma(n_2,d,\delta)  \leftrightarrow
\\
&
\left(\frac{n_1'}{n_1} +d(n_1, \delta)\right) \leq F_1(\Delta + \gamma(n_1',d,\delta) - 2\gamma(n_2,d,\delta))
\end{align}
Recall $F_1(t)= \mathbb{P}[\|Z + \mu_1 - \nu_1\|_2 \leq t]$. In $d=1$, (without loss of generality $\Delta= \mu_1-\nu_1>0$) 
\[
F_1(t)
= 
\mathbb{P}[-t\leq Z + \mu_1 - \nu_1 \leq t]
= 
\mathbb{P}[-t -\Delta \leq Z \leq t -\Delta] \text{ for } Z \sim N(0,1).
\]
Now, taking $t= \Delta \gamma(n_1',d,\delta) - 2\gamma(n_2,d,\delta)$ in the above we need to have (throughout $d=1$) 
\[
\frac{n_1'}{n_1} +d(n_1, \delta)
\leq 
\mathbb{P}[-2\Delta - \gamma(n_1',d,\delta) + 2\gamma(n_2,d,\delta) \leq Z \leq  \gamma(n_1',d,\delta) - 2\gamma(n_2,d,\delta)]
\]
\[
\frac{n_1'}{n_1} +d(n_1, \delta)
\leq 
\mathbb{P}[- \gamma(n_1',d,\delta) + 2\gamma(n_2,d,\delta)  \leq Z \leq   2\Delta + \gamma(n_1',d,\delta) - 2\gamma(n_2,d,\delta) ]
\]
\begin{equation} \label{eq:SCSR}
\leftrightarrow q \leq \mathbb{P}[A \leq Z \leq 2\Delta -A] = \Phi(2\Delta -A) - \Phi(A), \text{ where}
\end{equation}
\[
A :=  2 \gamma(n_2,1,\delta)-\gamma(n_1',1,\delta),
q
\coloneqq
\frac{n_1'}{n_1}+d(n_1,\delta),
\text{ and }
\Phi(t)=\mathbb{P}[Z \leq t]
\]

The above is a sufficient condition\footnote{In principle, one could obtain a tighter lower bound for $\ve$ in the selective removal algorithm.} on $\Delta$ on when one can obtain a smaller preservation level $\ve$ with selective removal than random removal. Given this sufficient condition  \eqref{eq:SCSR} a finite \(\Delta\) satisfying
\[
q
\leq
\Phi(2\Delta- A) - \Phi(A)
\text{
exists if and only if }
q < 1-\Phi(A). 
\]
\[
\text{Equivalently }
\Phi(A)
+
\sqrt{\frac{\log(2/\delta)}{2n_1}}
<
\frac{n_r}{n_1}.
\]

Now, recall that $\sqrt{n} \gamma(n, 1, \delta):= \sqrt{1+2\sqrt{1 \log (1/\delta)} + 2 \log (1/\delta)}$. So, $\sqrt{n}\gamma(n,1,\delta) \sim 1$ assuming $\delta$ is fixed or user specified and independent of all the other parameters $n_r,n_1,n_2$, and therefore $\Phi(A) \sim \Phi(0) \sim \frac{1}{2}$. As a consequence, we need to be in a regime where we remove at least half the unwanted samples.
\[
n_r
>
n_1
\Phi(A)
+
\sqrt{n_1\log(2/\delta)}
\sim \frac{n_1}{2}.
\]
Under this condition, the minimal finite value of separation \(\Delta = \|\mu_1- \nu_1\|_2\) required is given by
\[
2\Delta 
\geq 
2\Delta_{m}:
=
\left[
A+
\Phi^{-1}\!\left(\Phi(A)+q\right)
\right].
\]
provided that the following feasibility condition is satisfied.
\[
n_r
>
n_1
\Phi(A)
+
\sqrt{n_1\log(2/\delta)}
\sim \frac{n_1}{2}.
\]
As a consequence of the above, we have proven the following result of a sufficient condition for when the selective removal allows a smaller preservation level than the random removal.
\end{remark}

\begin{restatable}[Comparison of Selective and Random Removal on one dimensional Gaussians]{proposition}{CSRRODG}
\label{th:CSRRODG}
Let $p_1=N(\mu_1,1), q_1= N(\nu_1,1) \in \mathcal{P}= \{N(\mu,1): \mu \in \R\}$ and $\delta \in (0,1)$. We observe $n_1$ IID samples from $p_1$ and $n_2$ IID samples from $q_1$, and remove $n_r$ samples from $p_1$ before fitting a weighted MLE , according to \ref{algo:RR} and \ref{algo:SR}. Recall the notation\footnote{We use the tighter bound of $\ve_m(S)$ obtained in the proof of Proposition \ref{th:selective}.}
\begin{align*}
 \ve_m(R):= \frac{n_1'  \Delta }{n_1' + n_2} + \frac{n_1' \gamma(n_1',1,\delta) + n_2 \gamma(n_2, 1,\delta) }{n_1' + n_2}.
\end{align*}
\begin{align*}
\ve_m(S)
:=\frac{n_1'}{n_1'+n_2}\left(F_1^{-1}\left(\frac{n_1'}{n_1} +d(n_1, \delta)\right)\right)
    + 
 \gamma(n_2,1,\delta) + \frac{n_1'}{n_1'+n_2} \gamma(n_2,1,\delta)
\end{align*}
where $\Delta := |\mu_1- \nu_1|$, $n_1'= n_1- n_r$ and $\sqrt{n} \gamma(n, 1, \delta):= \sqrt{1+2\sqrt{1 \log (1/\delta)} + 2 \log (1/\delta)}$, and $F_1(t)= \mathbb{P}[|Z + \mu_1-\nu_1| \leq t]$.  
Now,  we have $\ve_m(S)\leq \ve_m(R)$ if and only if 
\begin{equation} \label{eq:Delta_m}
    2\Delta 
\geq 
2\Delta_{m}:
=
\left[
A+
\Phi^{-1}\!\left(\Phi(A)+q\right)
\right] \text{ provided }
\end{equation} 
\begin{equation} \label{eq:n_r}
n_r
>
n_1
\Phi(A)
+
\sqrt{n_1\log(2/\delta)}
\sim \frac{n_1}{2}, \text{ where}
\end{equation}
\begin{equation}
    A :=  2 \gamma(n_2,1,\delta)-\gamma(n_1',1,\delta)
\end{equation}
\end{restatable}
\begin{proof}
    The proof is the same as given in Remark \ref{rem:CSROGH} above.
\end{proof}

\begin{remark}[Benefits of Selective removal over Random removal]
    We observe that in a one-dimensional shifted Gaussian family, a sufficient condition for the superior behavior of the Selective removal over the random removal requires us to be in the regime of removing at least more than half the unwanted samples. So, it would be interesting to observe if a tighter bound is possible.
\end{remark}

\newpage

\end{document}